\documentclass[twoside]{article}

\usepackage[accepted]{aistats2025}

\makeatletter
\renewcommand{\@copyrightspace}{}
\makeatother

\usepackage{xcolor}
\usepackage{hyperref}
\usepackage{booktabs}
\usepackage{subcaption}
\usepackage{algorithm}
\usepackage{algorithmic}
\usepackage{mathtools}
\usepackage{amsthm}
\usepackage{amssymb}
\usepackage{dsfont}
\usepackage{graphicx}
\usepackage{makecell}
\usepackage{float}
\usepackage[capitalize,noabbrev]{cleveref}
\crefname{section}{Sec.}{Sec.}
\crefname{figure}{Fig.}{Fig.}
\crefname{table}{Tab.}{Tab.}
\crefname{algorithm}{Alg.}{Alg.}
\crefname{equation}{Eq.}{Eq.}
\crefname{properties}{Properties}{Properties}
\crefname{assumption}{Assumption}{Assumptions}
\creflabelformat{equation}{#2\textup{#1}#3}
\theoremstyle{plain}
\newtheorem{theorem}{Theorem}[section]
\newtheorem{proposition}[theorem]{Proposition}
\newtheorem{lemma}[theorem]{Lemma}
\newtheorem{corollary}[theorem]{Corollary}
\theoremstyle{definition}
\newtheorem{definition}[theorem]{Definition}
\newtheorem{assumption}[theorem]{Assumption}
\newtheorem{properties}[theorem]{Properties}
\newtheorem{example}[theorem]{Example}
\theoremstyle{remark}
\newtheorem{remark}[theorem]{Remark}
\usepackage{enumitem}
\crefname{example}{Example}{Examples}
\AtBeginEnvironment{proposition}{\crefalias{theorem}{proposition}}
\AtBeginEnvironment{lemma}{\crefalias{theorem}{lemma}}
\AtBeginEnvironment{corollary}{\crefalias{theorem}{corollary}}
\AtBeginEnvironment{definition}{\crefalias{theorem}{definition}}
\AtBeginEnvironment{assumption}{\crefalias{theorem}{assumption}}
\AtBeginEnvironment{properties}{\crefalias{theorem}{properties}}
\AtBeginEnvironment{remark}{\crefalias{theorem}{remark}}
\AtBeginEnvironment{example}{\crefalias{theorem}{example}}

\newcommand{\argmin}[1]{\underset{#1}{\arg \min} \;}

\newcommand{\R}{\mathbb{R}}
\newcommand{\N}{\mathbb{N}}
\newcommand{\E}{\mathbb{E}}
\newcommand{\K}{\mathfrak{K}}
\newcommand{\A}{\mathcal{A}}
\renewcommand{\H}{\mathcal{H}}
\newcommand{\KL}{\mathrm{KL}}
\newcommand{\difi}{\: \mathrm{d}}
\newcommand{\dif}{\mathrm{d}}

\renewcommand{\cal}[1]{\mathcal{#1}}
\newcommand{\bb}[1]{\mathbb{#1}}

\newcommand{\eps}{\varepsilon}
\newcommand{\svgd}{\textsc{svgd}\@\xspace}
\newcommand{\adalipo}{\textsc{adalipo}\@\xspace}
\newcommand{\cma}{\textsc{cma-es}\@\xspace}
\newcommand{\woa}{\textsc{woa}\@\xspace}
\newcommand{\bayes}{\textsc{bayesopt}\@\xspace}
\newcommand{\mala}{\textsc{mala}\@\xspace}
\newcommand{\langevin}{\textsc{langevin}\@\xspace}
\newcommand{\cbo}{\textsc{cbo}\@\xspace}

\newcommand{\sbs}{\textsc{sbs}\@\xspace}
\newcommand{\sbspf}{\textsc{sbs-pf}\@\xspace}
\newcommand{\sbshybrid}{\textsc{sbs-hybrid}\@\xspace}
\newcommand{\sbspfhybrid}{\textsc{sbs-pf-hybrid}\@\xspace}

\newcommand{\lima}[1]{\xrightarrow[#1]{}}

\newcommand{\bd}{BD\@\xspace}
\newcommand{\almostall}[1]{
  \ifthenelse{\equal{#1}{}}{\overline{\forall}}{{\overline{\forall}}_{\hspace{-0.2em} #1}}
}
\newcommand{\algrule}[1][.2pt]{\par\vskip.5\baselineskip\hrule height #1\par\vskip.5\baselineskip}
\DeclarePairedDelimiter{\norm}{\lVert}{\rVert}
\DeclarePairedDelimiter{\abs}{|}{|}
\DeclarePairedDelimiter{\pare}{(}{)}
\DeclarePairedDelimiter{\bra}{\{}{\}}
\DeclarePairedDelimiter{\ket}{\langle}{\rangle}

\graphicspath{{figures/}}

\usepackage[normalem]{ulem}

\newcounter{marginNoteCounter}

\newcommand{\algcomment}[1] {\hfill{\color{gray} \scriptsize -~-~#1}}
\usepackage{xspace}
\newcommand{\vs}         {vs.\@\xspace}
\newcommand{\ie}         {i.e.\@\xspace}
\newcommand{\eg}         {e.g.\@\xspace}

\makeatletter
\newcommand\footnoteref[1]{\protected@xdef\@thefnmark{\ref{#1}}\@footnotemark}
\makeatother

\usepackage[labelfont=bf]{caption}
\newcommand{\inlinetitle}[2] {\noindent\textbf{#1{#2}}}

\definecolor{mydarkblue}{rgb}{0,0.08,0.45}
\definecolor{mylightpurple}{rgb}{0.26, 0.32, 0.82}
\definecolor{mylightblue}{rgb}{0.26, 0.38, 0.93}
\hypersetup{
  colorlinks=true,
  linkcolor=red,
  citecolor=mylightpurple,
  filecolor=mylightpurple,
  urlcolor=mylightpurple,
}

\begin{document}

\twocolumn[

\aistatstitle{Stein Boltzmann Sampling: A Variational Approach for Global Optimization}

\aistatsauthor{ Gaëtan Serré \And Argyris Kalogeratos \And Nicolas Vayatis }

\aistatsaddress{ Centre Borelli, École Normale Supérieure Paris-Saclay } ]

\begin{abstract}
We present a deterministic particle-based method for global optimization of continuously differentiable functions, called \emph{Stein Boltzmann Sampling} (\sbs). \sbs initializes uniformly a number of particles representing candidate solutions, then uses the \emph{Stein Variational Gradient Descent} (\svgd) algorithm to sequentially and deterministically move those particles in order to approximate a target distribution whose mass is concentrated around promising areas of the domain of the optimized function. The target is chosen to be a properly parametrized Boltzmann distribution. For the purpose of global optimization, we study \svgd on a compact subset of $\R^d$, without any regularity of its boundary. With a kernel that vanishes on the boundary, we prove that the continuous-time dynamics of \svgd exists, keeps the particles in the domain and decreases the Kullback--Leibler divergence to the target. When the kernel is moreover Stein-universal, this dynamics converges weakly to the target. These results are formalized in the Lean proof assistant. In addition to the main \sbs algorithm, we present two variants, the \sbspf that includes a particle filtering strategy, and the \sbshybrid one that uses \sbs or \sbspf as a continuation after other particle- or distribution-based optimization methods. An extensive comparison with state-of-the-art methods on benchmark functions demonstrates that \sbs and its variants are highly competitive, while the combination of the two variants provides the best trade-off between accuracy and computational cost.
\end{abstract}
\section{Introduction}
We consider the problem of global optimization of an unknown a priori nonconvex, continuously differentiable function, under the concern of making efficient use of the computational budget, (\ie function evaluations at candidate minimizers). Optimizing an unknown function is a typical situation in real applications, \eg hyperparameter calibration or complex system design emerge in several domains (\eg \cite{Pinter1991, Lee2017}). 
For this, sequential methods are usually employed, where at each iteration the algorithm uses information extracted from the previous candidate solutions to propose new ones. Such methods rely on a deterministic or stochastic process to explore the search space, and on a selection process to choose the next candidate solutions given the previous ones.

In this work, we introduce a new sequential and deterministic particle-based method, called \emph{Stein Boltzmann Sampling} (\sbs), for continuously differentiable functions. \sbs uses the \emph{Stein Variational Gradient Descent} (\svgd) \cite{Liu2016} method to sample from a target distribution whose mass is concentrated at areas of the domain where minimizers are possible to be found. We choose as target the Boltzmann distribution (\bd), which by definition converges toward a distribution with a support spanning over all minimizers of the optimized function. The idea of sampling from the \bd for approximating the minimizers of a function is not new (\eg \cite{Azencott1989, Debortoli2021}), yet utilizing \svgd for global optimization is novel, and therefore, part of our contribution concerns the adaptation of the generic \svgd theoretical framework to our objective. \svgd is a generic variational inference method that approximates a target distribution. Specifically, \svgd constructs a flow in the space of probability measures (similarly to a gradient flow evolving in $\R^d$) that moves toward the target distribution. In the discrete case, candidate solutions are represented by particles, and their updates that displace them are affected by attraction and repulsion forces. The \sbs optimization process is illustrated in \cref{fig:sbs-illustration} (some elements will be clarified in \cref{sec:sbs-variants}), where the sequence of updates over the candidate solutions are shown as trajectories of particles aiming to reach the global minimum. The pseudocode of the proposed \sbs method can be found in \cref{alg:sbs}.

The related global optimization literature is rich of methods. 
\adalipo \cite{Malherbe2017} is a method that is consistent over Lipschitz functions and is adapted for a very low computational budget%
. The well-known \bayes method \cite{Martinezcantin2014} is also adapted for low budgets. Then, there are approaches that use \mala to sample from the Boltzmann distribution, in a similar way to our method \cite{Grenander1994, Welling2011, Raginsky2017, Erdogdu2018}. \cma \cite{Hansen1996} and \woa \cite{Mirjalili2016}, are two inconsistent methods, but are known to be very efficient in practice. Due to either early stopping conditions or time complexity, these two methods do not scale well computationally, hence they are not suited for when the available budget is large and several function evaluations need to be performed. The recent method in \cite{Rudi2024} subsamples a finite subset of constraints from an uncountable one and uses an SDP solver to approximate the global minimum.

The rest of the contribution of the paper is as follows. We study \svgd on a compact domain $\Omega$, as required by global optimization. The theory of \svgd is usually developed on $\R^d$, for smooth target distributions (\eg \cite{Lu2019, Korba2020, Duncan2023}). On a compact domain, nothing keeps the particles in $\Omega$. Moreover, the boundary of $\Omega$ may have corners or cusps, where the integration by parts of the classical proofs is not available. We show that a kernel that vanishes on $\partial \Omega$ overcomes these two difficulties, without any regularity of $\partial \Omega$. We prove that the continuous-time dynamics of \svgd exists, keeps the particles in $\Omega$ and decreases the Kullback--Leibler divergence to the target. We also prove that it converges weakly to the target when the kernel is Stein-universal (see \cref{sec:sbs-theory}). All these results are formalized in the Lean proof assistant \cite{Moura2021, Mathlib2020}. Then, we introduce two \sbs variants: one that uses particle filtering to reduce the budget needed (see \cref{fig:sbs-pf}), and a hybrid one that uses \sbs as a continuation of \cma or \woa, to combine their efficiency with the consistency and scalability of our method (see \cref{sec:sbs-variants}). We discuss the optimal values for the hyperparameters of \sbs and compare our approaches with six state-of-the-art methods on several standard global optimization benchmark functions (see \cref{sec:hyperparameters,sec:benchmark}). Finally, we interpret, in the global optimization context, the internal attraction and repulsion forces between particles, which come in effect during the \svgd sampling (see \cref{sec:discussions}).

\inlinetitle{Notations}{.}~%
Let $d \in \N$ be the dimension of the optimization problem. Let $\lambda$ be the Lebesgue measure on the Borel sets of $\R^d$. Let $\Omega \subset \R^d$ be a compact set such that $\lambda(\Omega) > 0$. No regularity of its boundary $\partial \Omega$ is assumed. Let $f : \R^d \to \R$ be the function to minimize on $\Omega$. Let $X^*$ be the set of its global minimizers on $\Omega$, and let $f^* = f(x^*)$ for $x^* \in X^*$. Let $\cal{P}(\Omega)$ be the set of probability measures on $\R^d$ carried by $\Omega$, \ie such that $\rho(\R^d \setminus \Omega) = 0$. Given a measurable map $T$ and a measure $\rho$, let $T_\# \rho$ denote the pushforward measure, \ie%
$$
\forall B, \; T_\# \rho(B) = \rho(T^{-1}(B)).
$$%
The full list of notations is provided in the appendix (see \cref{tab:notations}).

\begin{algorithm}[t]
  \footnotesize
  \caption{\footnotesize Stein Boltzmann Sampling (\sbs)}
  \label{alg:sbs}
  \begin{algorithmic}
    \STATE {\bfseries Input:} $f : \Omega \to \R$; number of vectors (particles) $N$; Boltzmann parameter $\kappa$; step-size $\eps$; number of \svgd iterations $n$; an initial distribution $\rho_0$ over the particles
    \STATE {\bfseries Output:} $\hat{x}$, an estimate of $x^*$
    \algrule
    \STATE Sample $N$ particles: $X_1 \leftarrow \big(x^{(1)}, ..., x^{(N)} \big)\!\sim\!\rho_0^{\otimes N}$\!\!\!
    \FOR{$i=1$ {\bfseries to} $n$}
    \STATE Compute the vector field $(\phi^*_{\hat{\rho}_i})_\Phi$ \algcomment{see \cref{sec:sbs-method}}
    \STATE $X_{i+1} \leftarrow X_i + \eps (\phi^*_{\hat{\rho}_i})_\Phi(X_i)$  \algcomment{update the particles}
    \STATE $\hat{\rho}_{i+1} \leftarrow \frac{1}{N} \sum_{j=1}^N \delta_{X_{i + 1}^{(j)}}$ \algcomment{empirical measure}
    \ENDFOR
    \STATE $\hat{x} \leftarrow \argmin{}_{1 \leq j \leq N} f\left(X_{n+1}^{(j)}\right)$ \algcomment{the "best" particle}
    \STATE {\bfseries return} $\hat{x}$
  \end{algorithmic}
\end{algorithm}

\begin{algorithm}[t]
  \footnotesize
  \caption{\footnotesize Initialization choice of \sbshybrid}
  \label{alg:sbs-hybrid-init}
  \begin{algorithmic}
    \STATE {\bfseries Input:} number of vectors (particles) $N$;\\ {\sc cma-es} budget $b$%
    \STATE {\bfseries Output:} $N$ candidates
    \algrule
    \STATE Run {\sc cma-es} for $b$ function evaluations
    \STATE Run {\sc woa} with $N$ candidates
    \IF {{\sc cma-es} found a better value than {\sc woa}}
    \STATE Sample $N$ candidates from the last Gaussian
    \ELSE
    \STATE Use the $N$ candidates from {\sc woa}
    \ENDIF
    \STATE {\bf return} the $N$ candidates
  \end{algorithmic}
\end{algorithm}

\begin{figure*}[t]
  \centering
  \begin{subfigure}{0.26\textwidth}
    \centering
    \includegraphics[width=\textwidth]{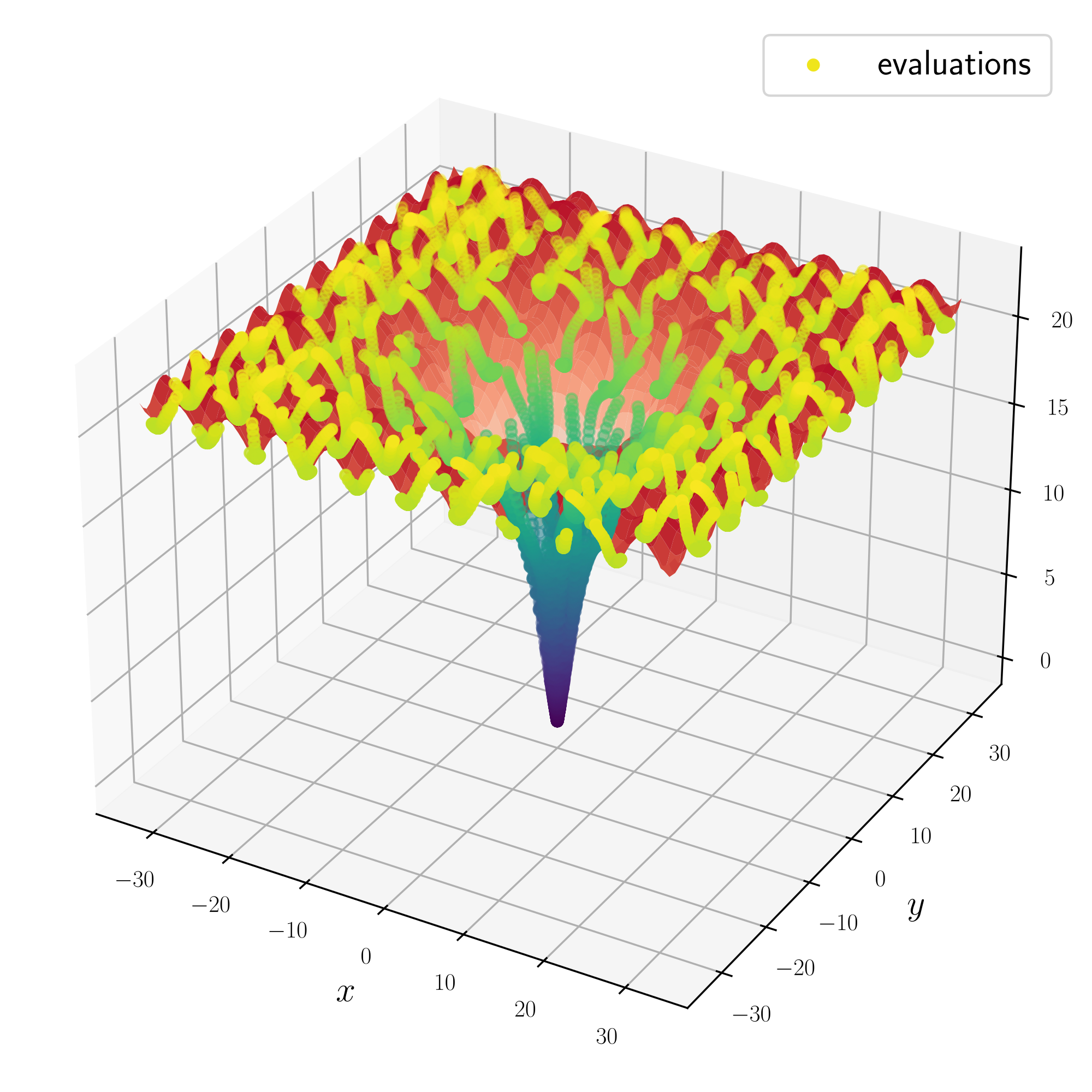}
    \caption{\sbs}
  \end{subfigure}
  \hspace{7em}
  \begin{subfigure}{0.26\textwidth}
    \centering
    \includegraphics[width=\textwidth]{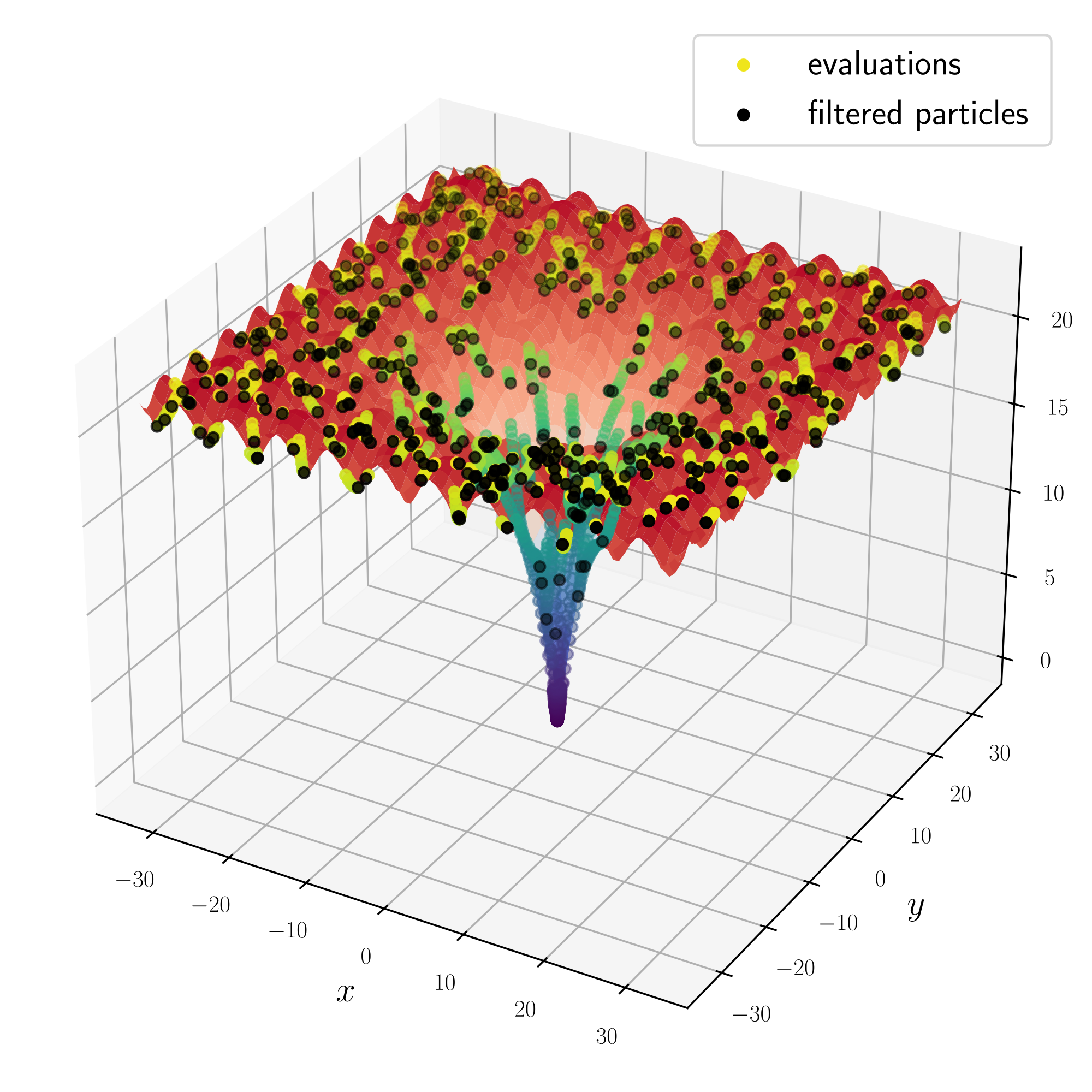}
    \caption{\sbspf}
    \label{fig:sbs-pf}
  \end{subfigure}
		\vspace{-0.2em}
  \caption{Illustration of the flow of measures and the trajectories of particles over the iterations. The color gradient represents the $2$d Ackley function value, from {\color[RGB]{58, 76, 192} blue} (low) to {\color[RGB]{179, 3, 38} red} (high). The trajectories draw the discretized flow of measures. \textbf{a)}~\sbs: the particles are initialized uniformly at random over the domain, and then get updated by making a small step in the direction induced by \svgd forces. \textbf{b)}~\sbspf variant with particle filtering: the particles are initialized and updated as before, but the less promising ones get rapidly removed and are not replaced. This is visible as there are less persisting trajectories in areas where the function has high values. This strategy results in a significant reduction of the budget while having comparable performance.
  }
  \label{fig:sbs-illustration}
\end{figure*}

\section{Stein Boltzmann Sampling}\label{sec:sbs-method}
Let us now introduce the proposed \emph{Stein Boltzmann Sampling} (\sbs) method.
While \emph{Stein Variational Gradient Descent} (\svgd) has been thoroughly studied in the literature \cite{Liu2017, Lu2019, Korba2020, Duncan2023, Sun2023}, this work is the first to consider it in a global optimization context. Its analysis on a compact domain is the subject of \cref{sec:sbs-theory}, and its precise statements and proofs are in \cref{sec:svgd-compact}.

\svgd \cite{Liu2016} iteratively transforms an arbitrary measure $\rho$ into a target distribution $\pi$, using only the score of $\pi$. In the case of \sbs, $\pi$ is the \bd of $f$ (\cref{def:boltzmann-distribution}) for some $\kappa > 0$, whose score is $s = -\kappa \nabla f$. Its normalizing constant is not needed.

\inlinetitle{Kernel and Stein witness}{.}~%
Let $k$ be a kernel on $\R^d$, given by a feature map $\Phi : \R^d \to H_0$, \ie $k(x, y) = \ket{\Phi(x), \Phi(y)}_{H_0}$. Let $\H = H_0^d$. An element $g \in \H$ defines the vector field $g_\Phi(y) = (\ket{g_i, \Phi(y)}_{H_0})_{1 \leq i \leq d}$. The Stein operator of $\pi$ maps a vector field $v$ to $\A_s v = s^\top v + \nabla \cdot v$. By \cref{thm:stein-witness}, there exists a unique $\phi^*_\rho \in \H$, the \emph{Stein witness} of $\rho$, such that
$$
\forall g \in \H, \quad \E_{x \sim \rho}[\A_s g_\Phi(x)] = \ket{g, \phi^*_\rho}_\H.
$$
Up to normalization, it is the direction of $\H$ along which $\rho$ violates Stein's identity the most. It is also the direction of steepest descent of the Kullback--Leibler divergence to $\pi$ (\cref{thm:kl-descent}). Thus, \svgd updates $\rho$ in this direction, with a small step size $\eps > 0$, \ie
\begin{equation}\label{eq:svgd-update}
  \rho_{n + 1} = (I_d + \eps (\phi^*_{\rho_n})_\Phi)_\# \rho_n.
\end{equation}

\inlinetitle{Particles}{.}~%
In practice, $\rho_0$ is the empirical measure of $N$ particles. Let $\hat{\rho}_n = \frac{1}{N} \sum_{j = 1}^N \delta_{X^j_n}$. Then, \cref{eq:svgd-update} moves each particle as follows
\begin{equation}\label{eq:svgd-diff-eq}
  X^i_{n + 1} = X^i_n + \eps (\phi^*_{\hat{\rho}_n})_\Phi(X^i_n),
\end{equation}
where $(X^i_0)_{1 \leq i \leq N} \sim \rho_0^{\otimes N}$ and
\begin{align*}
  (\phi^*_{\hat{\rho}_n})_\Phi(x) = \frac{1}{N} \sum_{j = 1}^N \big(&k(X^j_n, x) s(X^j_n) \\
  &+ \nabla_{X^j_n} k(X^j_n, x)\big).
\end{align*}
A usual choice for $k$ is the Gaussian kernel $k_\sigma$ with bandwidth $\sigma$. An illustration of this equation is given in \cref{fig:vector-field}. The forces driving the particles are determined by a mixture of individual and collective information (see \cref{sec:discussions}).

\inlinetitle{Continuous time}{.}~%
A deep analysis of particle-based models for a large number of particles is exceedingly complex. A popular workaround is to study the evolution of the distribution of the particles \cite{Liu2017,Korba2020,Pinnau2017,Carrillo2018,Fornasier2021}. When the number of particles tends to infinity and the step size $\eps$ tends to $0$, the distribution $\rho_t$ of the particles follows a McKean-Vlasov equation with a zero diffusion \cite{Liu2017, Lu2019}. Let $T_t(x)$ be the position at time $t$ of the particle that starts from $x$. Then, $T_0 = I_d$ and
\begin{equation}\label{eq:svgd-diff-vlasov}
  \frac{\dif}{\dif t} T_t(x) = (\phi^*_{\rho_t})_\Phi(T_t(x)), \quad \text{where } \rho_t = (T_t)_\# \rho_0.
\end{equation}
We call it the continuous-time \svgd equation. The theory of \sbs studies this dynamics (see \cref{sec:sbs-theory}).

To use \svgd as a global optimization method, we need a target distribution that concentrates its mass around the global minimizers of the optimized function, and the continuous Boltzmann distribution (\bd) has this feature. Moreover, it is a classical object in the global optimization theory, and makes a link between our method Simulated Annealing \cite{Kirkpatrick1983} (see \cref{sec:discussions}).
\begin{definition}[Continuous Boltzmann distribution]\label{def:boltzmann-distribution}
Given a function $f \in C^0(\Omega, \R)$, the Boltzmann distribution over $\Omega$ is induced by the probability density function $m_{f, \Omega}^{(\kappa)} : \Omega \to \R_{\geq 0}$ defined by:%
  \begin{equation}\label{eq:BD}
  m_{f, \Omega}^{(\kappa)}(x) = m^{(\kappa)}(x) = \frac{e^{-\kappa f(x)}} { \int_\Omega e^{-\kappa f(t)} \difi t}, \ \ \forall \kappa \in \R_{\geq 0}.%
  \end{equation}
\end{definition}%
A characteristic property of the \bd is that, as $\kappa$ tends to infinity, 
the \bd tends to a distribution supported only over the set of minimizers $X^*$. If $\lambda(X^*) > 0$, the \bd tends to a uniform distribution over $X^*$ (see \cref{fig:boltzmann-distribution}). If $\lambda(X^*) = 0$, it tends to a distribution over $X^*$, where the concentration of the mass depends on the local geometry of the minimizing manifold \cite{Hwang1980} (see details in \cref{sec:boltzmann}).

The proposed \sbs method is essentially an \svgd sampler applied to the \bd. Note that the theoretical results of \cref{sec:sbs-theory} hold for any target distribution that satisfies \cref{ass:target}.

\begin{figure*}[t]
  \centering
  \begin{subfigure}{0.35\textwidth}
    \centering
    \includegraphics[width=\textwidth]{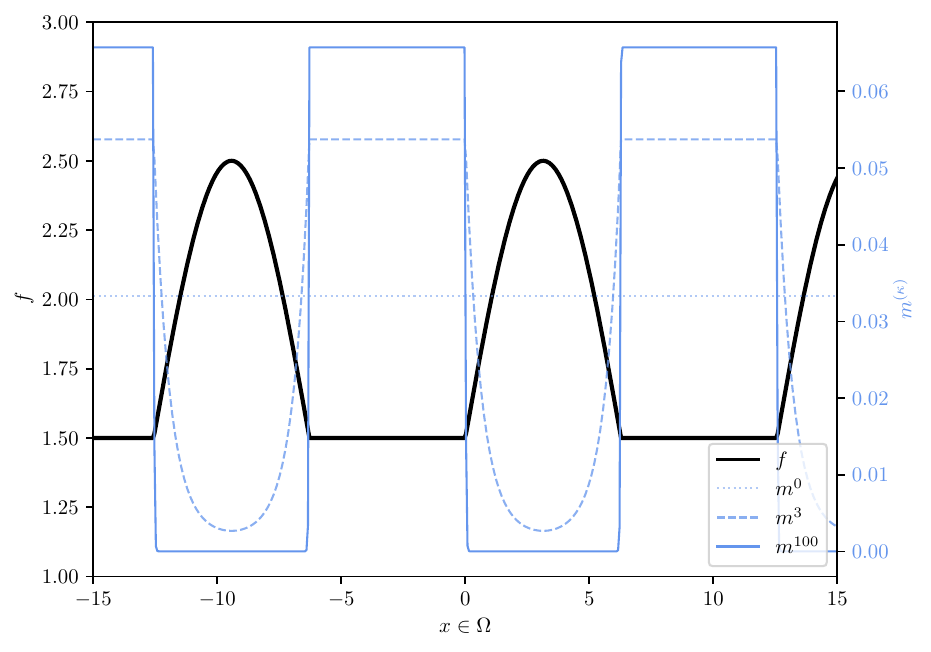}%
    \caption{}
    \label{fig:boltzmann-distribution}
  \end{subfigure}
  \begin{subfigure}{0.35\textwidth}
    \centering
    \includegraphics[width=\textwidth]{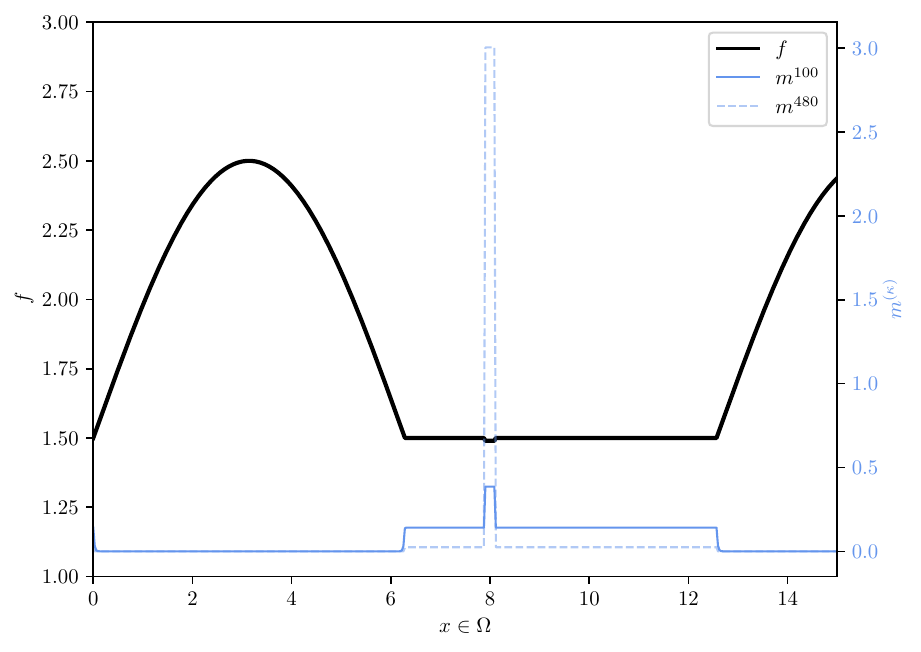}
    \caption{}
    \label{fig:boltzmann-distribution-case}
  \end{subfigure}
	\vspace{-0.2em}
    \caption{\textbf{a)}~The density of the Boltzmann distribution $m^{(\kappa)}$ (\cref{def:boltzmann-distribution}) ({\bf \color[HTML]{6495ed} blue} lines) becomes uniform over the set of minimizers $X^*$ of the given function $f$ to optimize ({\bf black} lines), as its parameter $\kappa$ tends to infinity. \textbf{b)}~In this example, the volume of the set $X^*$ is much smaller than the volume of local minimizers in the flat region. The value of the function at the local minimizers is also closer to the value of the global ones. Setting $\kappa$ to $100$ does not suffice to concentrate the majority of the mass of $m^{(\kappa)}$ around the global minimizers.}
    \label{fig:ex-choice-kappa}
\end{figure*}

\section{Theory of SBS}\label{sec:sbs-theory}
The theoretical analysis of \sbs studies the continuous-time \svgd equation \cref{eq:svgd-diff-vlasov} on the compact domain $\Omega$. Two difficulties arise. First, nothing keeps the particles in $\Omega$, and the theory on $\R^d$ does not apply. Second, the boundary of $\Omega$ may have corners or cusps, where it has no outward normal. The integration by parts on $\Omega$, on which the classical proofs rely, is then not available. In this section, we show that a suitable kernel overcomes these two difficulties, without any regularity of $\partial \Omega$. The precise statements and their proofs are in \cref{sec:svgd-compact}.

\inlinetitle{Assumptions}{.}~%
We work under two assumptions. The target distribution $\pi$ has a positive and continuous density $p$ on $\Omega$, and a continuous score $s$, which plays the role of $\nabla \log p$ (\cref{ass:target}). The kernel is given by a feature map $\Phi : \R^d \to H_0$, which is continuously differentiable, Lipschitz, bounded, and zero on $\partial \Omega$ (\cref{ass:kernel}). For instance, a weighted kernel $w(x) w(y) k_0(x, y)$, where $w$ vanishes on $\partial \Omega$, satisfies the last condition. We prove the following three results.

\inlinetitle{Existence}{.}~%
For every $\rho_0 \in \cal{P}(\Omega)$, the continuous-time \svgd equation has a solution $(T_t)_{t \geq 0}$, defined for all times, and the particles never leave $\Omega$ (\cref{thm:continuous-time,lem:solution-properties}).

\inlinetitle{Decrease of the Kullback--Leibler divergence}{.}~%
If $\KL(\rho_0 \,\|\, \pi) < \infty$, then, along every solution,
$$
\forall t \geq 0, \quad \frac{\dif}{\dif t} \KL(\rho_t \,\|\, \pi) = -\K(\rho_t | \pi) \leq 0,
$$
where $\K$ is the kernelized Stein discrepancy (\cref{thm:kl-derivative}).

\inlinetitle{Convergence}{.}~%
If, moreover, the kernel is Stein-universal (\cref{def:stein-universal}), then $\rho_t$ converges weakly to $\pi$ as $t$ tends to infinity (\cref{thm:weak-convergence}).

No regularity of $\partial \Omega$ is assumed, and the measures need no density. All these results are formalized in Lean. The formalization\footnote{\url{https://github.com/gaetanserre/CompactSVGD}} and a website that presents it\footnote{\url{https://gaetanserre.fr/CompactSVGD/}} are available online.
\begin{figure*}[t]
  \centering
  \begin{subfigure}{0.3\textwidth}
    \includegraphics[width=0.8\textwidth]{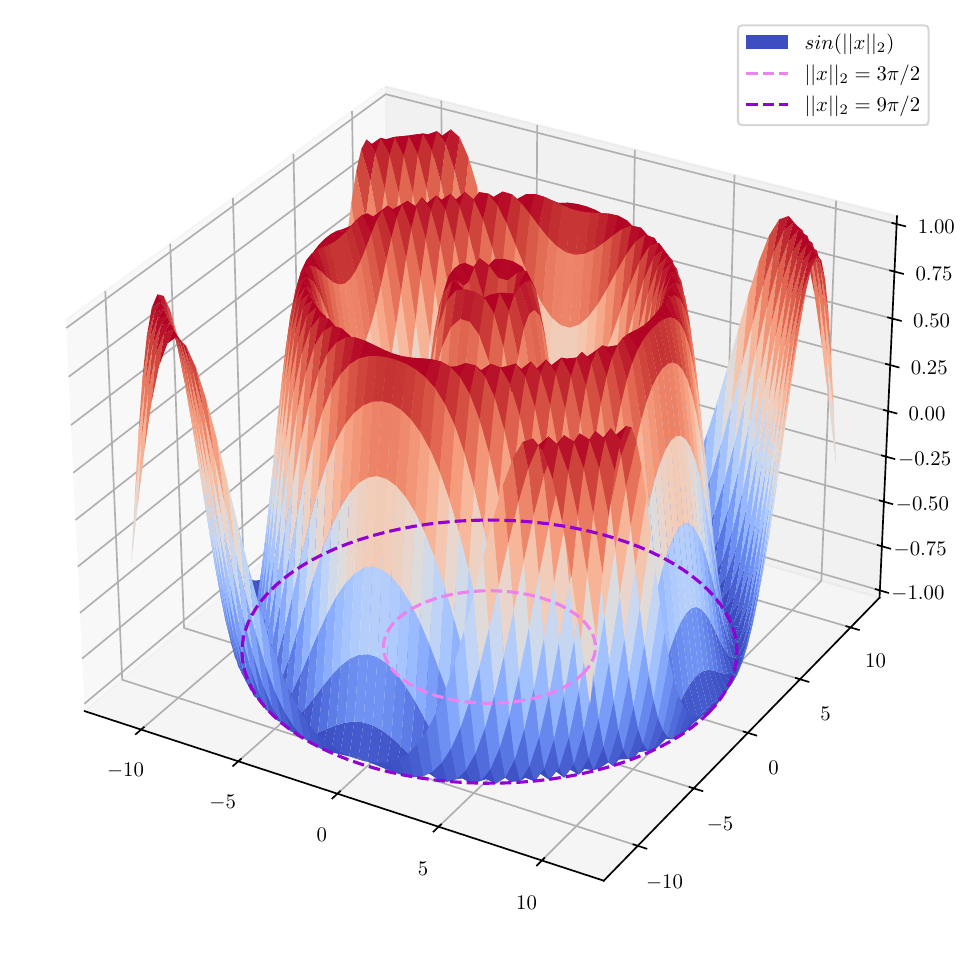}
    \vspace{-0.8em}
    \caption{$x \mapsto \sin \norm{x}_2$}
  \end{subfigure}
  \hspace{1em}
  \begin{subfigure}{0.3\textwidth}
    \includegraphics[width=0.8\textwidth]{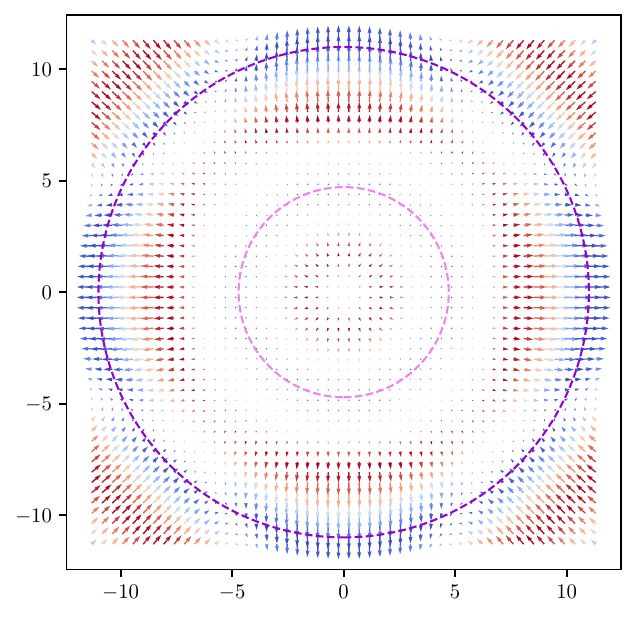}
    \vspace{-0.8em}
    \caption{$t = 0$}
  \end{subfigure}
  \hspace{1em}
  \begin{subfigure}{0.3\textwidth}
    \includegraphics[width=0.8\textwidth]{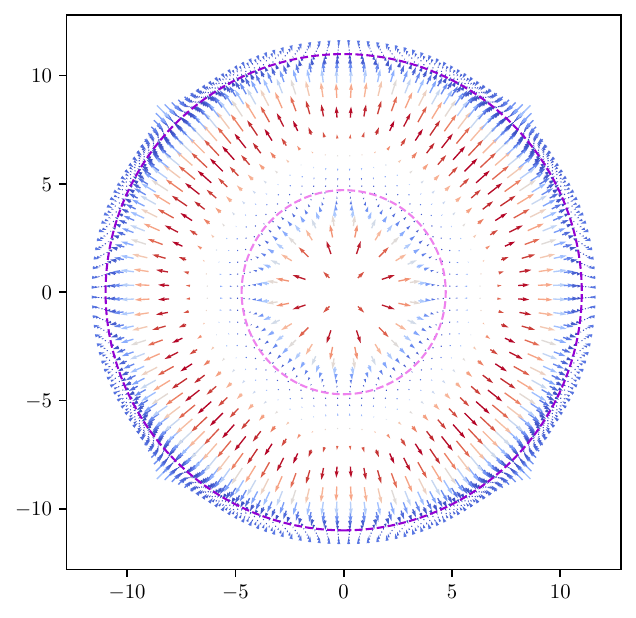}
    \vspace{-0.8em}
    \caption{$t = 700$}
  \end{subfigure}
  \vspace{-0.5em}
  \caption{Illustration of the vector field induced by \cref{eq:svgd-diff-eq} in a discrete-time setting where $\pi$ is the \bd. \textbf{a)}~The optimized function $x \mapsto \sin \norm{x}_2$ and the two manifolds at which it is minimized (dashed {\bf \color[HTML]{808080} gray} lines). \textbf{b)}~The initial particles (not shown) start getting attracted toward the two ring-shaped manifolds. \textbf{c)} After some \svgd iterations, there are stronger forces in the vector field and the particles get concentrated around those minimizing regions.}
  \label{fig:vector-field}
\end{figure*}

\subsection*{Application to SBS}\label{sec:asymptotic-convergence}
The \bd of $f$ satisfies \cref{ass:target}, with $s = -\kappa \nabla f$ (\cref{ex:boltzmann}). Thus, the previous results apply to \sbs with any kernel that satisfies \cref{ass:kernel}, as follows.

\begin{corollary}[Continuous-time \sbs]\label{cor:sbs}
  Let $f : \R^d \to \R$ be continuously differentiable, let $\kappa > 0$, and assume that $\lambda(\Omega) > 0$. Let $\pi$ be the Boltzmann distribution of $f$ on $\Omega$, and let $k$ be a kernel whose feature map $\Phi$ satisfies \cref{ass:kernel}. Let $\rho_0 \in \cal{P}(\Omega)$. Then, the following holds.
  \begin{enumerate}[label=(\roman*)]
    \item The continuous-time dynamics \cref{eq:svgd-diff-vlasov} from $\rho_0$ has a solution $(T_t)_{t \geq 0}$, and $T_t(\Omega) = \Omega$ for every $t \geq 0$.
    \item If $\KL(\rho_0 \,\|\, \pi) < \infty$, then, for every such solution, the function $t \mapsto \KL(\rho_t \,\|\, \pi)$, where $\rho_t = (T_t)_\# \rho_0$, is nonincreasing, with derivative $-\K(\rho_t | \pi)$.
    \item If, moreover, $k$ is Stein-universal, then $\rho_t \xrightharpoonup[t \to \infty]{} \pi$.
  \end{enumerate}
  In particular, (ii) and (iii) hold for the uniform distribution $\rho_0 = \lambda|_\Omega / \lambda(\Omega)$.
\end{corollary}
The proof is in \cref{sec:svgd-sbs}. Thus, the continuous-time dynamics of \sbs exists and keeps the particles in $\Omega$. From any initial distribution with a finite Kullback--Leibler divergence to the \bd, such as the uniform distribution, it decreases this divergence. If a kernel that satisfies \cref{ass:kernel} is Stein-universal, this dynamics converges weakly to the \bd. As $\kappa$ tends to infinity, this distribution concentrates on the minimizers of $f$ (see \cref{sec:sbs-method}). Thus, for a large $\kappa$, the distribution of the particles approaches a distribution concentrated around the global minimizers of $f$.

\inlinetitle{Limitations}{.}~%
The Gaussian kernel $k_\sigma$, which we use in the experiments below, does not satisfy \cref{ass:kernel}, as it does not vanish on $\partial \Omega$. The results above thus apply to \sbs with another kernel, such as $w(x) w(y) k_\sigma(x, y)$, where $w$ vanishes on $\partial \Omega$ (see \cref{sec:svgd-sbs}). The construction of a Stein-universal kernel that satisfies \cref{ass:kernel}, and new experiments with this kernel, are left for future work. Moreover, these guarantees concern the continuous-time dynamics, and not the iteration with finitely many particles that \sbs implements.

\inlinetitle{Implementation}{.}~%
The implementation of \sbs estimates the gradients using finite differences. At each iteration, it updates the set of particles by a small step in the direction induced by $(\phi^*_{\hat{\rho}_i})_\Phi$, which is computed using the Adam optimizer \cite{Kingma2015} that gives better experimental results. We choose the initial distribution $\rho_0$ to be the uniform distribution on $\Omega$, as it maximizes the entropy (\ie high initial exploration), and we use the Gaussian kernel. These two objects are used in most of the \svgd literature.

\begin{figure*}[bt]
  \centering
  \begin{subfigure}{0.27\textwidth}
    \includegraphics[width=\textwidth]{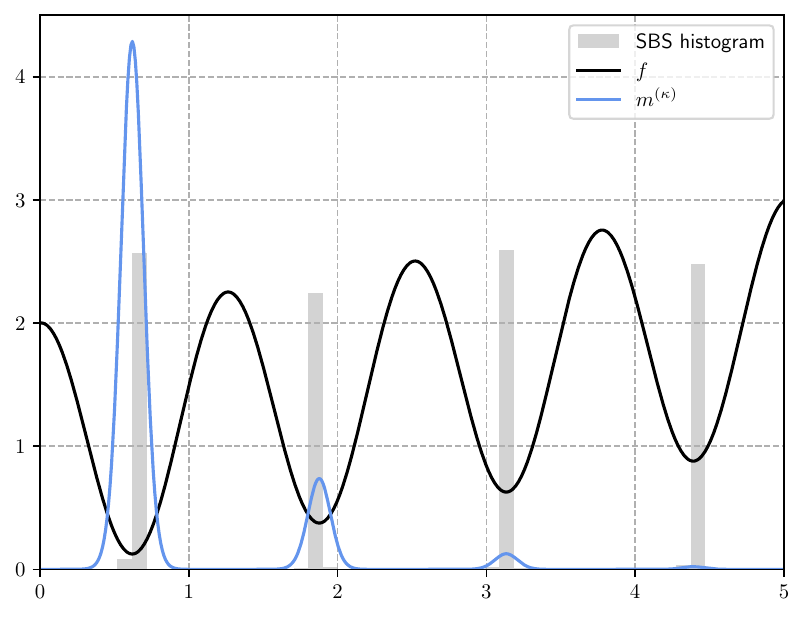}
  \end{subfigure}
  \hspace{1em}
  \begin{subfigure}{0.27\textwidth}
    \includegraphics[width=\textwidth]{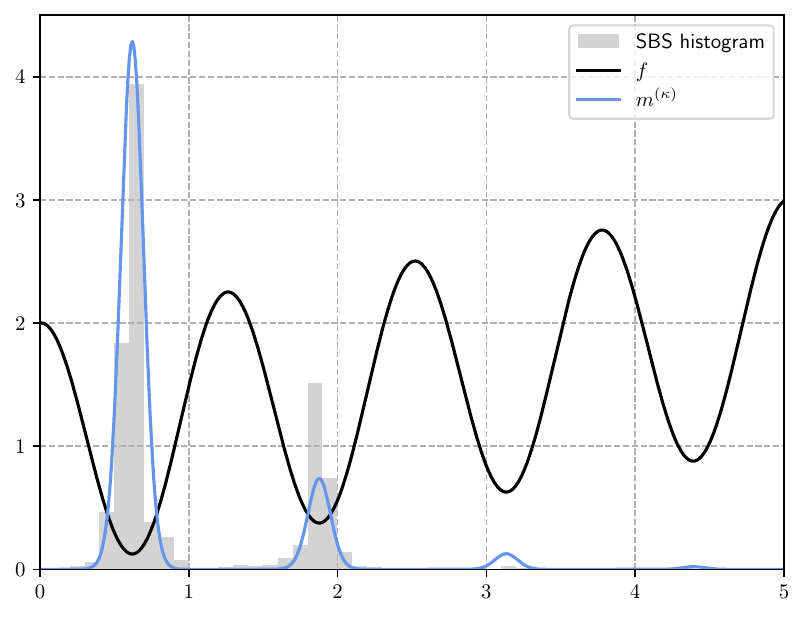}
  \end{subfigure}
  \hspace{1em}
  \begin{subfigure}{0.27\textwidth}
    \includegraphics[width=\textwidth]{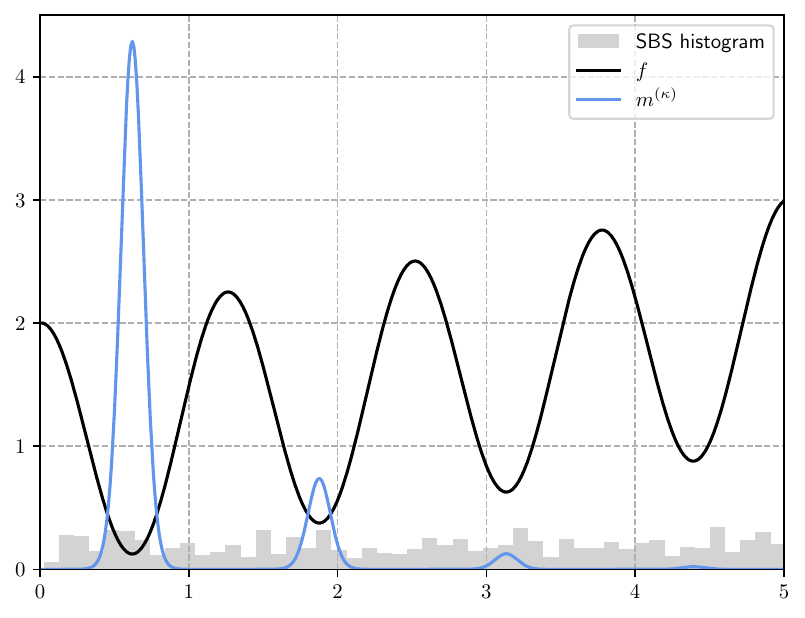}
  \end{subfigure}
  \vspace{-0.5em}
  \caption{\footnotesize Illustration of the exploration/exploitation trade-off in \sbs with different values of $\sigma$. In {\bf black}, the function $x \mapsto \mathrm{cos}(5x) + x/5 + 1$; in {\bf \color[HTML]{808080} grey}, the distribution of the particles; in {\bf \color[HTML]{6495ed} blue}, the BD $m^{(\kappa)}$. When $\sigma$ is too small, the particles are uniformly distributed over $X^*$. When $\sigma$ is too large, they are uniformly distributed over the whole domain $\Omega$.}
    \label{fig:sbs-boltzmann}
\end{figure*}
\section{SBS variants}\label{sec:sbs-variants}
In addition to the main \sbs method, we introduce two variants that can be more efficient in practice. The first one, called \sbspf, uses a particle filtering approach that removes the less promising particles (without replacing them). The second one, called \sbshybrid, is a hybrid method that uses \sbs as a continuation for other global optimization methods, or --seen the other way around-- those methods are used to initialize \sbs. 
\sbspf uses less budget than \sbs, and  
\sbshybrid uses some of the budget to run one of the pre-existing methods to initialize \sbs with better starting points; the aim is to approximate the global minimum better than \sbs with the same budget.%

\inlinetitle {SBS-PF}{.}~%
We use a simple particle filtering idea: to remove particles (\ie candidate minimizers of $f$) that are less promising or stuck in bad local minima. We choose to remove particles that do not move and correspond to significantly higher function values than the others, hence particles that are very likely stuck in bad local minima. This is done by removing particles using their function values and the distance between their previous and actual positions. More precisely, if these two quantities are respectively higher than the $q$-th and lower than the $p$-th percentiles of the function values and "previous-to-actual" distances of the particles, then the particle is removed. The difference between \sbs and this variant is visualized in \cref{fig:sbs-illustration}. One can see that, in \sbspf, the least promising candidates are rapidly removed without being replaced, so that the remaining particles are more likely to converge to the global minimum. This strategy results in a significant reduction of the budget used, while having comparable optimization results to \sbs. Note that \cref{cor:sbs} does not apply to \sbspf, which removes particles. Thus, \sbspf has no theoretical guarantee. However, the empirical results show that \sbspf is efficient in practice, and it is a good alternative to \sbs when the budget is limited.

\inlinetitle{SBS-HYBRID}{.}
This variant is based on the idea of using \sbs or \sbspf as a continuation for particles- or distribution-based methods, such as {\sc woa} or {\sc cma-es}. Indeed, the design of \sbs allows to initialize the particles with the result of one such method, and then resume the optimization process with an \sbs variant. More specifically, we introduce \sbshybrid that runs few iterations of both {\sc cma-es} and {\sc woa} to choose the most promising result, and then continues the optimization with \sbs (see \cref{alg:sbs-hybrid-init}). Both {\sc woa} and {\sc cma-es} are efficient methods, thus, a small number of iterations allows to find a good starting region for \sbs. Moreover, both methods are not well-suited for a large budget, but for different reasons: {\sc cma-es} uses early stopping rules (\eg the condition number of the covariance matrix), and {\sc woa} takes more time to run than \sbs for the same budget. \sbshybrid can be seen as a combination of the \sbs method, on top of very efficient but non-consistent methods. Among the strengths of \sbshybrid, we can mention that i)~it empirically provides high-quality results, and ii)~it keeps the guarantees of \cref{cor:sbs}, provided that the initial distribution given by {\sc woa} and {\sc cma-es} has a finite Kullback--Leibler divergence to the \bd.
\section{Choice of hyperparameters}\label{sec:hyperparameters}
In this section, we discuss the choice of the hyperparameters of \sbs and its variants. We focus on the choice of $\kappa$ and $\sigma$, which carry complex information about the behavior of the method.

\inlinetitle{Choice of $\kappa$}{.}~%
As detailed earlier, $\kappa$ controls the shape of the \bd from which \svgd samples. The bigger $\kappa$ is, the more the mass of the distribution gets concentrated around the global minimizers of the function. Intuitively, the optimal $\kappa$ such that a satisfying amount of the mass is around the global minimizers depends on the geometry of the function around local minima (the asymptotic behavior of the \bd depends on the local geometry, see \cite{Hwang1980}). Nevertheless, one can see in \cref{fig:kappa-comparison} that, in practice, the choice of $\kappa$ does not significantly affect the performance of \sbs. The reason is that, if the modes of the \bd that contain most of the density mass are the ones around the global minimizers, then \svgd would succeed in moving some particles in the areas of those modes, provided there are enough particles. The three following parameters control how the mass is repartitioned: the value of $\kappa$, the ratio between the volume of the region of local minimizers and the volume of the region of global minimizers, and the value of the function at the local minimizers. These three parameters are interdependent: the value of one can compensate for the value of the others. It is rather unlikely to encounter a function where these three parameters do not compensate each other. It would require the function to have an arbitrary small volume ratio or an arbitrary small distance between the local and global minima (see \cref{fig:ex-choice-kappa}). Thus, a very large $\kappa$, such as $10^3$, compensates almost all potential issues related to the geometry of the function and ensures a good performance on average.

\begin{figure*}[tb]
  \vspace{-1em}
   \centering
   \begin{subfigure}{\textwidth}
     \centering
     \begin{subfigure}{0.35\textwidth}
     \centering
     \includegraphics[width=0.78\textwidth]{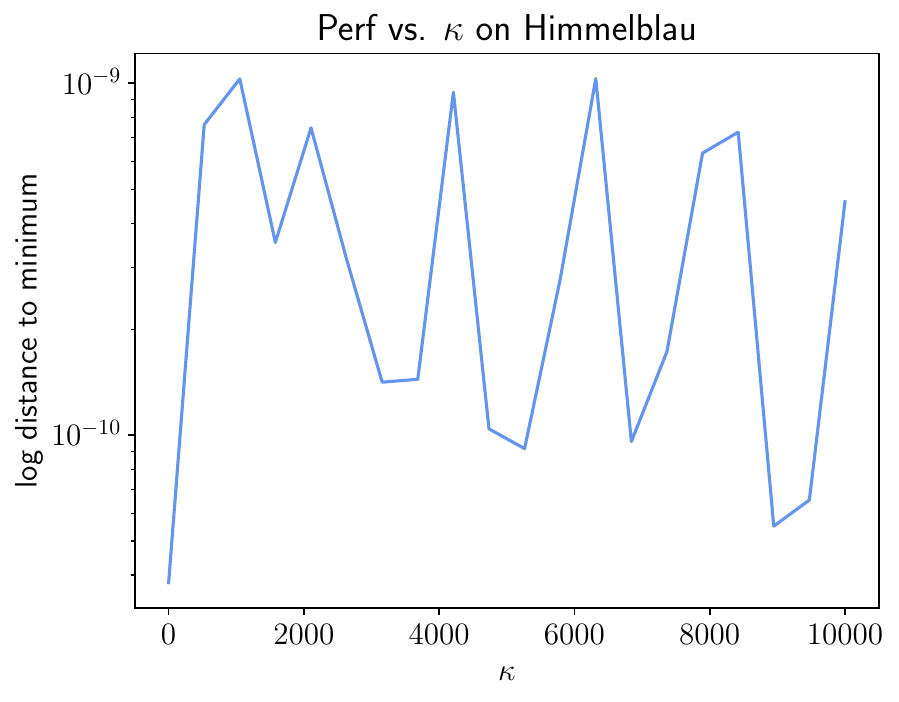}
     \end{subfigure}
     \hspace{1em}
     \begin{subfigure}{0.35\textwidth}
     \centering
     \includegraphics[width=0.78\textwidth]{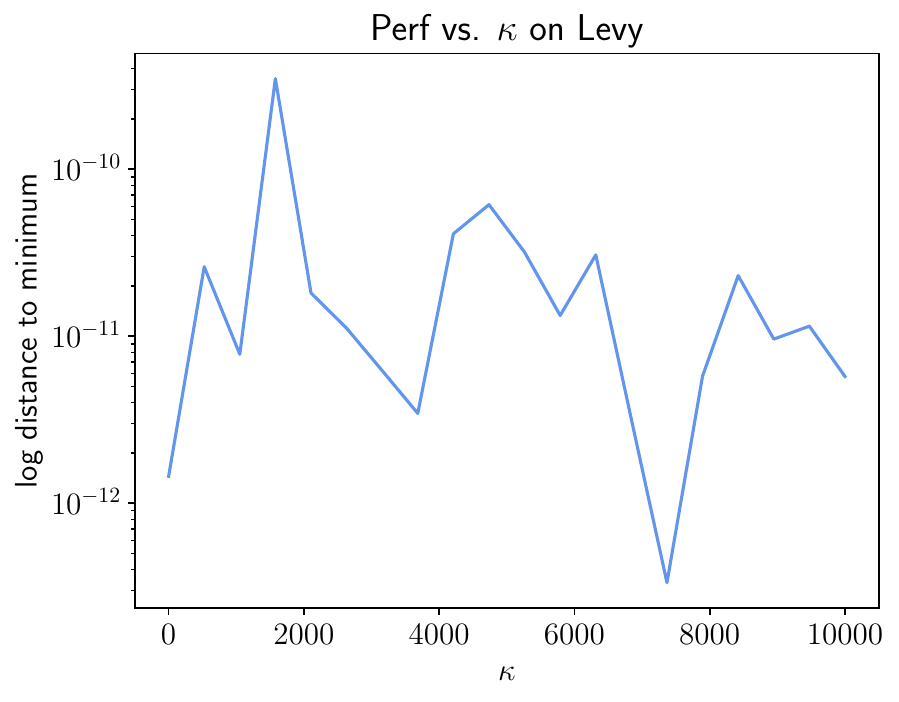}
     \end{subfigure}
     \caption{Log-distance between the computed solution and the global minimum \vs the value of $\kappa$ for \sbs.}
     \label{fig:kappa-comparison}
   \end{subfigure}
   \\
   \vspace{0.6em}
   \begin{subfigure}{\textwidth}
     \centering
     \begin{subfigure}{0.35\textwidth}
     \centering
     \includegraphics[width=0.7\textwidth]{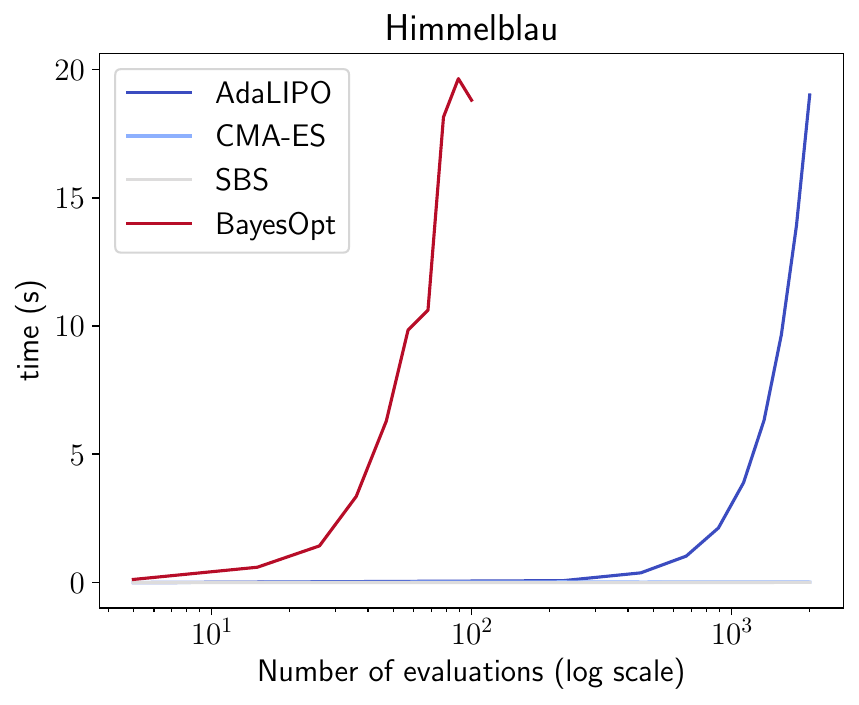}
     \end{subfigure}
     \hspace{1em}
     \begin{subfigure}{0.35\textwidth}
     \centering
     \includegraphics[width=0.7\textwidth]{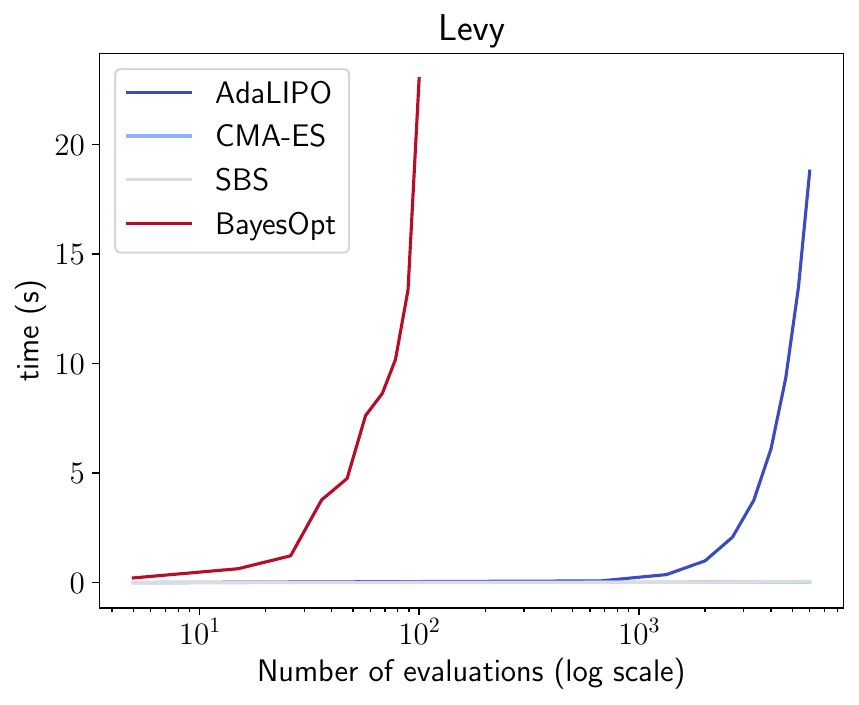}
     \end{subfigure}
     \caption{\footnotesize Execution time \vs the number of evaluations for \bayes, \adalipo, \cma, and \sbs}
     \label{fig:time-vs-eval}
   \end{subfigure}
   \vspace{-1.em}
   \caption{Insights for the compared algorithms:~\textbf{a)}~shows the low impact of $\kappa$ on the performance of \sbs. \textbf{b)}~shows the time to run for \bayes and \adalipo grows exponentially and is significantly higher than for \cma or \sbs. In each case (a) and (b), the \textbf{left} plot is for the Himmelblau, and the \textbf{right} is for the Levy function.}
   \label{fig:kappa-time}
 \end{figure*}

\inlinetitle{Choice of $\sigma$}{.}~%
All \sbs variants use the RBF kernel with a bandwidth $\sigma$, the choice of which is crucial for their performance. As detailed in \cref{sec:discussions}, the size of $\sigma$ controls the forces developed between particles. When a lot of particles are close together, they repel each other. This behavior enforces the exploration of the function domain, but at the same time it prevents \sbs from converging at narrow regions where global minimizers could be located. A natural choice is $\sigma = \frac{1}{N^2}$, where $N$ is the number of particles, which ensures that, when the particles are few, $\sigma$ gets large enough and \sbs explores the domain. On the other hand, with a lot a particles the exploration is ensured by the initial uniform distribution $\rho_0$ and the small induced $\sigma$ allows the particles to converge to the global minimum, even in narrow regions. For the \sbspf variant, $\sigma$ changes during the optimization process, as particles are being filtered out. For the \sbshybrid variant, as the initial particles are supposed to be already well-positioned and possibly close to the global minimum, $\sigma$ is set to a very small value, \eg $\sigma = 10^{-10}$.
\section{Experimental evaluation}\label{sec:benchmark}

Among the \sbs variants, we provide one named \sbspfhybrid that employs \sbspf as the final step of \sbshybrid. In our numerical comparison, we compare the \sbs variants with the following state-of-the-art global optimization methods: \cma \cite{Hansen1996}; \woa \cite{Mirjalili2016}, a particle-swarm method; \adalipo \cite{Malherbe2017}; \bayes \cite{Martinezcantin2014}; \langevin, a similar method to \sbs but using \mala instead of \svgd \cite{Grenander1994, Welling2011, Raginsky2017}; and \cbo \cite{Pinnau2017} that is a stochastic particle-based method similar to \sbs (we 
implemented the algorithm as presented in \cite{Pinnau2017}).  

We use classical two dimensional benchmark functions for global optimization. Some are noisy and multimodal (Ackley, Drop wave, Egg Holder, Holder Table, Michalewicz, Rastrigin, Levy), others are smooth (Branin, Goldstein Price, Himmelblau, Rosenbrock, Camel, Sphere) (see more information in \cite{simulationlib}). We provide the implementation of this experiment\!\alink{github.com/gaetanserre/Stochastic-Global-Optimization}\!. For the results of \cref{tab:2d-benchmark}, we ran each method $10$ times on each function.

In the literature, the budget is defined as the number of function evaluations, however, the actual computational time can vary significantly between the methods and needs to be taken into account. Thus, the budget is set in order for the methods to stop in a reasonable time when run on a personal computer\footnote{The experiments were performed on an Apple M2 chip with $8$ cores and $16$GB of RAM.}. As one can see in \cref{fig:time-vs-eval}, the running time of \adalipo and \bayes is significantly higher than for the other methods. For this reason, their budget is set lower than for the other methods: $2$K for \adalipo, $100$ for \bayes, and $800$K for the rest. 

We introduce the following \textit{empirical competitive ratio}:
$$
\text{ECR}(m) = \frac{1}{|F|} \sum_{f \in F} \min \left(100, \frac{df_\text{m}}{df^*} \right),
$$
where $F$ is the set of benchmark functions, $df_\text{m}$ is the distance between the global minimum of $f$ and the approximation found by the method m, and $df^*$ is the smallest distance among all the methods. \text{ECR} provides information about the average precision compared to the best method (lower is better, and the best is $1$).

\begin{table*}[h]
  \fontsize{7pt}{7pt}\selectfont
  \caption{\footnotesize Comparison between all \sbs variants with several state-of-the-art methods on two dimensional benchmark functions. For each function, we report the average distance to the global minimum and standard deviation (lower is better). The precision is truncated. \sbshybrid runs $1$K iterations of {\sc cma-es} and {\sc woa}. As one can see, \sbshybrid and \sbs respectively rank $1^\text{st}$ and $4^\text{th}$ while \sbspfhybrid and \sbspf achieve competitive results with significantly less evaluations (respectively $\sim 67\%$ and $\sim 97\%$ budget reduction).}
  \vspace{-1.5em}
  \label{tab:2d-benchmark}
  \vskip 0.15in

  \begin{center}
    \begin{sc}
      \addtolength{\tabcolsep}{-0.2em}
      \begin{tabular}{lllllll||rrrr}
        \cline{2-11}
        \rule{0pt}{1em} & \multicolumn{6}{c||}{state-of-the-art} & \multicolumn{4}{c}{proposed method and variants} \\
        \toprule
        functions & langevin & bayesopt & cbo & adalipo & cma-es & woa & sbs-pf & sbs & sbs-pf-hybrid & sbs-hybrid \\
        \midrule
        Ackley & \makecell{$6.8$ \\ $\pm 1.7$} & \makecell{$0.1$ \\ $\pm 0.1$} & \makecell{$0.0$ \\ $\pm 0.0$} & \makecell{$1.3$ \\ $\pm 0.7$} & \makecell{$19.8$ \\ $\pm 0.0$} & \makecell{$\mathbf{8e^{-8}}$ \\ $\mathbf{\pm 5e^{-8}}$} & \makecell{$0.0$ \\ $\pm 9e^{-4}$} & \makecell{$8e^{-4}$ \\ $\pm 3e^{-4}$} & \makecell{$1e^{-5}$ \\ $\pm 1e^{-5}$} & \makecell{$7e^{-6}$ \\ $\pm 3e^{-6}$} \\
        Branin & \makecell{$8e^{-5}$ \\ $\pm 1e^{-4}$} & \makecell{$2e^{-4}$ \\ $\pm 1e^{-4}$} & \makecell{$0.4$ \\ $\pm 0.3$} & \makecell{$0.0$ \\ $\pm 0.0$} & \makecell{$3e^{-7}$ \\ $\pm 0$} & \makecell{$1e^{-6}$ \\ $\pm 1e^{-6}$} & \makecell{$5e^{-7}$ \\ $\pm 2e^{-7}$} & \makecell{$3e^{-7}$ \\ $\pm 3e^{-11}$} & \makecell{$\mathbf{3e^{-7}}$ \\ $\mathbf{\pm 0}$} & \makecell{$3e^{-7}$ \\ $\pm 5e^{-16}$} \\
        Drop Wave & \makecell{$0.9$ \\ $\pm 1e^{-16}$} & \makecell{$0.2$ \\ $\pm 0.1$} & \makecell{$0.1$ \\ $\pm 0.1$} & \makecell{$0.1$ \\ $\pm 0.0$} & \makecell{$0.3$ \\ $\pm 0.3$} & \makecell{$\mathbf{1e^{-15}}$ \\ $\mathbf{\pm 1e^{-15}}$} & \makecell{$0.0$ \\ $\pm 0.0$} & \makecell{$0.0$ \\ $\pm 0.0$} & \makecell{$0.1$ \\ $\pm 0.0$} & \makecell{$0.0$ \\ $\pm 0.0$} \\
        Egg Holder & \makecell{$2008.8$ \\ $\pm 2e^{-13}$} & \makecell{$90.3$ \\ $\pm 60.6$} & \makecell{$833.5$ \\ $\pm 1e^{-10}$} & \makecell{$40.0$ \\ $\pm 18.7$} & \makecell{$393.6$ \\ $\pm 5e^{-14}$} & \makecell{$\mathbf{3e^{-5}}$ \\ $\mathbf{\pm 5e^{-10}}$} & \makecell{$18.0$ \\ $\pm 19.0$} & \makecell{$8.0$ \\ $\pm 10.2$} & \makecell{$7.8$ \\ $\pm 9.6$} & \makecell{$21.5$ \\ $\pm 32.0$} \\
        Goldstein Price & \makecell{$391387.4$ \\ $\pm 301031.1$} & \makecell{$8.3$ \\ $\pm 6.6$} & \makecell{$4e^{-4}$ \\ $\pm 2e^{-4}$} & \makecell{$0.4$ \\ $\pm 0.3$} & \makecell{$8.1$ \\ $\pm 24.3$} & \makecell{$1e^{-6}$ \\ $\pm 7e^{-7}$} & \makecell{$6e^{-7}$ \\ $\pm 5e^{-7}$} & \makecell{$1e^{-9}$ \\ $\pm 2e^{-9}$} & \makecell{$\mathbf{6e^{-14}}$ \\ $\mathbf{\pm 6e^{-15}}$} & \makecell{$6e^{-13}$ \\ $\pm 1e^{-12}$} \\
        Himmelblau & \makecell{$47.5$ \\ $\pm 31.0$} & \makecell{$7e^{-4}$ \\ $\pm 9e^{-4}$} & \makecell{$0.1$ \\ $\pm 0.1$} & \makecell{$0.0$ \\ $\pm 0.0$} & \makecell{$9e^{-16}$ \\ $\pm 1e^{-15}$} & \makecell{$1e^{-6}$ \\ $\pm 1e^{-6}$} & \makecell{$1e^{-7}$ \\ $\pm 3e^{-7}$} & \makecell{$5e^{-11}$ \\ $\pm 4e^{-11}$} & \makecell{$4e^{-19}$ \\ $\pm 7e^{-19}$} & \makecell{$\mathbf{1e^{-20}}$ \\ $\mathbf{\pm 1e^{-20}}$} \\
        Holder Table & \makecell{$3.4$ \\ $\pm 0.5$} & \makecell{$0.4$ \\ $\pm 0.9$} & \makecell{$0.0$ \\ $\pm 7e^{-4}$} & \makecell{$0.0$ \\ $\pm 0.0$} & \makecell{$5.0$ \\ $\pm 5.0$} & \makecell{$\mathbf{2e^{-6}}$ \\ $\mathbf{\pm 2e^{-6}}$} & \makecell{$2e^{-6}$ \\ $\pm 1e^{-7}$} & \makecell{$2e^{-6}$ \\ $\pm 2e^{-9}$} & \makecell{$2e^{-6}$ \\ $\pm 1e^{-10}$} & \makecell{$2e^{-6}$ \\ $\pm 1e^{-10}$} \\
        Michalewicz & \makecell{$9.6$ \\ $\pm 0.2$} & \makecell{$7.9$ \\ $\pm 1e^{-5}$} & \makecell{$8.6$ \\ $\pm 0.0$} & \makecell{$7.9$ \\ $\pm 0.0$} & \makecell{$8.0$ \\ $\pm 0.3$} & \makecell{$7.9$ \\ $\pm 1e^{-6}$} & \makecell{$7.9$ \\ $\pm 1e^{-6}$} & \makecell{$7.9$ \\ $\pm 1e^{-10}$} & \makecell{$7.9$ \\ $\pm 8e^{-14}$} & \makecell{$\mathbf{7.9}$ \\ $\mathbf{\pm 4e^{-15}}$} \\
        Rastrigin & \makecell{$30.5$ \\ $\pm 3.7$} & \makecell{$2.2$ \\ $\pm 1.2$} & \makecell{$9e^{-4}$ \\ $\pm 7e^{-4}$} & \makecell{$0.2$ \\ $\pm 0.2$} & \makecell{$5.4$ \\ $\pm 5.8$} & \makecell{$\mathbf{6e^{-15}}$ \\ $\mathbf{\pm 7e^{-15}}$} & \makecell{$5e^{-6}$ \\ $\pm 3e^{-6}$} & \makecell{$5e^{-9}$ \\ $\pm 1e^{-8}$} & \makecell{$0.5$ \\ $\pm 0.7$} & \makecell{$0.3$ \\ $\pm 0.6$} \\
        Rosenbrock & \makecell{$6852.4$ \\ $\pm 4943.7$} & \makecell{$0.2$ \\ $\pm 0.2$} & \makecell{$0.0$ \\ $\pm 0.0$} & \makecell{$0.1$ \\ $\pm 0.0$} & \makecell{$4e^{-16}$ \\ $\pm 7e^{-16}$} & \makecell{$2e^{-7}$ \\ $\pm 1e^{-7}$} & \makecell{$6e^{-5}$ \\ $\pm 9e^{-5}$} & \makecell{$2e^{-6}$ \\ $\pm 4e^{-6}$} & \makecell{$5e^{-17}$ \\ $\pm 8e^{-17}$} & \makecell{$\mathbf{1e^{-17}}$ \\ $\mathbf{\pm 1e^{-17}}$} \\
        Camel & \makecell{$397.9$ \\ $\pm 9.0$} & \makecell{$0.0$ \\ $\pm 0.0$} & \makecell{$0.0$ \\ $\pm 0.0$} & \makecell{$0.0$ \\ $\pm 0.0$} & \makecell{$2e^{-5}$ \\ $\pm 1e^{-14}$} & \makecell{$2e^{-5}$ \\ $\pm 2e^{-8}$} & \makecell{$\mathbf{2e^{-5}}$ \\ $\mathbf{\pm 1e^{-7}}$} & \makecell{$2e^{-5}$ \\ $\pm 1e^{-11}$} & \makecell{$2e^{-5}$ \\ $\pm 0$} & \makecell{$2e^{-5}$ \\ $\pm 1e^{-16}$} \\
        Levy & \makecell{$83.5$ \\ $\pm 12.9$} & \makecell{$0.1$ \\ $\pm 0.1$} & \makecell{$0.0$ \\ $\pm 0.0$} & \makecell{$0.0$ \\ $\pm 0.0$} & \makecell{$1.0$ \\ $\pm 2.1$} & \makecell{$8e^{-9}$ \\ $\pm 7e^{-9}$} & \makecell{$9e^{-8}$ \\ $\pm 8e^{-8}$} & \makecell{$1e^{-12}$ \\ $\pm 2e^{-12}$} & \makecell{$3e^{-19}$ \\ $\pm 8e^{-19}$} & \makecell{$\mathbf{9e^{-20}}$ \\ $\mathbf{\pm 2e^{-19}}$} \\
        Sphere & \makecell{$9e^{-5}$ \\ $\pm 6e^{-5}$} & \makecell{$5e^{-4}$ \\ $\pm 5e^{-4}$} & \makecell{$0.0$ \\ $\pm 0.0$} & \makecell{$0.0$ \\ $\pm 9e^{-4}$} & \makecell{$5e^{-16}$ \\ $\pm 1e^{-15}$} & \makecell{$1e^{-16}$ \\ $\pm 8e^{-17}$} & \makecell{$5e^{-8}$ \\ $\pm 7e^{-8}$} & \makecell{$6e^{-12}$ \\ $\pm 7e^{-12}$} & \makecell{$1e^{-19}$ \\ $\pm 2e^{-19}$} & \makecell{$\mathbf{1e^{-21}}$ \\ $\mathbf{\pm 3e^{-21}}$} \\
        \midrule
        ECR & $62.2$ & $46.8$ & $46.7$ & $46.7$ & $24.0$ & $22.4$ & $46.7$ & $20.8$ & $16.2$ & $\mathbf{15.4}$ \\
        \midrule
        Average rank & $9.38$ & $7.85$ & $7.46$ & $7.15$ & $6.69$ & $3.00$ & $4.15$ & $3.38$ & $3.15$ & $\mathbf{2.77}$ \\
        \midrule
        Final rank & $10$ & $9$ & $8$ & $7$ & $6$ & $2$ & $5$ & $4$ & $3$ & $\mathbf{1}$ \\
        \bottomrule
      \end{tabular}
    \end{sc}
  \end{center}
  \vskip -0.1in
\end{table*}

In the results of \cref{tab:2d-benchmark}, one can see that \sbs outperforms almost all state-of-the-art methods and scores the fourth rank on average. \sbspf achieves comparable results on average with significantly less function evaluations ($\sim 97\%$ budget reduction). Moreover, \sbshybrid and \sbspfhybrid outperform all the other methods on average. They combine the efficiency of either {\sc cma-es} and {\sc woa} with the suitability of \sbs for large budgets, while the addition of particle filtering reduces the budget by $\sim 67\%$. In parallel, \sbs, \sbspfhybrid, and \sbshybrid score respectively the $3^\text{rd}$, $2^\text{nd}$, and $1^\text{st}$ rank on the competitive ratio measure, showing that their approximations are precise on average compared to the other methods. 

In the appendix, we provide the results of the same experiment on $50$ dimensional benchmark functions, by restricting to only methods that can run in low computational time (see \cref{tab:50d-benchmark}). There, the budget is set to $8$M. One can observe that \sbs and its variants outperform all competitors, and \sbspf achieves the best results with a budget reduction of $\sim 97\%$, compared to \sbs.  However, the budget reduction of \sbspfhybrid is less significant ($\sim 16\%$), due to the fact that the high dimensionality of the functions and the initial distribution makes the less promising particles harder to distinguish. Overall, all \sbs variants seem robust to function shapes, contrary to \cma for instance, which is very precise on valley-shaped functions but struggles on the multimodal functions.
\section{Discussion}\label{sec:discussions}
\inlinetitle{Link with Simulated Annealing}{.}~%
The link between \sbs and Simulated Annealing \cite{Kirkpatrick1983} is not difficult to see. Indeed, both algorithms are asymptotic methods that sample from the \bd. However, the way they sample from that distribution is different. Simulated Annealing is a Markov Chain Monte-Carlo method \cite{Azencott1989}, while \sbs is a deterministic variational approach. The minimum temperature parameter of Simulated Annealing is the inverse of the $\kappa$ parameter of \sbs. Thus, any scheduler for Simulated Annealing's temperature can also be used in \sbs. However, there is an extra degree of exploration/exploitation in \sbs, brought by the kernel bandwidth employed by the \svgd sampling.

\inlinetitle{Locality of the kernel}{.}~%
In classical \svgd implementations, the RBF kernel used is:
$
k(x, x') = \exp \left( -\lVert x - x' \rVert_2^2 / 2 \sigma^2 \right),
$
as it is in the Stein class of any smooth density supported on $\R^d$.
The bandwidth $\sigma$ controls the locality of the attraction and repulsion forces applied on the particles, expressed as:
\begin{align*}
  \text{attr}(x) &= \E_{x' \sim \hat{\rho}_i} \left[ s(x') k(x, x') \right],\\
  \text{rep}(x) &= \E_{x' \sim \hat{\rho}_i} \left[ \nabla_{x'}k(x, x') \right].
\end{align*}
The first term attracts remote particles to a close cluster of particles, and the second term repels particles that are too close to each other. Hence, they are respectively exploitation and exploration forces. Indeed, the attraction allows particles to ``fall'' in local minima, wherein a lot of particles are already stuck. The repulsion prevents particles from getting stuck together at a narrow region of the search space, and forces them to explore the space. $\sigma$ controls the range of these forces. A small $\sigma$ value leads to a weak repulsion and thus more exploitation. An arbitrary small $\sigma$ leads to a uniform distribution over the local minima. In the contrary, a large $\sigma$ leads to more exploration, as the particles will repel themselves even from a very far distance. An arbitrary large $\sigma$ leads to a uniform discretization of the space. These behaviors are illustrated in \cref{fig:sbs-boltzmann}. In the case of \sbs, the value of $\sigma$ is a user parameter.

\inlinetitle{Some weaknesses}{.}~%
Because of the gradient approximation by finite difference that occurs in \sbs, our methods require a large budget. That is a common issue in gradient-based optimization algorithms. However, in the contrary of more frugal approches (\eg \bayes), \sbs and its variants have a way smaller execution time. Another weakness of \sbs is the difficult choice of the kernel. As explained above, the kernel controls the particles movements and the performance of one specific kernel choice highly depends on the geometry of the objective function. This choice is crucial and future users should tune this hyperparameter carefully. A way to mitigate this problem would be to find an adaptive kernel that uses only evaluations of the objective function to choose the best way for the particles to interact. However, we believe that is a quite complex subject that is out of the scope of the current study, and should be investigated in a more general point of view.

\section{Conclusion}
In this paper, we introduced the Stein Boltzmann Sampling (\sbs) method for global optimization, along some variants. We studied \svgd on a compact domain, without any regularity of its boundary. With a kernel that vanishes on the boundary, we proved that its continuous-time dynamics exists, keeps the particles in the domain and decreases the Kullback--Leibler divergence to the target. We also proved that it converges weakly to the target when the kernel is Stein-universal. These results apply to the \bd of any continuously differentiable function, and they are formalized in Lean. We showed in our experimental evaluation that \sbs outperforms state-of-the-art methods on average on classical benchmark functions, that \sbspf can lead to drastic reduction of the needed computational budget while having comparable performance than the original \sbs version, and that \sbshybrid outperforms all the other methods in practice. This work suggests that, for obtaining the best trade-off between accurate approximations and low budget, \sbs should be used as a continuation for others particles or distribution-based methods, conjointly with particles filtering strategies (\sbspfhybrid). As future work, a Stein-universal kernel that vanishes on the boundary of the domain can be constructed and evaluated in the experiments. The convergence rate of \sbs and its components can be further studied, and more sophisticated particle filtering strategies can be designed to make it more appealing for global optimization in real-world applications.

\acknowledgments{
  The authors acknowledge the support from the Industrial Data Analytics and Machine Learning Chair hosted at ENS Paris-Saclay. Also, we sincerely thank the anonymous reviewers for their valuable comments and suggestions, which have greatly helped improve the quality of this paper.
}

{
\small
\bibliographystyle{plain}
\bibliography{refs}
}
%

\newpage
\onecolumn

\pagebreak

\appendix

\begin{table}[H]
  \fontsize{7pt}{7pt}\selectfont
  \caption{\footnotesize Comparison between all \sbs variants with several state-of-the-art methods on $50$ dimensional benchmark functions. For each function, we report the average distance to the global minimum (lower is better). The precision is truncated. \sbshybrid runs $1$K iterations of {\sc cma-es} and {\sc woa}. As one can see, \sbspf ranks $1^\text{s}$, as \sbs, while having a significant budget reduction ($\sim 97\%$). Moreover, \sbspfhybrid outperforms \sbshybrid while having a budget reduction of $\sim 16\%$. The high dimensionality of the functions makes the particle filtering of \sbspfhybrid less efficient.}
   \vspace{-1.5em}
  \label{tab:50d-benchmark}
  \vskip 0.15in
  \begin{center}
    \begin{sc}
      \begin{tabular}{lllll||rrrr}
        \cline{2-9}
        \rule{0pt}{1em} & \multicolumn{4}{c||}{state-of-the-art} & \multicolumn{4}{c}{proposed methods} \\
        \toprule
        functions & langevin & cbo & woa & cma-es & sbs-hybrid & sbs-pf-hybrid & sbs & sbs-pf \\
        \midrule
        Ackley & \makecell{$21.5$ \\ $\pm 0.0$} & \makecell{$21.1$ \\ $\pm 0.1$} & \makecell{$19.8$ \\ $\pm 0.1$} & \makecell{$19.6$ \\ $\pm 0.0$} & \makecell{$19.5$ \\ $\pm 0.1$} & \makecell{$19.6$ \\ $\pm 0.1$} & \makecell{$\mathbf{19.0}$ \\ $\mathbf{\pm 0.1}$} & \makecell{$19.0$ \\ $\pm 0.1$} \\
        Michalewicz & \makecell{$8.6$ \\ $\pm 0.6$} & \makecell{$\mathbf{0.8}$ \\ $\mathbf{\pm 0.5}$} & \makecell{$4.4$ \\ $\pm 0.2$} & \makecell{$25.8$ \\ $\pm 2.4$} & \makecell{$24.2$ \\ $\pm 1.8$} & \makecell{$24.1$ \\ $\pm 2.5$} & \makecell{$3.8$ \\ $\pm 0.4$} & \makecell{$2.4$ \\ $\pm 0.8$} \\
        Rastrigin & \makecell{$884.4$ \\ $\pm 202.5$} & \makecell{$841.1$ \\ $\pm 302.9$} & \makecell{$593.8$ \\ $\pm 22.2$} & \makecell{$\mathbf{107.4}$ \\ $\mathbf{\pm 20.9}$} & \makecell{$114.5$ \\ $\pm 24.9$} & \makecell{$115.0$ \\ $\pm 10.9$} & \makecell{$280.5$ \\ $\pm 17.2$} & \makecell{$280.5$ \\ $\pm 17.2$} \\
        Rosenbrock & \makecell{$367368.5$ \\ $\pm 22139.8$} & \makecell{$176596.0$ \\ $\pm 0$} & \makecell{$20145.3$ \\ $\pm 3609.2$} & \makecell{$\mathbf{0.4}$ \\ $\mathbf{\pm 1.2}$} & \makecell{$28.9$ \\ $\pm 1.3$} & \makecell{$28.7$ \\ $\pm 2.0$} & \makecell{$39.5$ \\ $\pm 1.4$} & \makecell{$25.3$ \\ $\pm 4.0$} \\
        Levy & \makecell{$2931.7$ \\ $\pm 88.4$} & \makecell{$349.8$ \\ $\pm 38.3$} & \makecell{$224.5$ \\ $\pm 13.4$} & \makecell{$75.3$ \\ $\pm 27.0$} & \makecell{$81.1$ \\ $\pm 19.6$} & \makecell{$70.4$ \\ $\pm 23.2$} & \makecell{$\mathbf{55.9}$ \\ $\mathbf{\pm 2.6}$} & \makecell{$57.8$ \\ $\pm 4.4$} \\
        Sphere & \makecell{$0.0$ \\ $\pm 3e^{-4}$} & \makecell{$5000.0$ \\ $\pm 0$} & \makecell{$646.3$ \\ $\pm 31.5$} & \makecell{$1e^{-14}$ \\ $\pm 2e^{-15}$} & \makecell{$\mathbf{7e^{-20}}$ \\ $\mathbf{\pm 1e^{-20}}$} & \makecell{$2e^{-19}$ \\ $\pm 2e^{-19}$} & \makecell{$2e^{-10}$ \\ $\pm 1e^{-11}$} & \makecell{$1e^{-4}$ \\ $\pm 5e^{-6}$} \\
        \midrule
        ECR & $43.9$ & $36.0$ & $35.3$ & $\mathbf{1.5}$ & $13.4$ & $13.2$ & $18.1$ & $28.2$ \\
        \midrule
        Average rank & $7.17$ & $6.17$ & $5.83$ & $3.67$ & $3.67$ & $3.50$ & $\mathbf{3.00}$ & $\mathbf{3.00}$ \\
        \midrule
        Final rank & $7$ & $6$ & $5$ & $4$ & $3$ & $2$ & $\mathbf{1}$ & $\mathbf{1}$ \\
        \bottomrule
      \end{tabular}
    \end{sc}
  \end{center}
  \vskip -0.1in
\end{table}

\begin{table}[H]
  \fontsize{8pt}{8pt}\selectfont
  \caption{\small Collection of all notations and their meanings}
  \vspace{-0.5em}
  \label{tab:notations}
  \vskip 0.15in
  \centering
  {
    \renewcommand{\arraystretch}{1.5}%
    \begin{tabular}{c|l}
      \textbf{Notation} & \textbf{Definition} \\
      \hline
      $f$ & function to minimize \\
      $d$ & dimension of the domain of $f$ \\
      $\Omega$ & compact subset of $\R^d$, domain of $f$ \\
      $\partial \Omega$, $\Omega^\circ$ & boundary and interior of $\Omega$ \\
      $X^*$ & set of global minimizers of $f$ \\
      $\lambda$ & Lebesgue measure \\
      $m^{(\kappa)}$ & density of the \bd with parameter $\kappa$ \\
      $\cal{P}(\Omega)$ & set of the probability measures on $\R^d$ carried by $\Omega$ \\
      $\rho_n \rightharpoonup \rho$ & weak convergence of $(\rho_n)_{n \in \N}$ to $\rho$ in $\cal{P}(\Omega)$ \\
      $\pi$ & target distribution, the \bd of $f$ in \sbs context \\
      $p$, $s$ & density and score of $\pi$ \\
      $k$ & kernel of \svgd \\
      $\Phi$, $H_0$ & feature map of $k$, with values in the Hilbert space $H_0$ \\
      $\H$ & $H_0^d$, the space of the vector fields of \svgd \\
      $g_\Phi$ & vector field defined by $g \in \H$ \\
      $\A_s$ & Stein operator of $\pi$ \\
      $\phi^*_\rho$ & Stein witness of $\rho$ \\
      $\K(\rho | \pi)$ & kernelized Stein discrepancy \\
      $\KL(\rho \,\|\, \pi)$ & Kullback--Leibler divergence \\
      $(\rho_n)_{n \in \N}$ & sequence of measures constructed by \svgd \\
      $\hat{\rho}_n$ & empirical measure of the \svgd particles \\
      $(T_t)_{t \geq 0}$ & solution of the continuous-time \svgd equation \\
      $\rho_t$ & $(T_t)_\# \rho_0$, the distribution of the particles at time $t$ \\
      \hline
    \end{tabular}
  }
\end{table}

\section{Boltzmann distribution}\label{sec:boltzmann}
Recall that the \bd has been formally defined, in \cref{def:boltzmann-distribution}.
The \bd is a well-known distribution in statistical physics. It is used to model the distribution of the energy of a system in thermal equilibrium. The parameter $\kappa$ is called the {\it inverse temperature}. The higher $\kappa$ is, the more concentrated the mass is around the minima of $f$. When $\kappa$ tends to infinity, the \bd tends to a distribution supported over the minima of $f$.
The \bd is typically used in a discrete settings, \ie where the number of states is finite. The continuous version can be defined using the Gibbs measure. The following properties come from \cite{Luo2019}. For the sake of completeness, we provide the proofs in \cref{app:proof-boltzmann-properties}. 
\begin{properties}[Properties of the Boltzmann distribution]\label{properties:boltzmann-properties}
  Let $m^{(\kappa)}$ be defined as in \cref{def:boltzmann-distribution}. Then, we have the following properties:
  \begin{itemize}
    \item If $\lambda(X^*) = 0$, then, $\forall x \in \Omega$,
    $$
    \lim_{\kappa \to \infty} m^{(\kappa)}(x) = \begin{cases}
      \infty &\textrm{if } x \in X^* \\
      0 &\textrm{otherwise.}
    \end{cases}
    $$%
    \item If $0 < \lambda(X^*)$, then, $\forall x \in \Omega$,
    $$
    \lim_{\kappa \to \infty} m^{(\kappa)}(x) = \begin{cases}
      \lambda(X^*)^{-1} &\textrm{if } x \in X^* \\
      0 &\textrm{otherwise}.
    \end{cases}
    $$%
    \item $\forall f \in C^0(\Omega, \R)$,
    $$
    \lim_{\kappa \to \infty} \int_\Omega f(x) \; m^{(\kappa)}(x) \difi x = f^*.
    $$%
  \end{itemize}
\end{properties}
A visual representation of the \bd is given in \cref{fig:boltzmann-distribution}. One can see that, as $\kappa$ increases, $m^{(\kappa)}$ becomes more and more concentrated around the minima of $f$. We use the \bd induced by the density $m^{(\kappa)}$ (also denoted by $m^{(\kappa)}$ for simplicity) of \cref{eq:BD}. We provide the proof of the properties in \cref{app:proof-boltzmann-properties}. To sample from this distribution, we need to compute the integral $\int_\Omega e^{-\kappa f(t)} \difi t$, which however, is likely to be intractable for a general $f$.
\subsection{Proof of \cref{properties:boltzmann-properties}}\label{app:proof-boltzmann-properties}
The continuous \bd is a special case of the nascent minima distribution, introduced in \cite{Luo2019}, that has the generic form
\begin{equation}\label{eq:nascent-minima-distribution}
  m_{f, \Omega}^{(\kappa)}(x) = m^{(\kappa)}(x) = \frac{\tau^\kappa (f(x))} { \int_\Omega \tau^\kappa (f(t)) \difi t},
\end{equation}
where $\tau : \R \to \R_{>0}$ is monotonically decreasing. We have the following theorems for general $\tau$.
\begin{theorem}[Nascent minima distribution properties]\label{theorem:nascent-minima-distribution-properties}
  Let $m^{(\kappa)}$ and $\tau$ be defined in \cref{eq:nascent-minima-distribution}. Then, we have the following properties:
  \begin{itemize}
    \item If $\lambda(X^*) = 0$, then, $\forall x \in \Omega$,
    $$
    \lim_{\kappa \to \infty} m^{(\kappa)}(x) = \begin{cases}
      \infty &\textrm{if } x \in X^* \\
      0 &\textrm{otherwise}
    \end{cases}.
    $$
    \item If $0 < \lambda(X^*)$, then, $\forall x \in \Omega$,
    $$
    \lim_{\kappa \to \infty} m^{(\kappa)}(x) = \begin{cases}
      \lambda(X^*)^{-1} &\textrm{if } x \in X^* \\
      0 &\textrm{otherwise}
    \end{cases}.
    $$
  \end{itemize}
\end{theorem}
\begin{proof}
  Let's prove the two properties together. Let $p = \tau(f(x')) > 0$, $\forall x' \notin X^*$. Then, $\exists \Omega_p$, such that $0 < \lambda(\Omega_p)$, $p < \tau(f(t))$, \ie $f(t) < f(x')$. Thus,
  \begin{align*}
    m^{(\kappa)}(x') &= \frac{p^\kappa}{\int_{\Omega_p} \tau^\kappa(f(t)) \difi t + \int_{\Omega / \Omega_p} \tau^\kappa(f(t)) \difi t } \\
    &\le \frac{p^\kappa}{\int_{\Omega_p} \tau^\kappa(f(t)) \difi t} \\
    &= \frac{1}{\int_{\Omega_p} p^{-\kappa} \tau^\kappa(f(t)) \difi t}.
  \end{align*}
  For any $t$ in $\Omega_p$, $p^{-1} \tau(f(t)) > 1$. Therefore $\lim_{\kappa \to,\infty} \int_{\Omega_p} p^{-\kappa} \tau^\kappa(f(t)) \difi t = \infty$. Hence,
  $$
  \forall x' \notin X^*, \lim_{\kappa \to \infty} m^{(\kappa)}(x) = 0.
  $$
  Now, let's consider any $x'' \in X^*$ and $p = \tau(f(x''))$. We have
  \begin{align*}
    m^{(\kappa)}(x'') &= \frac{p^\kappa}{\int_\Omega \tau^\kappa(f(t)) \difi t} \\
    &= \frac{1}{\int_X^* p^{-\kappa} \tau^\kappa(f(t)) \difi t + \int_{\Omega / X^*} p^{-\kappa} \tau^\kappa(f(t)) \difi t} \\
    &= \frac{1}{\int_{X^*} \difi t + \int_{\Omega / X^*} p^{-\kappa} \tau^\kappa(f(t)) \difi t} \; (\forall t \in X^*, \tau(f(t)) = p) \\
    &= \frac{1}{\lambda(X^*) + \int_{\Omega / X^*} p^{-\kappa} \tau^\kappa(f(t)) \difi t}.
  \end{align*}
  For any $t$ in $\Omega / X^*$, $p^{-1}\tau(f(t)) < 1$. Therefore, $\lim_{\kappa \to \infty} \int_{\Omega / X^*} p^{-\kappa} \tau^\kappa(f(t)) \difi t = 0$. Thus,
  $$
  \forall x'' \in X^*, \lim_{\kappa \to \infty} m^{(\kappa)}(x'') = \begin{cases}
    \infty &\textrm{if } \lambda(X^*) = 0 \\
    \frac{1}{\lambda(X^*)} &\textrm{otherwise}
  \end{cases}.
  $$
\end{proof}
\begin{theorem}[Convergence of expectation]
  $\forall f \in C^0(\Omega, \R)$, the following holds
  $$
  \lim_{\kappa \to \infty} \int_\Omega f(x) \; m^{(\kappa)}(x) \difi x = f^*.
  $$
  Moreover, if $X^* = {x^*}$, we have
  $$
  \lim_{\kappa \to \infty} \int_\Omega x \; m^{(\kappa)}(x) \difi x = x^*.
  $$
\end{theorem}
\begin{proof}
  If $f$ is constant, it is straightforward as $m^{(\kappa)}$ is a PDF. Suppose $f$ not constant on $\Omega$. For any $\eps > 0$, let $0 < \delta \triangleq \frac{\eps}{1 + \pare{\max_{x \in \Omega} f(x) - f^*}} \leq \eps$. As $f$ is continuous, $\exists \Omega_\delta = \bra{x \in \Omega \; | \; f(x) - f^* < \delta}$, the corresponding level set. Using \cref{theorem:nascent-minima-distribution-properties}, $\exists K \in \N$ such that
  $$
  \int_{\Omega / \Omega_\delta} m^{(\kappa)}(x) \difi x < \delta
  $$
  holds $\forall \kappa > K$, as $m^{(\kappa)}$ tends to $0$ $\forall x \notin X^*$. Thus,
  \begin{align*}
    0 < &\int_\Omega f(x) m^{(\kappa)}(x) \difi x - f^* \\ 
    = &\int_\Omega f(x) m^{(\kappa)}(x) \difi x - f^* \int_\Omega m^{(\kappa)}(x) \difi x \\
    = &\int_\Omega \pare{f(x) - f^*}m^{(\kappa)}(x) \difi x \\
    = &\int_{\Omega_\delta} \pare{f(x) - f^*}m^{(\kappa)}(x) \difi x \\
     & \;\;\;\;\;\;\;\;\;\;\;\; + \int_{\Omega / \Omega_\delta} \pare{f(x) - f^*}m^{(\kappa)}(x) \difi x \\
    < \; &\delta \int_{\Omega_\delta} m^{(\kappa)}(x) \difi x \\
    & \;\;\;\;\;\;\;\;\;\;\;\; + \pare{\max_{x \in \Omega} f(x) - f^*} \int_{\Omega / \Omega_\delta} m^{(\kappa)}(x) \difi x \\
    < \; &\delta (1 - \delta) + \pare{\max_{x \in \Omega} f(x) - f^*} \delta \\
    < \; &\pare{1 + \pare{\max_{x \in \Omega} f(x) - f^*}} \delta = \eps.
  \end{align*}
  The proof is similar for the second statement, by setting
  $$
  \Omega_\delta = \bra{x \in \Omega | \norm{x - x^*} < \delta}.
  $$
\end{proof}
Letting $\tau = x \mapsto e^{-x}$ gives \cref{properties:boltzmann-properties}.
\section{SVGD on a compact domain}\label{sec:svgd-compact}
\svgd \cite{Liu2016} transports a set of particles toward a target distribution $\pi$. At each iteration, the particles move along the direction of steepest descent of the Kullback--Leibler divergence to $\pi$. This direction is taken in the unit ball of a space of vector fields defined by a kernel. When the number of particles tends to infinity, the distribution $\rho_t$ of the particles follows a McKean-Vlasov equation with a zero diffusion (see \cref{eq:svgd-diff-vlasov}) \cite{Liu2017, Lu2019}. The analysis of this equation is done on $\R^d$, for smooth target distributions (\eg \cite{Lu2019, Korba2020, Duncan2023}).

In our global optimization setting, the search space $\Omega$ is compact. \sbs (\cref{sec:sbs-method}) applies \svgd to the Boltzmann distribution of the objective function on $\Omega$. Two difficulties arise. First, nothing keeps the particles in $\Omega$, and the theory on $\R^d$ does not apply. Second, the boundary of $\Omega$ may have corners or cusps, where it has no outward normal. The integration by parts on $\Omega$, on which the classical proofs rely, is then not available.

In this section, we show that a suitable kernel overcomes these two difficulties, without any regularity of $\partial \Omega$. Our three main results are summarized in \cref{sec:sbs-theory}. All the results of this section are formalized in Lean. In this section, $f$ denotes an element of $\H$, written $g$ in \cref{sec:sbs-method}. Thus, the function to minimize of \sbs is denoted by $V$ instead of $f$.

The proofs of these results rely on many intermediate results, which we develop in the rest of this section. In \cref{sec:svgd-setting}, we state the assumptions on the domain and on the target distribution. In \cref{sec:stein}, we recall Stein's method, and we prove Stein's identity for the vector fields that vanish on $\partial \Omega$. In \cref{sec:ksd}, we define the kernelized Stein discrepancy and the Stein witness. In \cref{sec:kl-descent}, we show that the Stein witness is the direction of steepest descent of the Kullback--Leibler divergence. We prove the existence of the continuous-time dynamics in \cref{sec:flow}, the decrease of the Kullback--Leibler divergence in \cref{sec:kl-derivative}, and the convergence in \cref{sec:convergence}. Finally, we apply these results to \sbs in \cref{sec:svgd-sbs}.

\subsection{Setting}\label{sec:svgd-setting}

In this section, we introduce the domain on which the measures live and the target distribution. We state the assumptions under which we work.

\inlinetitle{Notations}{.}~Let $d \geq 1$ be the dimension. The space $\R^d$ has the Euclidean norm $\norm{\cdot}$ and its Borel $\sigma$-algebra $\cal{B}(\R^d)$. Let $\Omega \subseteq \R^d$ be a compact set. Recall that $\cal{P}(\Omega)$ is the set of probability measures $\rho$ on $\R^d$ carried by $\Omega$, \ie such that $\rho(\R^d \setminus \Omega) = 0$. Note that a measure of $\cal{P}(\Omega)$ needs no density with respect to $\lambda$. For $\rho \in \cal{P}(\Omega)$ and a function $g$ defined on $\Omega$, we write $\bb{E}_{x \sim \rho}[g(x)] = \int_\Omega g \, \dif \rho$, when this integral exists. For $y \in \R^d$, let $\delta_y$ be the Dirac mass at $y$.

\svgd iteratively transforms an arbitrary measure $\rho$ into a target distribution $\pi$. We assume that $\pi$ has a density $p$ with respect to $\lambda|_{\Omega}$, and a score $s$ that plays the role of $\nabla \log p$.

\begin{assumption}[Target distribution]\label{ass:target}
  Let $p : \R^d \to \R$ and $s : \R^d \to (\R^d)^*$. The target distribution is $\pi = p \cdot \lambda|_{\Omega}$, \ie
  \[
    \forall A \in \cal{B}(\R^d), \quad \pi(A) = \int_{A \cap \Omega} p \, \dif \lambda.
  \]
  The following conditions hold.
  \begin{enumerate}[label=(\roman*)]
    \item $\pi$ is a probability measure.
    \item $p$ is measurable and continuous on $\Omega$, and $p(x) > 0$ for every $x \in \Omega$.
    \item For every $x \in \Omega^\circ$, $\log p$ is differentiable at $x$ and $D(\log p)(x) = s(x)$.
    \item $s$ is continuous on $\Omega$.
  \end{enumerate}
\end{assumption}

By (iii), $s = \nabla \log p$ on $\Omega^\circ$, once $(\R^d)^*$ is identified with $\R^d$. Thus, $s$ is the score of $\pi$ on the interior of $\Omega$, and it is only required to be continuous up to the boundary. By (i), $\Omega$ is nonempty. As $\Omega$ is compact, (ii) gives that $p$ reaches a positive minimum on $\Omega$. Thus, $\log p$ is continuous and bounded on $\Omega$. Likewise, by (iv), $M_s = \max_{x \in \Omega} \norm{s(x)}$ is finite. No regularity of the boundary $\partial \Omega$ is assumed.

\begin{example}\label{ex:boltzmann}
  Let $V : \R^d \to \R$ be continuously differentiable, let $\kappa > 0$, and assume that $\lambda(\Omega) > 0$. The Boltzmann distribution of $V$ on $\Omega$ has the density
  \[
    p(x) = \frac{e^{-\kappa V(x)}}{Z}, \quad \text{where} \quad Z = \int_\Omega e^{-\kappa V} \, \dif \lambda.
  \]
  It satisfies \cref{ass:target} with $s = -\kappa D V$. Indeed, $e^{-\kappa V}$ is positive and continuous on the compact set $\Omega$, which has a positive measure. Thus, $0 < Z < \infty$ and $\pi(\R^d) = 1$. Moreover, $\log p = -\kappa V - \log Z$ is differentiable on $\R^d$, and its derivative $-\kappa D V$ is continuous. A characteristic property of the Boltzmann distribution is that, as $\kappa$ tends to infinity, it tends to a distribution supported on the set of the minimizers of $V$. It is the target distribution of \sbs (see \cref{def:boltzmann-distribution} and \cref{sec:svgd-sbs}).
\end{example}

\subsection{Stein's method}\label{sec:stein}

\svgd is based on Stein's method, introduced by \cite{Stein1972}. In this section, we first recall the classical notions of this method. Then, we show that Stein's identity holds for the target distribution, without any assumption on $\partial \Omega$, for the vector fields that vanish on $\partial \Omega$. To do so, we compute the first variation of the Kullback--Leibler divergence, \ie its derivative along a small perturbation of the identity (\cref{thm:first-variation}). This result is also the basis of \svgd (\cref{sec:kl-descent}), and its proof is the template of \cref{sec:kl-derivative}. Finally, we define the Stein discrepancy.

\subsubsection{The Stein operator and Stein's identity}\label{sec:stein-operator}

Stein's method defines the Stein operator associated with a distribution.

\begin{definition}[Stein operator]\label{def:stein-operator}
  Let $E$ be a finite-dimensional real normed space and let $s : E \to E^*$. The Stein operator of score $s$ maps a vector field $v : E \to E$ to the function $\cal{A}_s v$, defined by
  \[
    \cal{A}_s v(x) = s(x)(v(x)) + \nabla \cdot v(x)
  \]
  at every $x \in E$ where $v$ is differentiable.
\end{definition}

When $E = \R^d$ and $s = \nabla \log p$ for a positive differentiable density $p$, we have that $\cal{A}_s v = \nabla \log p^\top v + \nabla \cdot v$. This is the Stein operator associated with $p$, as in \cite{Liu2016Kernel} and \cite{Liu2016}. It depends on $p$ only through its score $s$.

Stein's identity states that the expectation of $\cal{A}_s v$ under the distribution of density $p$ is zero \cite{Liu2016Kernel}. It holds for the vector fields $v$ of the Stein class of $p$, which we now describe. As $D p = p s$, we have that $\nabla \cdot (p v) = p \cal{A}_s v$. Thus, Stein's identity holds when the integral of $\nabla \cdot (p v)$ is zero. The Stein class of $p$ is the set of the regular vector fields $v$ that satisfy this condition. On $\Omega$, when $\partial \Omega$ and $p v$ are regular, the divergence theorem gives
\[
  \int_\Omega \nabla \cdot (p v) \, \dif \lambda = \int_{\partial \Omega} p v \cdot \vec{n} \, \dif S,
\]
where $\vec{n}$ is the outward unit normal and $S$ is the surface measure of $\partial \Omega$. Thus, this condition requires that $p v$ has no flux through $\partial \Omega$.

We now turn to the target distribution $\pi = p \cdot \lambda|_{\Omega}$ of \cref{ass:target}. Let $x \in \Omega^\circ$. As $\Omega^\circ$ is open and $p > 0$ on it, $p = \exp(\log p)$ on a neighborhood of $x$. By the chain rule and (iii) of \cref{ass:target}, $p$ is differentiable at $x$ and $D p(x) = p(x) s(x)$. The boundary $\partial \Omega$ may have a positive Lebesgue measure, and \cref{ass:target} says nothing about the derivative of $p$ on it. The divergence theorem also requires a regular boundary. The classical proof of Stein's identity for $\pi$ therefore needs a regular $\partial \Omega$. However, for the vector fields that vanish on $\partial \Omega$, this identity holds without any assumption on $\partial \Omega$, as we now show. The proof relies on the first variation of the Kullback--Leibler divergence.

\subsubsection{First variation of the Kullback--Leibler divergence}\label{sec:first-variation}

Let $\rho$ and $\rho'$ be probability measures on $\R^d$. The Kullback--Leibler divergence of $\rho$ with respect to $\rho'$ is
\[
  \KL(\rho \,\|\, \rho') = \int \log \frac{\dif \rho}{\dif \rho'} \, \dif \rho
\]
if $\rho \ll \rho'$ and $\log \left( \frac{\dif \rho}{\dif \rho'} \right)$ is $\rho$-integrable, and $\KL(\rho \,\|\, \rho') = \infty$ otherwise. It is nonnegative, and $\KL(\rho \,\|\, \rho) = 0$. Under \cref{ass:target}, if $\KL(\rho \,\|\, \pi) < \infty$, then $\rho \ll \pi$. As $\pi(\R^d \setminus \Omega) = 0$, we have that $\rho \in \cal{P}(\Omega)$.

A small perturbation of the identity by a vector field that vanishes on $\partial \Omega$ maps $\Omega$ onto itself (see \cref{fig:perturbation}). Let $I_d$ be the identity map of $\R^d$.

\begin{figure}[t]
  \centering
  \begin{subfigure}{0.49\textwidth}
    \centering
    \includegraphics[width=0.85\textwidth]{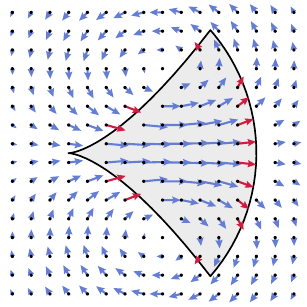}
    \caption{$v = u$}
  \end{subfigure}
  \hfill
  \begin{subfigure}{0.49\textwidth}
    \centering
    \includegraphics[width=0.85\textwidth]{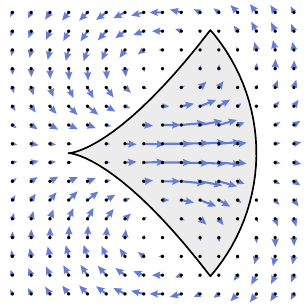}
    \caption{$v = \min(1, \text{dist}(\cdot, \partial \Omega) / \delta) u$}
  \end{subfigure}
  \caption{Transport of a grid of particles by $I_d + \eps v$, inside and outside $\Omega = \{(x, y) \in \R^2 \mid y^2 \leq x^3, x^2 + y^2 \leq 1\}$, where $d = 2$ and $\eps = 0.06$. The set $\Omega$ is in gray. Its boundary $\partial \Omega$ has a cusp at $0$ and two corners, where it has no outward normal. The field $u = (\partial_y \chi, -\partial_x \chi)$ has two counter-rotating vortices, with $\chi(z) = e^{-\norm{z - p_1}^2 / \sigma^2} - e^{-\norm{z - p_2}^2 / \sigma^2}$, $p_1 = (0.5, 0.3)$, $p_2 = (0.5, -0.3)$ and $\sigma = 0.6$. Each arrow goes from a particle $x$ to $x + \eps v(x)$. \textbf{a)}~The field $v = u$ does not vanish on $\partial \Omega$. The red particles cross $\partial \Omega$. \textbf{b)}~The field $v = \min(1, \text{dist}(\cdot, \partial \Omega) / \delta) u$, with $\delta = 0.2$, is equal to $u$ at distance at least $\delta$ from $\partial \Omega$, and vanishes on $\partial \Omega$. As $\eps K_v < 1$, each arrow is shorter than the distance from its particle to $\partial \Omega$, so no particle crosses $\partial \Omega$. This illustrates \cref{lem:perturbation}.}
  \label{fig:perturbation}
\end{figure}

\begin{lemma}[Perturbations of the identity]\label{lem:perturbation}
  Let $v : \R^d \to \R^d$ be $K_v$-Lipschitz, with $v(y) = 0$ for every $y \in \partial \Omega$. Let $\eps \in \R$ such that $\abs{\eps} K_v < 1$, and let $\Psi = I_d + \eps v$. Then, $\Psi$ is a homeomorphism of $\R^d$, and
  \[
    \Psi(\Omega) = \Omega \quad \text{and} \quad \Psi(\Omega^\circ) = \Omega^\circ.
  \]
\end{lemma}

\begin{proof}
  For all $x, y \in \R^d$, we have that
  \[
    \norm{\Psi(x) - \Psi(y)} \geq \norm{x - y} - \abs{\eps} \norm{v(x) - v(y)} \geq (1 - \abs{\eps} K_v) \norm{x - y}.
  \]
  Thus, $\Psi$ is injective, and its inverse is $(1 - \abs{\eps} K_v)^{-1}$-Lipschitz on $\Psi(\R^d)$. Let $y \in \R^d$. The map $x \mapsto y - \eps v(x)$ is a contraction of $\R^d$. By the Banach fixed point theorem, it has a fixed point $x$, \ie $\Psi(x) = y$. Thus, $\Psi$ is a homeomorphism of $\R^d$.

  The same holds for $\Psi_t = I_d + t \eps v$, for every $t \in [0, 1]$. Let $x \in \R^d \setminus \partial \Omega$ and $t \in [0, 1]$. Let $z \in \partial \Omega$. As $v(z) = 0$, we have that $\Psi_t(z) = z$. As $x \neq z$ and $\Psi_t$ is injective, $\Psi_t(x) \neq \Psi_t(z) = z$. Thus, $\Psi_t(x) \notin \partial \Omega$. As $t$ is arbitrary, the continuous path $\gamma : t \mapsto \Psi_t(x) = x + t \eps v(x)$ lies in $\R^d \setminus \partial \Omega$. As $\Omega$ is closed, this set is the disjoint union of the open sets $\Omega^\circ$ and $\R^d \setminus \Omega$. Thus, $\gamma^{-1}(\Omega^\circ)$ and $\gamma^{-1}(\R^d \setminus \Omega)$ are disjoint open subsets of $[0, 1]$, whose union is $[0, 1]$. As $[0, 1]$ is connected, one of them is empty, \ie $\gamma$ stays in $\Omega^\circ$ or stays in $\R^d \setminus \Omega$. In particular, $\Psi(x) = \gamma(1)$ lies in $\Omega^\circ$ if and only if $x = \gamma(0)$ does. Hence, $\Psi$ maps each of the three disjoint sets $\Omega^\circ$, $\partial \Omega$ and $\R^d \setminus \Omega$ into itself. As these sets cover $\R^d$ and $\Psi$ is a bijection, these inclusions are equalities. Thus, $\Psi(\Omega^\circ) = \Omega^\circ$ and $\Psi(\Omega) = \Psi(\Omega^\circ \cup \partial \Omega) = \Omega$, which concludes the proof.
\end{proof}

The next lemma expresses the Kullback--Leibler divergence of a transported measure with respect to $\pi$. For a map $\Psi$ differentiable at $x$, let $\det D \Psi(x)$ be the Jacobian determinant of $\Psi$ at $x$.

\begin{lemma}[Transport of the Kullback--Leibler divergence]\label{lem:transport}
  Suppose that \cref{ass:target} holds. Let $\Psi$ be a homeomorphism of $\R^d$, differentiable at every point of $\Omega$, such that $\Psi(\Omega) = \Omega$ and $\det D \Psi(x) > 0$ for every $x \in \Omega$. Let $\rho$ be a probability measure on $\R^d$ such that $\KL(\rho \,\|\, \pi) < \infty$, and such that $\log \det D \Psi$ is $\rho$-integrable. Then, $\KL(\Psi_\# \rho \,\|\, \pi) < \infty$ and
  \[
    \KL(\Psi_\# \rho \,\|\, \pi) = \KL(\rho \,\|\, \pi) - \int (\log p(\Psi(x)) - \log p(x) + \log \det D \Psi(x)) \, \dif \rho(x).
  \]
\end{lemma}

\begin{proof}
  The Kullback--Leibler divergence is invariant under the bijection $\Psi$. Indeed, let $\rho'$ be a probability measure on $\R^d$. As $\Psi$ and $\Psi^{-1}$ are measurable, $\rho \ll \rho'$ if and only if $\Psi_\# \rho \ll \Psi_\# \rho'$, and then $\frac{\dif \Psi_\# \rho}{\dif \Psi_\# \rho'} = \frac{\dif \rho}{\dif \rho'} \circ \Psi^{-1}$. Thus, $\KL(\Psi_\# \rho \,\|\, \Psi_\# \rho') = \KL(\rho \,\|\, \rho')$. With $\rho' = \Psi^{-1}_\# \pi$, we obtain
  \[
    \KL(\Psi_\# \rho \,\|\, \pi) = \KL(\rho \,\|\, \Psi^{-1}_\# \pi).
  \]
  Let $q(x) = p(\Psi(x)) \det D \Psi(x) / p(x)$ for $x \in \Omega$, and $q(x) = 1$ otherwise. Let $A \in \cal{B}(\R^d)$. As $\Psi$ is a bijection with $\Psi(\Omega) = \Omega$, we have that $\Psi(A) \cap \Omega = \Psi(A \cap \Omega)$. The map $\Psi$ is injective and differentiable at every point of the measurable set $A \cap \Omega$. This is all that the change of variables formula of \cite[Section 263]{Fremlin2010} requires. As $\det D \Psi > 0$ on $\Omega$, it gives
  \[
    \Psi^{-1}_\# \pi(A) = \pi(\Psi(A)) = \int_{\Psi(A \cap \Omega)} p \, \dif \lambda = \int_{A \cap \Omega} p(\Psi(x)) \det D \Psi(x) \, \dif \lambda(x) = (q \cdot \pi)(A).
  \]
  Thus, $\Psi^{-1}_\# \pi = q \cdot \pi$, which is a probability measure. As $\Psi(\Omega) = \Omega$ and $p > 0$ on $\Omega$, $q > 0$ on $\Omega$. Thus, $\pi \ll q \cdot \pi$ and $\frac{\dif \pi}{\dif (q \cdot \pi)} = 1 / q$, $\pi$-almost everywhere. As $\rho \ll \pi$, we have that $\rho \ll q \cdot \pi$ and
  \[
    \log \frac{\dif \rho}{\dif (q \cdot \pi)} = \log \frac{\dif \rho}{\dif \pi} - \log q, \quad \rho\text{-almost everywhere}.
  \]
  As $\rho \in \cal{P}(\Omega)$, we have that $\log q(x) = \log p(\Psi(x)) - \log p(x) + \log \det D \Psi(x)$ for $\rho$-almost every $x$. The functions $\log p \circ \Psi$ and $\log p$ are bounded on $\Omega$, as $p$ reaches a positive minimum and a maximum on the compact set $\Omega$, and $\Psi(\Omega) = \Omega$. Thus, $\log q$ is $\rho$-integrable. As $\log \left( \frac{\dif \rho}{\dif \pi} \right)$ is $\rho$-integrable, so is $\log \left( \frac{\dif \rho}{\dif (q \cdot \pi)} \right)$. Therefore,
  \[
    \KL(\Psi_\# \rho \,\|\, \pi) = \KL(\rho \,\|\, q \cdot \pi) = \KL(\rho \,\|\, \pi) - \int \log q \, \dif \rho < \infty,
  \]
  which concludes the proof.
\end{proof}

We can now differentiate the Kullback--Leibler divergence along a perturbation of the identity. Its derivative at the identity is called the \emph{first variation} of the Kullback--Leibler divergence.

\begin{theorem}[First variation of the Kullback--Leibler divergence]\label{thm:first-variation}
  Suppose that \cref{ass:target} holds. Let $v : \R^d \to \R^d$ be continuously differentiable and $K_v$-Lipschitz, with $v(y) = 0$ for every $y \in \partial \Omega$. Let $\rho$ be a probability measure on $\R^d$ such that $\KL(\rho \,\|\, \pi) < \infty$. Then, $\cal{A}_s v$ is $\rho$-integrable, and there exists $\eps_0 > 0$ such that
  \[
    \forall \eps \in (-\eps_0, \eps_0), \quad \KL((I_d + \eps v)_\# \rho \,\|\, \pi) < \infty.
  \]
  Moreover, the function $h : \eps \mapsto \KL((I_d + \eps v)_\# \rho \,\|\, \pi)$ is differentiable at $0$, and
  \[
    h'(0) = -\bb{E}_{x \sim \rho}[\cal{A}_s v(x)].
  \]
\end{theorem}

\begin{proof}
  As $\KL(\rho \,\|\, \pi) < \infty$, $\rho \in \cal{P}(\Omega)$. The function $\cal{A}_s v : x \mapsto s(x)(v(x)) + \nabla \cdot v(x)$ is continuous on the compact set $\Omega$, which carries $\rho$. Thus, it is $\rho$-integrable.

  \inlinetitle{Transport}{.}~As $D v$ is continuous and $\Omega$ is compact, $D v(\Omega)$ is compact. The map $(\tau, A) \mapsto \det(I_d + \tau A)$ is continuous and equal to $1$ at $\tau = 0$. Thus, there exists $\eps_0 \in (0, 1]$ such that $\eps_0 K_v < 1$ and
  \[
    \forall \tau \in [-\eps_0, \eps_0], \forall A \in D v(\Omega), \quad \det(I_d + \tau A) \geq \frac{1}{2}.
  \]
  Let $\eps \in (-\eps_0, \eps_0)$ and $\Psi_\eps = I_d + \eps v$. By \cref{lem:perturbation}, $\Psi_\eps$ is a homeomorphism of $\R^d$ such that $\Psi_\eps(\Omega) = \Omega$ and $\Psi_\eps(\Omega^\circ) = \Omega^\circ$. It is differentiable, with $D \Psi_\eps(x) = I_d + \eps D v(x)$, so that $\det D \Psi_\eps \geq 1 / 2$ on $\Omega$. The function $\log \det D \Psi_\eps$ is continuous on the compact set $\Omega$, which carries $\rho$. Thus, it is $\rho$-integrable, and \cref{lem:transport} gives $\KL((\Psi_\eps)_\# \rho \,\|\, \pi) < \infty$ and
  \[
    h(\eps) = \KL(\rho \,\|\, \pi) - \int \ell(\eps, x) \, \dif \rho(x),
  \]
  where $\ell(\eps, x) = \log p(x + \eps v(x)) - \log p(x) + \log \det(I_d + \eps D v(x))$.

  \inlinetitle{Differentiation}{.}~Let $x \in \Omega$, which holds for $\rho$-almost every $x$, and let $\eps \in (-\eps_0, \eps_0)$. If $x \in \Omega^\circ$, then $x + \eps v(x) \in \Omega^\circ$ by \cref{lem:perturbation}. As $\log p$ is differentiable at $x + \eps v(x)$, with derivative $s(x + \eps v(x))$, the chain rule gives
  \[
    \partial_\eps \log p(x + \eps v(x)) = s(x + \eps v(x))(v(x)).
  \]
  Otherwise, $x \in \partial \Omega$, so that $v(x) = 0$ and both sides vanish. Let $D \det(Y)(Z)$ be the derivative of the determinant at the matrix $Y$ in the direction $Z$. As $\det(I_d + \eps D v(x)) \geq 1 / 2 > 0$, we obtain
  \[
    \partial_\eps \ell(\eps, x) = s(x + \eps v(x))(v(x)) + \frac{D \det(I_d + \eps D v(x))(D v(x))}{\det(I_d + \eps D v(x))}.
  \]
  Let $M_v$ be the maximum of $\norm{v}$ on $\Omega$, and let $M$ be the maximum of $\abs{D \det(I_d + \tau A)(A)}$ for $(\tau, A) \in [-1, 1] \times D v(\Omega)$. They are finite by continuity on compact sets. As $x + \eps v(x) \in \Omega$, we have that $\abs{\partial_\eps \ell(\eps, x)} \leq M_s M_v + 2 M$. This constant is $\rho$-integrable, as $\rho$ is a probability measure. Moreover, $\ell(0, \cdot) = 0$, and every $\ell(\eps, \cdot)$ is continuous on $\Omega$, which carries $\rho$. By the Leibniz integral rule \cite[Theorem 2.27]{Folland1999}, $\eps \mapsto \int \ell(\eps, x) \, \dif \rho(x)$ is differentiable at $0$, with derivative $\int \partial_\eps \ell(0, x) \, \dif \rho(x)$. Finally, Jacobi's formula at the identity, $D \det(I_d)(A) = \operatorname{tr} A$, gives
  \[
    \partial_\eps \ell(0, x) = s(x)(v(x)) + \operatorname{tr} D v(x) = \cal{A}_s v(x),
  \]
  which concludes the proof.
\end{proof}

Stein's identity for the target distribution follows, as the Kullback--Leibler divergence is minimal at $\pi$.

\begin{corollary}[Stein's identity for the target distribution]\label{cor:stein-identity}
  Suppose that \cref{ass:target} holds. Let $v : \R^d \to \R^d$ be continuously differentiable and Lipschitz, with $v(y) = 0$ for every $y \in \partial \Omega$. Then, $\cal{A}_s v$ is $\pi$-integrable and
  \[
    \bb{E}_{x \sim \pi}[\cal{A}_s v(x)] = 0.
  \]
\end{corollary}

\begin{proof}
  As $\KL(\pi \,\|\, \pi) = 0 < \infty$, \cref{thm:first-variation} applies to $\rho = \pi$. Thus, $\cal{A}_s v$ is $\pi$-integrable, and $h : \eps \mapsto \KL((I_d + \eps v)_\# \pi \,\|\, \pi)$ is real-valued on $(-\eps_0, \eps_0)$ and differentiable at $0$, with $h'(0) = -\bb{E}_{x \sim \pi}[\cal{A}_s v(x)]$. As the Kullback--Leibler divergence is nonnegative, $h \geq 0 = h(0)$. Thus, $h$ has a minimum at $0$, and $h'(0) = 0$ by Fermat's theorem on stationary points, which concludes the proof.
\end{proof}

Thus, every continuously differentiable and Lipschitz vector field that vanishes on $\partial \Omega$ satisfies Stein's identity for $\pi$, without the divergence theorem and without any regularity of $\partial \Omega$. The boundary is handled by the vector field itself, which vanishes on $\partial \Omega$.

\subsubsection{The Stein discrepancy}\label{sec:stein-discrepancy}

Let $\cal{F}$ be a set of continuously differentiable and Lipschitz vector fields that vanish on $\partial \Omega$. By \cref{cor:stein-identity}, $\bb{E}_{x \sim \pi}[\cal{A}_s v(x)] = 0$ for every $v \in \cal{F}$. Now, one can consider this expectation for another measure $\rho \in \cal{P}(\Omega)$. If $\rho \neq \pi$, it is no longer zero in general. In fact, its magnitude relates to how different $\rho$ and $\pi$ are. It is used to define a discrepancy measure, known as the Stein discrepancy \cite{Gorham2015}. The latter considers the maximum violation of Stein's identity over $\cal{F}$, \ie
\[
  \bb{S}(\rho, \pi) = \sup_{v \in \cal{F}} \bb{E}_{x \sim \rho}[\cal{A}_s v(x)].
\]
Note that $\bb{S}(\rho, \pi)$ is not symmetric. The choice of $\cal{F}$ determines the discriminative power and the tractability of the Stein discrepancy. In \cref{sec:ksd}, $\cal{F}$ is the unit ball of a space of vector fields defined by a kernel, as in \cite{Liu2016Kernel}. This choice gives a closed-form expression of the square of the supremum, known as the \emph{kernelized Stein discrepancy}.

\subsection{The kernelized Stein discrepancy}\label{sec:ksd}

In this section, we define the kernel and the space of vector fields of \svgd, in a setting that contains that of \cite{Liu2016}. Then, we define the kernelized Stein discrepancy of \cite{Liu2016Kernel} in this setting, and we deduce its closed-form expression through the Stein witness.

\subsubsection{The kernel}\label{sec:kernel-assumption}

In \cite{Liu2016}, the vector fields of \svgd lie in a product of reproducing kernel Hilbert spaces (RKHS). Here, the kernel is only given by a feature map $\Phi$, with values in a real Hilbert space $H_0$. The space $H_0$ need not be an RKHS. This setting only uses the chain rule to differentiate the vector fields of $\cal{H}$ (see below), whereas the RKHS setting requires a reproducing property for the derivatives.

\begin{assumption}[Kernel]\label{ass:kernel}
  Let $H_0$ be a real Hilbert space, let $\Phi : \R^d \to H_0$ and let $K, C \geq 0$. The kernel is
  \[
    k(x, y) = \ket{\Phi(x), \Phi(y)}_{H_0}, \quad x, y \in \R^d.
  \]
  The following conditions hold.
  \begin{enumerate}[label=(\roman*)]
    \item $\Phi$ is continuously differentiable on $\R^d$.
    \item $\Phi$ is $K$-Lipschitz, \ie $\norm{\Phi(x) - \Phi(y)}_{H_0} \leq K \norm{x - y}$ for all $x, y \in \R^d$.
    \item $\norm{\Phi(y)}_{H_0} \leq C$ for every $y \in \R^d$.
    \item $\Phi(y) = 0$ for every $y \in \partial \Omega$.
  \end{enumerate}
\end{assumption}

Let $\cal{H} = H_0^d$ be the space of the $d$-tuples $f = (f_1, \dots, f_d)$ of elements of $H_0$, with the inner product
\[
  \ket{f, g}_{\cal{H}} = \sum_{i = 1}^d \ket{f_i, g_i}_{H_0}.
\]
It is a Hilbert space. The elements of $H_0$ are abstract vectors, and need not be functions on $\R^d$. Thus, an element $f \in \cal{H}$ is not a vector field by itself. It defines one through the feature map, as each component $f_i$ is paired with the feature $\Phi(y)$ of the point $y$. The vector field $f_\Phi : \R^d \to \R^d$ of $f$ is defined by
\[
  f_\Phi(y) = (\ket{f_1, \Phi(y)}_{H_0}, \dots, \ket{f_d, \Phi(y)}_{H_0}).
\]
In other words, each coordinate of $f_\Phi$ is a linear model in the features, whose coefficients are given by $f_i$. The subscript $\Phi$ distinguishes the element $f$ of $\cal{H}$ from its vector field $f_\Phi$ on $\R^d$.

\begin{example}
  Let $d = 1$, $H_0 = \R^3$ and $\Phi(y) = (1, y, y^2)$. An element $f = (a, b, c)$ of $\cal{H} = \R^3$ is a triple of coefficients. Its vector field is the polynomial $f_\Phi(y) = a + b y + c y^2$, and the kernel is $k(x, y) = 1 + x y + x^2 y^2$. This feature map only illustrates the notation, as it does not satisfy \cref{ass:kernel}.
\end{example}

This setting contains the product RKHS of \cite{Liu2016}, as a special case. Let $k_0$ be a kernel on $\R^d$, let $H_0$ be its RKHS \cite{Aronszajn1950}, and let $\Phi(y) = k_0(\cdot, y)$. By the reproducing property, $\ket{\Phi(x), \Phi(y)}_{H_0} = k_0(x, y)$ and $f_\Phi(y) = (f_1(y), \dots, f_d(y))$ for every $f \in \cal{H}$. Thus, $f_\Phi = f$, \ie the subscript $\Phi$ only evaluates the functions $f_i$ at $y$. In this case, \cref{ass:kernel} is a condition on the map $y \mapsto k_0(\cdot, y)$. Conversely, a feature map defines the kernel $k$, but $H_0$ need not be the RKHS of $k$. Thus, $\cal{H}$ is not a product RKHS in general. No strict positive definiteness of $k$ is required either, such as the integral strict positive definiteness used by \cite{Liu2016Kernel}.

In general, we have that $f_\Phi = P_f \circ \Phi$, where $P_f : u \mapsto (\ket{f_i, u}_{H_0})_{1 \leq i \leq d}$ is a linear map from $H_0$ to $\R^d$. By the Cauchy--Schwarz inequality in $H_0$, we have that
\[
  \forall u \in H_0, \quad \norm{P_f(u)}^2 = \sum_{i = 1}^d \ket{f_i, u}_{H_0}^2 \leq \sum_{i = 1}^d \norm{f_i}_{H_0}^2 \norm{u}_{H_0}^2 = \norm{f}_{\cal{H}}^2 \norm{u}_{H_0}^2.
\]
Taking square roots, $\norm{P_f(u)} \leq \norm{f}_{\cal{H}} \norm{u}_{H_0}$ for every $u \in H_0$. Thus, the linear map $P_f$ is continuous, and its operator norm is at most $\norm{f}_{\cal{H}}$. By (i) of \cref{ass:kernel} and the chain rule, $f_\Phi$ is continuously differentiable, with $D f_\Phi(y) = P_f \circ D \Phi(y)$. As $P_f$ is linear, we have that
\[
  \forall x, y \in \R^d, \quad \norm{f_\Phi(x) - f_\Phi(y)} = \norm{P_f(\Phi(x) - \Phi(y))} \leq \norm{f}_{\cal{H}} \norm{\Phi(x) - \Phi(y)}_{H_0}.
\]
Thus, by (ii) of \cref{ass:kernel}, $f_\Phi$ is $\norm{f}_{\cal{H}} K$-Lipschitz. Likewise, $\norm{f_\Phi(y)} \leq \norm{f}_{\cal{H}} \norm{\Phi(y)}_{H_0}$, so that $f_\Phi$ is bounded by $\norm{f}_{\cal{H}} C$ by (iii), and zero on $\partial \Omega$ by (iv).

Conditions (i) and (ii) of \cref{ass:kernel} make every vector field $f_\Phi$ continuously differentiable and Lipschitz. Thus, \cref{thm:first-variation} and \cref{cor:stein-identity} apply to $f_\Phi$. In \cref{sec:flow}, these conditions also give global solutions to the ordinary differential equations driven by these fields, by the Cauchy--Lipschitz theorem. Condition (iii) bounds $f_\Phi$ uniformly on $\R^d$. Condition (iv) makes every $f_\Phi$ vanish on $\partial \Omega$. By \cref{lem:perturbation}, $I_d + \eps f_\Phi$ then maps $\Omega$ onto $\Omega$ when $\abs{\eps} \norm{f}_{\cal{H}} K < 1$ (see \cref{fig:perturbation}). This requires no regularity of $\partial \Omega$.

Translation-invariant kernels, such as the Gaussian kernel, do not satisfy condition (iv) (see \cref{fig:stein-witness}). Indeed, for such a kernel, $k(y, y) = \norm{\Phi(y)}_{H_0}^2$ is a constant, which is positive unless $k = 0$. Weighted kernels $k(x, y) = w(x) w(y) k_0(x, y)$, where $w : \R^d \to \R$ vanishes on $\partial \Omega$, satisfy it. Their feature map is $\Phi = w \Phi_0$, where $\Phi_0$ is a feature map of $k_0$.

In the rest of this section, \cref{ass:target} and \cref{ass:kernel} hold, unless stated otherwise. We call them the \emph{standing assumptions}.

\subsubsection{Definition of the kernelized Stein discrepancy}\label{sec:ksd-definition}

The kernelized Stein discrepancy is the Stein discrepancy of \cref{sec:stein-discrepancy}, with the vector fields $f_\Phi$ of the closed unit ball of $\cal{H}$.

\begin{definition}[Kernelized Stein discrepancy]\label{def:ksd}
  Suppose that the standing assumptions hold. Let $\rho \in \cal{P}(\Omega)$. The kernelized Stein discrepancy (KSD) of $\rho$ with respect to $\pi$ is
  \[
    \K(\rho | \pi) = \sup_{f \in \cal{H}, \, \norm{f}_{\cal{H}} \leq 1} \bb{E}_{x \sim \rho}[\cal{A}_s f_\Phi(x)]^2.
  \]
\end{definition}

Let $f \in \cal{H}$. As $s$ is continuous on $\Omega$ and $f_\Phi$ is continuously differentiable, $\cal{A}_s f_\Phi : x \mapsto s(x)(f_\Phi(x)) + \nabla \cdot f_\Phi(x)$ is continuous on $\Omega$. Recall that $M_s = \max_{x \in \Omega} \norm{s(x)}$ (see \cref{sec:svgd-setting}). As $f_\Phi$ is $\norm{f}_{\cal{H}} K$-Lipschitz, $\norm{D f_\Phi(x)} \leq \norm{f}_{\cal{H}} K$. The trace of a linear map $L$ of $\R^d$ satisfies $\abs{\operatorname{tr} L} \leq d \norm{L}$. Thus,
\[
  \forall x \in \Omega, \quad \abs{\cal{A}_s f_\Phi(x)} \leq \norm{s(x)} \norm{f_\Phi(x)} + d \norm{D f_\Phi(x)} \leq \norm{f}_{\cal{H}} (M_s C + d K).
\]
As $\rho$ is carried by $\Omega$, every expectation in \cref{def:ksd} is finite, and $0 \leq \K(\rho | \pi) \leq (M_s C + d K)^2$. The KSD depends on $\pi$ only through its score $s$. Let $\cal{F}$ be the set of the vector fields $f_\Phi$ with $\norm{f}_{\cal{H}} \leq 1$. They are continuously differentiable and Lipschitz, and they vanish on $\partial \Omega$. As the closed unit ball of $\cal{H}$ is symmetric and $f \mapsto \cal{A}_s f_\Phi$ is linear, we have that $\K(\rho | \pi) = \bb{S}(\rho, \pi)^2$. In particular, $\K(\pi | \pi) = 0$ by \cref{cor:stein-identity}.

The construction of $\cal{H}$ is motivated by the fact that the supremum in \cref{def:ksd} has a closed-form expression. This supremum is reached at an explicit element of $\cal{H}$, which we compute in \cref{sec:stein-witness}.

\subsubsection{The Stein witness}\label{sec:stein-witness}

To find this direction, we write the Stein operator of $f_\Phi$ as an inner product in $\cal{H}$. Let $(e_1, \dots, e_d)$ be the canonical basis of $\R^d$. For $x \in \R^d$, let $\xi_x \in \cal{H}$ be defined by
\[
  (\xi_x)_i = s(x)(e_i) \Phi(x) + D \Phi(x)(e_i), \quad 1 \leq i \leq d.
\]
We call $\xi_x$ the \emph{Stein feature} of $x$. The next lemma links $\xi_x$ to the Stein operator of $f_\Phi$. It is the key step of the closed-form expression of the KSD.

\begin{lemma}[Stein feature]\label{lem:stein-feature}
  Suppose that the standing assumptions hold. Then,
  \[
    \forall x \in \R^d, \forall f \in \cal{H}, \quad \cal{A}_s f_\Phi(x) = \ket{f, \xi_x}_{\cal{H}}.
  \]
\end{lemma}

\begin{proof}
  Let $x \in \R^d$ and $f \in \cal{H}$. By (i) of \cref{ass:kernel}, $\Phi$ is differentiable at $x$. As $P_f$ is linear and continuous, the chain rule gives that $f_\Phi = P_f \circ \Phi$ is differentiable at $x$, with $D f_\Phi(x) = P_f \circ D \Phi(x)$. The divergence is the trace of the derivative. As $e_i^\top P_f(u) = \ket{f_i, u}_{H_0}$ for every $u \in H_0$, we have that
  \[
    \nabla \cdot f_\Phi(x) = \sum_{i = 1}^d e_i^\top D f_\Phi(x)(e_i) = \sum_{i = 1}^d e_i^\top P_f(D \Phi(x)(e_i)) = \sum_{i = 1}^d \ket{f_i, D \Phi(x)(e_i)}_{H_0}.
  \]
  As $f_\Phi(x) = \sum_{i = 1}^d \ket{f_i, \Phi(x)}_{H_0} e_i$ and $s(x)$ is linear, we have that
  \begin{align*}
    \cal{A}_s f_\Phi(x) & = s(x)(f_\Phi(x)) + \nabla \cdot f_\Phi(x) \\
    & = \sum_{i = 1}^d (\ket{f_i, \Phi(x)}_{H_0} s(x)(e_i) + \ket{f_i, D \Phi(x)(e_i)}_{H_0}) \\
    & = \sum_{i = 1}^d \ket{f_i, s(x)(e_i) \Phi(x) + D \Phi(x)(e_i)}_{H_0} \\
    & = \ket{f, \xi_x}_{\cal{H}},
  \end{align*}
  which concludes the proof.
\end{proof}

\begin{remark}
  In the RKHS case, \ie when $\Phi(x) = k_0(\cdot, x)$, we have that $f_\Phi = f$ for every $f \in \cal{H}$, and the Stein feature is $\xi_x = \nabla \log p(x) k_0(\cdot, x) + \nabla_x k_0(\cdot, x)$ for $x \in \Omega^\circ$, as in \cite{Liu2016Kernel} and \cite{Liu2016}. Thus, \cref{lem:stein-feature} reads
  \[
    \forall x \in \R^d, \forall f \in \cal{H}, \quad \cal{A}_s f(x) = \ket{f, \xi_x}_{\cal{H}}.
  \]
\end{remark}

The next lemma gives the regularity of $\xi$ on $\Omega$.

\begin{lemma}[Regularity of $\xi$]\label{lem:xi-regularity}
  Suppose that the standing assumptions hold. Then, the following holds.
  \begin{enumerate}[label=(\roman*)]
    \item The map $x \mapsto \xi_x$ is uniformly continuous from $\Omega$ to $\cal{H}$.
    \item $B = \max_{x \in \Omega} \norm{\xi_x}_{\cal{H}}$ is finite.
    \item For every $\rho \in \cal{P}(\Omega)$, the map $x \mapsto \xi_x$ is $\rho$-integrable.
  \end{enumerate}
\end{lemma}

\begin{proof}
  By (iv) of \cref{ass:target}, $s$ is continuous on $\Omega$, and so is $x \mapsto s(x)(e_i)$. By (i) of \cref{ass:kernel}, $\Phi$ and $x \mapsto D \Phi(x)$ are continuous, and so is $x \mapsto D \Phi(x)(e_i)$. Thus, every component $x \mapsto (\xi_x)_i$ is continuous on $\Omega$. As $\norm{f}_{\cal{H}}^2 = \sum_{i = 1}^d \norm{f_i}_{H_0}^2$, the map $x \mapsto \xi_x$ is continuous from $\Omega$ to $\cal{H}$. As $\Omega$ is compact, the Heine--Cantor theorem gives (i). As $\Omega$ is nonempty and compact, the continuous function $x \mapsto \norm{\xi_x}_{\cal{H}}$ reaches its maximum on $\Omega$, which gives (ii). Finally, $x \mapsto \xi_x$ is $\rho$-integrable, as it is continuous and bounded by $B$ on $\Omega$, which carries the probability measure $\rho$. Here, the integral of a function with values in $\cal{H}$ is the Bochner integral, which extends the expectation to functions with values in a complete normed space. As $\xi$ is continuous on the compact set $\Omega$, its image is compact, hence separable, so that $\xi$ is measurable in the sense of this integral.
\end{proof}

We can now deduce this direction, and the closed-form expression of the KSD.

\begin{theorem}[Stein witness]\label{thm:stein-witness}
  Suppose that the standing assumptions hold. Let $\rho \in \cal{P}(\Omega)$. Then, there exists a unique $\phi^*_\rho \in \cal{H}$ such that
  \[
    \forall f \in \cal{H}, \quad \bb{E}_{x \sim \rho}[\cal{A}_s f_\Phi(x)] = \ket{f, \phi^*_\rho}_{\cal{H}}.
  \]
  It is the Bochner integral $\phi^*_\rho = \bb{E}_{x \sim \rho}[\xi_x]$. Moreover, the following holds.
  \begin{enumerate}[label=(\roman*)]
    \item On the closed unit ball of $\cal{H}$, the function $f \mapsto \bb{E}_{x \sim \rho}[\cal{A}_s f_\Phi(x)]$ reaches its maximum at $f^* = \phi^*_\rho / \norm{\phi^*_\rho}_{\cal{H}}$ if $\phi^*_\rho \neq 0$, and at $f^* = 0$ otherwise.
    \item $\K(\rho | \pi) = \norm{\phi^*_\rho}_{\cal{H}}^2 = \bb{E}_{x \sim \rho}[\cal{A}_s (\phi^*_\rho)_\Phi(x)]$.
  \end{enumerate}
\end{theorem}

The proof strategy is to remark that, for any $f \in \cal{H}$, $\bb{E}_{x \sim \rho}[\cal{A}_s f_\Phi(x)] = \ket{f, \bb{E}_{x \sim \rho}[\xi_x]}_{\cal{H}}$. Then, the result follows from the Cauchy--Schwarz inequality.

\begin{proof}
  By (iii) of \cref{lem:xi-regularity}, $x \mapsto \xi_x$ is $\rho$-integrable. As $\cal{H}$ is a Hilbert space, it is complete, so that the Bochner integral $\phi^*_\rho = \bb{E}_{x \sim \rho}[\xi_x]$ is well defined. Let $f \in \cal{H}$. The continuous linear form $\ket{f, \cdot}_{\cal{H}}$ commutes with the Bochner integral. By \cref{lem:stein-feature}, we have that
  \[
    \ket{f, \phi^*_\rho}_{\cal{H}} = \bb{E}_{x \sim \rho}[\ket{f, \xi_x}_{\cal{H}}] = \bb{E}_{x \sim \rho}[\cal{A}_s f_\Phi(x)].
  \]
  This gives the existence of $\phi^*_\rho$. We now show that it is unique. Let $f' \in \cal{H}$ such that
  \[
    \forall f \in \cal{H}, \quad \bb{E}_{x \sim \rho}[\cal{A}_s f_\Phi(x)] = \ket{f, f'}_{\cal{H}}.
  \]
  Let $f \in \cal{H}$. By the assumption on $f'$ and the identity above, we have that
  \[
    \ket{f, f'}_{\cal{H}} = \bb{E}_{x \sim \rho}[\cal{A}_s f_\Phi(x)] = \ket{f, \phi^*_\rho}_{\cal{H}}.
  \]
  As the inner product is linear in its second argument, $\ket{f, f' - \phi^*_\rho}_{\cal{H}} = \ket{f, f'}_{\cal{H}} - \ket{f, \phi^*_\rho}_{\cal{H}} = 0$.
  This holds for every $f \in \cal{H}$, in particular for $f = f' - \phi^*_\rho$. Thus,
  \[
    \norm{f' - \phi^*_\rho}_{\cal{H}}^2 = \ket{f' - \phi^*_\rho, f' - \phi^*_\rho}_{\cal{H}} = 0.
  \]
  As a norm vanishes only at $0$, we have that $f' - \phi^*_\rho = 0$, \ie $f' = \phi^*_\rho$.

  Let $f \in \cal{H}$ such that $\norm{f}_{\cal{H}} \leq 1$. By the Cauchy--Schwarz inequality, we have that
  \[
    \ket{f, \phi^*_\rho}_{\cal{H}} \leq \abs{\ket{f, \phi^*_\rho}_{\cal{H}}} \leq \norm{f}_{\cal{H}} \norm{\phi^*_\rho}_{\cal{H}} \leq \norm{\phi^*_\rho}_{\cal{H}}.
  \]
  Moreover, $\norm{f^*}_{\cal{H}} \leq 1$ and $\ket{f^*, \phi^*_\rho}_{\cal{H}} = \norm{\phi^*_\rho}_{\cal{H}}$, in both cases. This gives (i). We also have that
  \[
    \K(\rho | \pi) = \sup_{f \in \cal{H}, \, \norm{f}_{\cal{H}} \leq 1} \ket{f, \phi^*_\rho}_{\cal{H}}^2.
  \]
  Every term is at most $\norm{\phi^*_\rho}_{\cal{H}}^2$, and the term of $f^*$ is equal to it, so that $\K(\rho | \pi) = \norm{\phi^*_\rho}_{\cal{H}}^2$. Finally, the identity above with $f = \phi^*_\rho$ gives
  \[
    \norm{\phi^*_\rho}_{\cal{H}}^2 = \ket{\phi^*_\rho, \phi^*_\rho}_{\cal{H}} = \bb{E}_{x \sim \rho}[\cal{A}_s (\phi^*_\rho)_\Phi(x)].
  \]
  This gives (ii), which concludes the proof.
\end{proof}

We call $\phi^*_\rho$ the \emph{Stein witness} of $\rho$, by analogy with the witness function of the maximum mean discrepancy \cite{Gretton2012}. The latter is the function at which the supremum that defines this discrepancy is reached. Up to normalization, it is the direction of $\cal{H}$ along which $\rho$ violates Stein's identity the most. In \cref{sec:kl-descent}, we show that it is also the direction of steepest descent of the Kullback--Leibler divergence. As $\norm{\xi_x}_{\cal{H}} \leq B$ for $\rho$-almost every $x$, we have that $\norm{\phi^*_\rho}_{\cal{H}} \leq B$. By (iii) of \cref{ass:kernel} and the bound on $P_f$, we have that
\[
  \forall y \in \R^d, \quad \norm{(\phi^*_\rho)_\Phi(y)} \leq \norm{\phi^*_\rho}_{\cal{H}} \norm{\Phi(y)}_{H_0} \leq B C.
\]
The vector field of $\phi^*_\rho$ is thus bounded by a constant that does not depend on $\rho$. \cref{fig:stein-witness} shows this vector field in dimension $1$.

\begin{figure}[ht]
  \centering
  \includegraphics[width=0.8\textwidth]{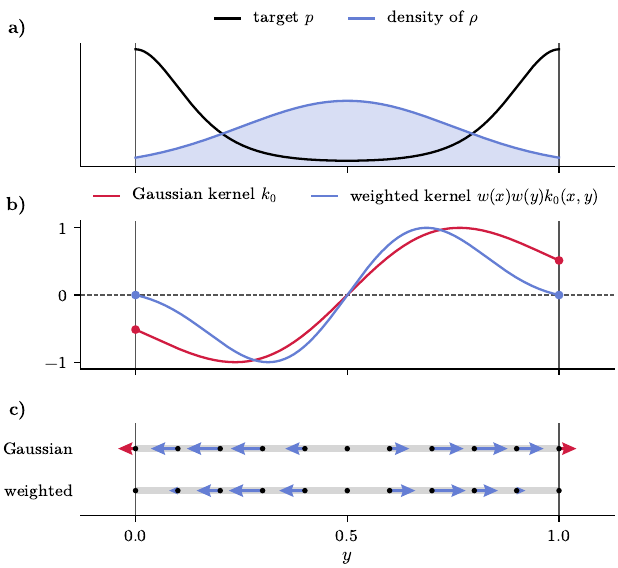}
  \caption{The vector field of the Stein witness on $\Omega = [0, 1]$. \textbf{a)}~The target density $p \propto e^{-3 f}$, with $f(x) = 16 x^2 (1 - x)^2$, whose modes are the endpoints of $\Omega$. The density of $\rho$ is a Gaussian of mean $0.5$ and standard deviation $0.25$, restricted to $\Omega$. \textbf{b)}~The field $(\phi^*_\rho)_\Phi$ for the Gaussian kernel $k_0(x, y) = e^{-(x - y)^2 / (2 \cdot 0.15^2)}$, in red, and for the weighted kernel $w(x) w(y) k_0(x, y)$ with $w(x) = x(1 - x)$, in blue. Each field is divided by its maximum absolute value on $\Omega$. A positive value moves a particle to the right. \textbf{c)}~The particles of a regular grid of $\Omega$, and their displacements $y \mapsto y + \eps (\phi^*_\rho)_\Phi(y)$ for a small step $\eps > 0$, as in \cref{fig:perturbation}. The arrows are red when they leave $\Omega$. Both fields move the particles toward the modes of $p$. With the Gaussian kernel, the particles located at $0$ and $1$ leave $\Omega$. With the weighted kernel, the field vanishes at $0$ and $1$, and these particles do not move. Outside $\Omega$, the weight $w$ can be modified to be bounded and Lipschitz, as required by \cref{ass:kernel}, without changing the figure.}
  \label{fig:stein-witness}
\end{figure}

As $\K(\pi | \pi) = 0$, (ii) of \cref{thm:stein-witness} gives $\phi^*_\pi = 0$, \ie the Stein witness of $\pi$ vanishes.

By (ii) of \cref{thm:stein-witness}, the KSD is expressed with the feature map $\Phi$, whose values may lie in an infinite-dimensional space. It can also be computed from the kernel $k$ and the score $s$ only. Let $x, y \in \R^d$ and $1 \leq i \leq d$. As $k(x, y) = \ket{\Phi(x), \Phi(y)}_{H_0}$ and $\Phi$ is continuously differentiable, we have that
\begin{align*}
  \partial_{x_i} k(x, y) & = \ket{D \Phi(x)(e_i), \Phi(y)}_{H_0}, \\
  \partial_{y_i} k(x, y) & = \ket{\Phi(x), D \Phi(y)(e_i)}_{H_0}, \\
  \partial_{x_i} \partial_{y_i} k(x, y) & = \ket{D \Phi(x)(e_i), D \Phi(y)(e_i)}_{H_0}.
\end{align*}
Expanding the inner product of the Stein features, we obtain
\[
  \ket{\xi_x, \xi_y}_{\cal{H}} = k(x, y) \ket{s(x), s(y)} + \ket{s(x), \nabla_y k(x, y)} + \ket{s(y), \nabla_x k(x, y)} + \sum_{i = 1}^d \partial_{x_i} \partial_{y_i} k(x, y).
\]
Here, $s(x)$ and $s(y)$ are identified with vectors of $\R^d$, and $\nabla_x k$ and $\nabla_y k$ are the gradients of $k$ in its first and second variables. Let $\rho \in \cal{P}(\Omega)$. As $\phi^*_\rho = \bb{E}_{x \sim \rho}[\xi_x]$ and the inner product commutes with the Bochner integral, (ii) of \cref{thm:stein-witness} gives
\[
  \K(\rho | \pi) = \norm{\bb{E}_{x \sim \rho}[\xi_x]}_{\cal{H}}^2 = \bb{E}_{x \sim \rho, y \sim \rho}[\ket{\xi_x, \xi_y}_{\cal{H}}],
\]
where $x$ and $y$ are independent. This is the expression of the KSD of \cite{Liu2016Kernel}, used to compute it in practice (see \cref{fig:svgd-simulation}).

\subsection{The steepest descent of the KL divergence}\label{sec:kl-descent}

In this section, we show that the Stein witness of \cref{sec:stein-witness} is the direction of steepest descent of the Kullback--Leibler divergence to $\pi$, as in \cite{Liu2016}. This is the key of the \svgd algorithm.

Let $\rho$ be a probability measure on $\R^d$ such that $\KL(\rho \,\|\, \pi) < \infty$. For $f \in \cal{H}$ and $\eps \in \R$, the map $I_d + \eps f_\Phi$ moves each point $x \in \R^d$ to $x + \eps f_\Phi(x)$. Let
\[
  h_f(\eps) = \KL((I_d + \eps f_\Phi)_\# \rho \,\|\, \pi).
\]
The derivative $h'_f(0)$ is the first variation of the Kullback--Leibler divergence at $\rho$ along $f$. The next theorem computes it, and minimizes it over the closed unit ball of $\cal{H}$.

\begin{theorem}[KL steepest descent]\label{thm:kl-descent}
  Suppose that the standing assumptions hold. Let $\rho$ be a probability measure on $\R^d$ such that $\KL(\rho \,\|\, \pi) < \infty$. Then, the following holds.
  \begin{enumerate}[label=(\roman*)]
    \item For every $f \in \cal{H}$, $h_f$ is real-valued on a neighborhood of $0$, differentiable at $0$, and $h'_f(0) = -\ket{f, \phi^*_\rho}_{\cal{H}}$.
    \item On the closed unit ball of $\cal{H}$, the function $f \mapsto h'_f(0)$ reaches its minimum at $f^* = \phi^*_\rho / \norm{\phi^*_\rho}_{\cal{H}}$ if $\phi^*_\rho \neq 0$, and at $f^* = 0$ otherwise. If $\phi^*_\rho \neq 0$, $f^*$ is its only minimizer on the closed unit ball.
    \item $h'_{\phi^*_\rho}(0) = -\K(\rho | \pi)$.
  \end{enumerate}
\end{theorem}

\begin{proof}
  As $\KL(\rho \,\|\, \pi) < \infty$, we have that $\rho \in \cal{P}(\Omega)$ (see \cref{sec:first-variation}). Thus, $\phi^*_\rho$ is given by \cref{thm:stein-witness}. Let $f \in \cal{H}$. By \cref{sec:kernel-assumption}, $f_\Phi$ is continuously differentiable, $\norm{f}_{\cal{H}} K$-Lipschitz, and zero on $\partial \Omega$. Thus, \cref{thm:first-variation} applies to $v = f_\Phi$. It gives that $h_f$ is real-valued on a neighborhood of $0$, differentiable at $0$, and
  \[
    h'_f(0) = -\bb{E}_{x \sim \rho}[\cal{A}_s f_\Phi(x)] = -\ket{f, \phi^*_\rho}_{\cal{H}},
  \]
  where the second equality is \cref{thm:stein-witness}. This gives (i).

  By (i), minimizing $f \mapsto h'_f(0)$ on the closed unit ball amounts to maximizing $f \mapsto \ket{f, \phi^*_\rho}_{\cal{H}} = \bb{E}_{x \sim \rho}[\cal{A}_s f_\Phi(x)]$. By (i) of \cref{thm:stein-witness}, $f^*$ is a maximizer, and $\ket{f^*, \phi^*_\rho}_{\cal{H}} = \norm{\phi^*_\rho}_{\cal{H}}$. Suppose that $\phi^*_\rho \neq 0$. Let $f \in \cal{H}$ such that $\norm{f}_{\cal{H}} \leq 1$ and $\ket{f, \phi^*_\rho}_{\cal{H}} = \norm{\phi^*_\rho}_{\cal{H}}$. Then, $\norm{f^*}_{\cal{H}} = 1$ and $\ket{f, f^*}_{\cal{H}} = 1$. Thus,
  \[
    \norm{f - f^*}_{\cal{H}}^2 = \norm{f}_{\cal{H}}^2 - 2 \ket{f, f^*}_{\cal{H}} + \norm{f^*}_{\cal{H}}^2 = \norm{f}_{\cal{H}}^2 - 1 \leq 0,
  \]
  so that $f = f^*$. This gives (ii). Finally, (i) with $f = \phi^*_\rho$ and (ii) of \cref{thm:stein-witness} give
  \[
    h'_{\phi^*_\rho}(0) = -\ket{\phi^*_\rho, \phi^*_\rho}_{\cal{H}} = -\norm{\phi^*_\rho}_{\cal{H}}^2 = -\K(\rho | \pi),
  \]
  which concludes the proof.
\end{proof}

This result is the key of the \svgd algorithm. It means that $\phi^*_\rho$ is the optimal direction, within $\cal{H}$, to update $\rho$ in order to decrease the Kullback--Leibler divergence to $\pi$. If $\phi^*_\rho \neq 0$, then $h'_{\phi^*_\rho}(0) = -\K(\rho | \pi) < 0$, so that, for every $\eps > 0$ small enough,
\[
  \KL((I_d + \eps (\phi^*_\rho)_\Phi)_\# \rho \,\|\, \pi) < \KL(\rho \,\|\, \pi).
\]
Consequently, \svgd iteratively updates $\rho$ in the direction $\phi^*_\rho$, with a small step size $\eps > 0$, \ie
\begin{equation}\label{eq:svgd}
  \rho_{n + 1} = (I_d + \eps (\phi^*_{\rho_n})_\Phi)_\# \rho_n, \quad n \in \N.
\end{equation}
As $\phi^*_\pi = 0$, $\pi$ is a fixed point of \cref{eq:svgd}. Let $\rho_0 \in \cal{P}(\Omega)$. For every $\rho \in \cal{P}(\Omega)$, $\norm{\phi^*_\rho}_{\cal{H}} \leq B$, so that $(\phi^*_\rho)_\Phi$ is $B K$-Lipschitz. If $\eps B K < 1$, \cref{lem:perturbation} shows that every map $I_d + \eps (\phi^*_{\rho_n})_\Phi$ maps $\Omega$ onto $\Omega$. Thus, by induction, $\rho_n \in \cal{P}(\Omega)$ for every $n \in \N$.

In practice, $\rho_0$ is the empirical measure of $N$ particles $x_0^1, \dots, x_0^N \in \Omega$. Then, every $\rho_n$ is the empirical measure of the particles $x_n^1, \dots, x_n^N$, and \cref{eq:svgd} moves each particle to
\[
  x_{n + 1}^j = x_n^j + \eps (\phi^*_{\rho_n})_\Phi(x_n^j).
\]
For $\rho = \frac{1}{N} \sum_{j = 1}^N \delta_{x^j}$ and $y \in \R^d$, we have that $\phi^*_\rho = \frac{1}{N} \sum_{j = 1}^N \xi_{x^j}$. As $\partial_{x_i} k(x, y) = \ket{D \Phi(x)(e_i), \Phi(y)}_{H_0}$ (see \cref{sec:stein-witness}), we obtain
\[
  (\phi^*_\rho)_\Phi(y) = \frac{1}{N} \sum_{j = 1}^N (k(x^j, y) s(x^j) + \nabla_x k(x^j, y)).
\]
This is the update of \cite{Liu2016}, with the kernel $k$ of the feature map $\Phi$. \cref{fig:svgd-simulation} shows this iteration on the domain of \cref{fig:perturbation}.

\begin{figure}[t]
  \centering
  \includegraphics[width=\textwidth]{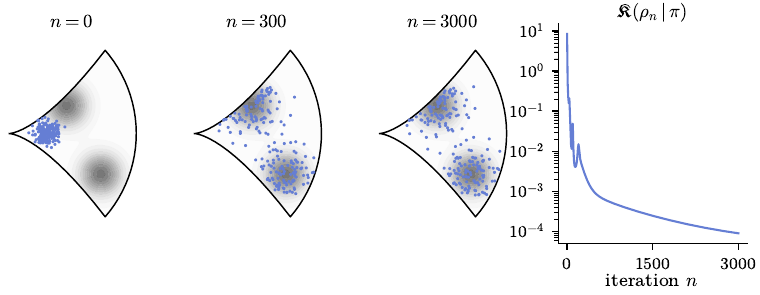}
  \caption{\svgd with $N = 200$ particles on the domain $\Omega$ of \cref{fig:perturbation}. The density of the target $\pi$ is proportional to a mixture of two Gaussians of standard deviation $0.12$, restricted to $\Omega$ (gray levels). The kernel is $k(z, z') = w(z) w(z') e^{-\norm{z - z'}^2 / (2 \cdot 0.3^2)}$, with $w(z) = c (z_1^3 - z_2^2)(1 - \norm{z}^2)$ and $c \approx 5.38$, so that $\max_\Omega w = 1$. The weight $w$ vanishes on $\partial \Omega$. Outside $\Omega$, the weight can be modified to be bounded and Lipschitz, as required by \cref{ass:kernel}, without changing the figure. The particles start near the cusp and follow \cref{eq:svgd} with $\eps = 0.08$. \textbf{Left.} The particles at the iterations $n = 0$, $300$ and $3000$. They spread over the two modes of $\pi$, and none of them leaves $\Omega$. \textbf{Right.} The KSD of the empirical measure $\rho_n$, computed with its expression in terms of the kernel (see \cref{sec:stein-witness}), in which the expectation is a double sum over the particles. Empirical measures are not covered by the convergence result of \cref{sec:convergence}.}
  \label{fig:svgd-simulation}
\end{figure}

\subsection{SVGD in continuous time}\label{sec:flow}

In this section, we study \svgd in continuous time. The particles then follow a differential equation, of which the iteration \cref{eq:svgd} is a time discretization. We show that this equation has a solution for all times, and that the particles never leave $\Omega$. We do not study the convergence of the iteration to this solution when $\eps$ tends to $0$.

\subsubsection{The continuous-time dynamics}\label{sec:mckean-vlasov}

We first recall the objects of \cref{sec:ksd} that we use. For $\rho \in \cal{P}(\Omega)$, the Stein witness is $\phi^*_\rho = \E_{x \sim \rho}[\xi_x]$, and $\norm{\phi^*_\rho}_{\cal{H}} \leq B$, where $B = \max_{x \in \Omega} \norm{\xi_x}_{\cal{H}}$. For $f \in \cal{H}$, the vector field $f_\Phi$ is $\norm{f}_{\cal{H}} K$-Lipschitz, bounded by $\norm{f}_{\cal{H}} C$, and zero on $\partial \Omega$.

Let $\rho_0 \in \cal{P}(\Omega)$ be the initial distribution of the particles. For $x \in \R^d$ and $\rho \in \cal{P}(\Omega)$, let
\[
  b(x, \rho) = (\phi^*_\rho)_\Phi(x),
\]
\ie the vector field of the Stein witness of $\rho$, evaluated at $x$. It is the velocity that \svgd gives to a particle located at $x$, when the particles are distributed according to $\rho$. Indeed, the iteration \cref{eq:svgd} moves this particle to $x + \eps b(x, \rho)$. This is one step of the explicit Euler scheme, with step $\eps$, for the equation in which every particle has the velocity $b$. We now write this equation. For a random variable $X$, let $\cal{L}(X)$ be its law. Let $X_t$ be the position at time $t$ of a particle whose initial position $X_0$ is drawn from $\rho_0$. Then,
\begin{equation}\label{eq:mckean-vlasov}
  \begin{cases}
    \dif X_t = b(X_t, \rho_t) \dif t \\
    \cal{L}(X_0) = \rho_0
  \end{cases} \quad \text{where} \quad \rho_t = \cal{L}(X_t).
\end{equation}
This is a McKean--Vlasov equation, \ie a differential equation whose coefficients depend on the law of its solution. The function $b$ is its drift, \ie the velocity of the particles. Here, the drift depends on the distribution $\rho_t$ of all the particles. The equation has no diffusion term, \ie no noise. It is studied by \cite{Liu2017} and \cite{Lu2019} on the whole space $\R^d$. Here, the particles live in the compact set $\Omega$, and the score is only continuous.

As there is no noise, a particle moves deterministically. For $x \in \R^d$, let $T_t(x)$ be the position at time $t$ of the particle that starts from $x$. Then, $X_t = T_t(X_0)$, and the distribution of the particles at time $t$ is $\rho_t = (T_t)_\# \rho_0$. Thus, \cref{eq:mckean-vlasov} reads
\begin{equation}\label{eq:flow}
  \frac{\dif}{\dif t} T_t(x) = b(T_t(x), \rho_t), \quad \text{where} \quad \rho_t = (T_t)_\# \rho_0.
\end{equation}
The drift $b(\cdot, \rho)$ is only defined for $\rho \in \cal{P}(\Omega)$. Thus, we say that a family $(T_t)_{t \geq 0}$ of continuous maps from $\R^d$ to $\R^d$ solves \cref{eq:flow} when $T_0 = I_d$, $\rho_t \in \cal{P}(\Omega)$ for every $t \geq 0$, and \cref{eq:flow} holds for all $x \in \R^d$ and $t \geq 0$. At $t = 0$, the derivative is a right derivative.

This equation is not a standard differential equation. In a standard one, the drift is given, and the Cauchy--Lipschitz theorem provides the positions. Here, the drift $b(\cdot, \rho_t)$ is not given. It depends on the distribution $\rho_t$ of the particles, which depends on their positions, which are the unknown. In short, the positions are needed to compute the drift, and the drift is needed to compute the positions.

\subsubsection{Existence of a solution}\label{sec:main-continuous}

Despite this circular dependence, \cref{eq:flow} has a solution, which is defined for all times. This is the main result of this section.

\begin{theorem}[Existence of the continuous-time dynamics]\label{thm:continuous-time}
  Suppose that the standing assumptions hold. Let $\rho_0 \in \cal{P}(\Omega)$. Then, there exists a family $(T_t)_{t \geq 0}$ of maps from $\R^d$ to $\R^d$ that solves \cref{eq:flow}.
\end{theorem}

Every solution of \cref{eq:flow}, and in particular the solutions given by \cref{thm:continuous-time}, has three properties, which we prove in \cref{sec:solution-properties} (\cref{lem:solution-properties}). Let $\rho_t = (T_t)_\# \rho_0$. Each property has a simple meaning for the particles. First, every $T_t$ is a homeomorphism of $\R^d$ (\cref{lem:solution-properties} (iii)). Thus, two distinct particles never meet, and the cloud of particles is deformed continuously. Second, $T_t(\Omega) = \Omega$ (\cref{lem:solution-properties} (iii)), \ie every particle that starts in $\Omega$ stays in $\Omega$. The reason is that $\Phi$, and thus every vector field $f_\Phi$, vanishes on $\partial \Omega$. The particles located on $\partial \Omega$ do not move (\cref{lem:solution-properties} (iv)), and, by uniqueness, no other trajectory can reach a point at rest. Note that the definition of a solution only requires that $\rho_t$ be carried by $\Omega$, \ie that almost every particle stays in $\Omega$. Third, the Stein witness $\phi^*_{\rho_t}$ varies continuously in time (\cref{lem:solution-properties} (i)), and so does the drift $b(\cdot, \rho_t)$. Moreover, every Stein witness is bounded by $B$, so that $\norm{b(y, \rho_t)} \leq B C$ for every $y \in \R^d$, \ie the speed of the particles is at most $B C$.

\cref{thm:continuous-time} does not claim that the solution is unique. Once the curve $(\rho_t)_{t \geq 0}$ of a solution is fixed, the drift $b(\cdot, \rho_t)$ is $B K$-Lipschitz in space and continuous in time, so that each particle has a unique trajectory (\cref{lem:solution-properties}). We do not know whether the curve $(\rho_t)_{t \geq 0}$ itself is unique. Our argument would give its uniqueness if the Stein feature were Lipschitz on $\Omega$, which requires more than the continuity of $s$ and $D \Phi$ (see \cref{sec:schauder}).

\subsubsection{Strategy of the proof}\label{sec:strategy}

The proof breaks the circular dependence of \cref{eq:flow} into two halves, which are treated one after the other.

The first half goes from the drift to the positions. Suppose that the drift is given at every time. Then, the particles no longer interact, as each of them follows a prescribed velocity field. Its position solves a standard differential equation.

The second half goes from the positions to the drift. Once the positions are known, so is the distribution of the particles at every time. \svgd then prescribes a drift, through the Stein witness of this distribution.

Chaining the two halves, we start from a candidate drift and we obtain a new one, \ie
\[
  \text{candidate drift} \quad \longmapsto \quad \text{positions} \quad \longmapsto \quad \text{distribution} \quad \longmapsto \quad \text{drift prescribed by \svgd}.
\]
In general, the last drift differs from the candidate. When they coincide, the candidate is consistent with the positions that it produces, and these positions solve \cref{eq:flow}. Thus, we look for a candidate that this chain leaves unchanged, \ie a fixed point.

In the whole proof, a \emph{candidate} is a continuous and bounded curve $m : [0, \infty) \to \cal{H}$. At time $t$, it gives the velocity $m(t)_\Phi(y)$ to a particle located at $y$. The proof has four steps.
\begin{enumerate}[label=\arabic*.]
  \item For a given candidate $m$, we compute the positions of the particles, and we study them (\cref{sec:frozen}).
  \item We write the chain above as a map $m \mapsto G_{\rho_0}(m)$. We show that a fixed point of $G_{\rho_0}$ gives a solution of \cref{eq:flow} (\cref{sec:fixed-point-equation}).
  \item We show that a fixed point exists when the time is restricted to a bounded interval $[0, a]$ (\cref{sec:schauder}). The argument relies on compactness, which requires a bounded time interval.
  \item We reach all times (\cref{sec:all-times}). As the fixed points need not be unique, those of two different intervals need not agree. Thus, we glue fixed points on $[0, 1]$, $[1, 2]$, and so on. Each of them starts from the distribution reached at the end of the previous one.
\end{enumerate}

Finally, \cref{sec:solution-properties} shows that every solution of \cref{eq:flow}, and not only the one built in this proof, has the properties announced in \cref{sec:main-continuous}.

\subsubsection{From the drift to the positions}\label{sec:frozen}

In this subsection, the drift is given. Let $m$ be a candidate, and let $R \geq 0$ such that $\norm{m(t)}_{\cal{H}} \leq R$ for every $t \geq 0$. At time $t$, a particle located at $y$ has the velocity $m(t)_\Phi(y)$. For $x \in \R^d$, let $T^m_t(x)$ be the position at time $t$ of the particle that starts from $x$. The next lemma states that these positions are well defined, and gives three properties. By (i), the particles never meet and never cross $\partial \Omega$. By (ii), their speed is at most $R C$. By (iii), two close candidates give close positions.

\begin{lemma}[Positions for a given drift]\label{lem:frozen-flow}
  Suppose that the standing assumptions hold. Let $R \geq 0$ and let $m : [0, \infty) \to \cal{H}$ be continuous, with $\norm{m(t)}_{\cal{H}} \leq R$ for every $t \geq 0$. Then, for every $x \in \R^d$, the equation
  \[
    \frac{\dif}{\dif t} T^m_t(x) = m(t)_\Phi(T^m_t(x)), \quad T^m_0(x) = x
  \]
  has a unique solution on $[0, \infty)$. Moreover, the following holds.
  \begin{enumerate}[label=(\roman*)]
    \item For every $t \geq 0$, $T^m_t$ is a homeomorphism of $\R^d$, and $T^m_t(\Omega) = \Omega$.
    \item For all $x \in \R^d$ and $t, t' \geq 0$, $\norm{T^m_t(x) - T^m_{t'}(x)} \leq R C \abs{t - t'}$.
    \item Let $m' : [0, \infty) \to \cal{H}$ be continuous, with $\norm{m'(t)}_{\cal{H}} \leq R$ for every $t \geq 0$. For all $x \in \R^d$ and $t \geq 0$,
    \[
      \norm{T^m_t(x) - T^{m'}_t(x)} \leq C \, t \, e^{R K t} \sup_{0 \leq r \leq t} \norm{m(r) - m'(r)}_{\cal{H}}.
    \]
  \end{enumerate}
\end{lemma}

\begin{proof}
  \inlinetitle{Existence and uniqueness}{.}~Let $v(t, y) = m(t)_\Phi(y)$ be the velocity at time $t$ of a particle located at $y$. By the properties of the vector fields recalled in \cref{sec:mckean-vlasov}, we have that
  \[
    \forall t, t' \geq 0, \forall y, y' \in \R^d, \quad \norm{v(t, y) - v(t', y')} \leq C \norm{m(t) - m(t')}_{\cal{H}} + R K \norm{y - y'},
  \]
  and $\norm{v(t, y)} \leq R C$. The first term comes from $v(t, y) - v(t', y) = (m(t) - m(t'))_\Phi(y)$. Thus, $v$ is continuous, and $R K$-Lipschitz in space at every time. By the Cauchy--Lipschitz theorem, also known as the Picard--Lindelöf theorem, for all $t_0 \geq 0$ and $y \in \R^d$, the equation $\alpha'(t) = v(t, \alpha(t))$ has a unique solution such that $\alpha(t_0) = y$. This solution is defined on the whole interval $[0, \infty)$, as the Lipschitz constant of $v$ does not depend on the time. With $t_0 = 0$, this defines $T^m_t$.

  \inlinetitle{Homeomorphism}{.}~Let $t \geq 0$. To invert $T^m_t$, we follow the trajectories backward. For $y \in \R^d$, let $S_t(y) = \alpha(0)$, where $\alpha$ is the solution such that $\alpha(t) = y$. By uniqueness of the solutions, $S_t$ is the inverse of $T^m_t$. By Grönwall's inequality, two solutions $\alpha$ and $\beta$ satisfy
  \[
    \forall r, r' \geq 0, \quad \norm{\alpha(r) - \beta(r)} \leq e^{R K \abs{r - r'}} \norm{\alpha(r') - \beta(r')}.
  \]
  Thus, $T^m_t$ and $S_t$ are continuous, \ie $T^m_t$ is a homeomorphism.

  \inlinetitle{Invariance of $\Omega$}{.}~The points of $\partial \Omega$ are at rest. Indeed, let $z \in \partial \Omega$. As $\Phi(z) = 0$ by (iv) of \cref{ass:kernel}, we have that $v(t, z) = 0$ for every $t \geq 0$. Thus, the constant curve equal to $z$ is a solution. Let $\alpha$ be a solution such that $\alpha(r) \in \Omega$ and $\alpha(r') \notin \Omega$ for some $r, r' \geq 0$. The image by $\alpha$ of the interval between $r$ and $r'$ is connected. It meets $\Omega$ and $\R^d \setminus \Omega$, so that it meets $\partial \Omega$ at some point $z$. By uniqueness of the solutions, $\alpha$ is the constant curve equal to $z$, which contradicts $\alpha(r) \neq \alpha(r')$. Hence, a solution that meets $\Omega$ stays in $\Omega$. In particular, $T^m_t$ and $S_t$ map $\Omega$ into $\Omega$. As $S_t$ is the inverse of $T^m_t$, we obtain $\Omega = T^m_t(S_t(\Omega)) \subseteq T^m_t(\Omega)$. Thus, $T^m_t(\Omega) = \Omega$, which gives (i).

  \inlinetitle{Speed}{.}~As $\norm{v(t, y)} \leq R C$, the derivative of $t \mapsto T^m_t(x)$ is bounded by $R C$, which gives (ii).

  \inlinetitle{Close candidates give close positions}{.}~Let $t \geq 0$ and $x \in \R^d$. We measure the distance between the two candidates by
  \[
    \Delta = \sup_{0 \leq r \leq t} \norm{m(r) - m'(r)}_{\cal{H}},
  \]
  and the gap between the two positions by $\eta(r) = T^m_r(x) - T^{m'}_r(x)$. For $r \in [0, t]$, we have that
  \begin{align*}
    \norm{\eta'(r)} & \leq \norm{m(r)_\Phi(T^m_r(x)) - m(r)_\Phi(T^{m'}_r(x))} + \norm{(m(r) - m'(r))_\Phi(T^{m'}_r(x))} \\
                    & \leq R K \norm{\eta(r)} + C \Delta.
  \end{align*}
  The first term compares the same velocity field at two points, and the second one compares two velocity fields at the same point. As $\eta(0) = 0$, Grönwall's inequality gives $\norm{\eta(t)} \leq C \, \Delta \, t \, e^{R K t}$, which concludes the proof.
\end{proof}

\subsubsection{From the positions to the drift}\label{sec:fixed-point-equation}

In this subsection, we close the chain of \cref{sec:strategy}. Let $\rho_0 \in \cal{P}(\Omega)$, and let $m$ be a candidate. By (i) of \cref{lem:frozen-flow}, $(T^m_t)_\# \rho_0 \in \cal{P}(\Omega)$ for every $t \geq 0$. For this distribution, \svgd prescribes the drift given by the Stein witness $\phi^*_{(T^m_t)_\# \rho_0}$. This defines a new curve from $[0, \infty)$ to $\cal{H}$, denoted by
\[
  G_{\rho_0}(m) : t \mapsto \phi^*_{(T^m_t)_\# \rho_0}.
\]
The candidate $m$ is \emph{consistent} with the positions that it produces when $G_{\rho_0}$ leaves it unchanged, \ie
\begin{equation}\label{eq:fixed-point}
  m = G_{\rho_0}(m), \quad \text{i.e.} \quad \forall t \geq 0, \quad m(t) = \phi^*_{(T^m_t)_\# \rho_0}.
\end{equation}
The curve $m$ appears on both sides of \cref{eq:fixed-point}, as the right-hand side depends on $m$ through the positions $T^m_t$. Thus, \cref{eq:fixed-point} is an equation on $m$, and not a definition of $m$.

We now show two facts. First, all the Stein witnesses live in a fixed compact set $\cal{W}$. Thus, we can restrict ourselves to the candidates with values in $\cal{W}$. Second, a candidate that satisfies \cref{eq:fixed-point} gives a solution of \cref{eq:flow}.

\begin{lemma}[Set of the Stein witnesses]\label{lem:witness-set}
  Suppose that the standing assumptions hold. Let $\cal{W}$ be the closed convex hull of $\{\xi_x \mid x \in \Omega\}$ in $\cal{H}$. Then, $\cal{W}$ is compact, and
  \[
    \forall f \in \cal{W}, \quad \norm{f}_{\cal{H}} \leq B, \quad \text{and} \quad \forall \rho \in \cal{P}(\Omega), \quad \phi^*_\rho \in \cal{W}.
  \]
\end{lemma}

\begin{proof}
  By (i) of \cref{lem:xi-regularity}, $x \mapsto \xi_x$ is continuous on the compact set $\Omega$. Thus, $\{\xi_x \mid x \in \Omega\}$ is compact. By (ii) of \cref{lem:xi-regularity}, it is contained in the closed ball of radius $B$ of $\cal{H}$. As $\cal{H}$ is a Hilbert space, the closed convex hull of a compact set is compact, by Mazur's theorem. Thus, $\cal{W}$ is compact. As the closed ball is closed and convex, it contains $\cal{W}$.

  Let $\rho \in \cal{P}(\Omega)$. A Stein witness is a mean of Stein features, and a mean of points of a closed convex set stays in this set. Indeed, by the Hahn--Banach theorem, $\cal{W}$ is an intersection of closed half-spaces. Each of them contains $\xi_x$ for $\rho$-almost every $x$. It thus contains the mean $\phi^*_\rho = \E_{x \sim \rho}[\xi_x]$, as continuous linear forms commute with the Bochner integral. Hence, $\phi^*_\rho \in \cal{W}$.
\end{proof}

In particular, $G_{\rho_0}(m)$ takes its values in $\cal{W}$ for every candidate $m$. Conversely, a continuous curve $m : [0, \infty) \to \cal{W}$ is bounded by $B$. Thus, it is a candidate, and \cref{lem:frozen-flow} applies to it with $R = B$.

\begin{lemma}[A consistent candidate gives a solution]\label{lem:consistent}
  Suppose that the standing assumptions hold. Let $\rho_0 \in \cal{P}(\Omega)$, and let $m$ be a candidate with values in $\cal{W}$ such that $m = G_{\rho_0}(m)$. Then, the maps $T_t = T^m_t$ solve \cref{eq:flow}.
\end{lemma}

\begin{proof}
  Let $T_t = T^m_t$ and $\rho_t = (T_t)_\# \rho_0$. As $m$ takes its values in $\cal{W}$, it is bounded by $B$ (\cref{lem:witness-set}). By \cref{lem:frozen-flow} with $R = B$, every $T_t$ is a homeomorphism of $\R^d$ such that $T_t(\Omega) = \Omega$. Thus, $\rho_t \in \cal{P}(\Omega)$. We also have that $T_0 = I_d$. As $m(t) = \phi^*_{\rho_t}$ by \cref{eq:fixed-point}, the equation of \cref{lem:frozen-flow} reads
  \[
    \frac{\dif}{\dif t} T_t(x) = m(t)_\Phi(T_t(x)) = (\phi^*_{\rho_t})_\Phi(T_t(x)) = b(T_t(x), \rho_t),
  \]
  which concludes the proof.
\end{proof}

Hence, to prove \cref{thm:continuous-time}, it remains to find a candidate $m$ with values in $\cal{W}$ such that $m = G_{\rho_0}(m)$.

\subsubsection{A consistent candidate on a bounded time interval}\label{sec:schauder}

Let $\rho \in \cal{P}(\Omega)$ and $a > 0$. In this subsection, the time is restricted to $[0, a]$. \cref{sec:all-times} applies its result to several initial distributions. Thus, we write $\rho$ for the initial distribution, instead of $\rho_0$, and $g$ for a candidate restricted to $[0, a]$. More precisely, let $g : [0, a] \to \cal{W}$ be a continuous curve. We extend it to $[0, \infty)$ by the constant $g(a)$. This extension is a candidate bounded by $B$, and $T^g_t$ denotes the positions that it gives by \cref{lem:frozen-flow}, with $R = B$. Let
\[
  G g(t) = \phi^*_{(T^g_t)_\# \rho}, \quad t \in [0, a].
\]
This is the map of \cref{sec:fixed-point-equation} for the initial distribution $\rho$, restricted to the times of $[0, a]$. We look for a curve $g$ such that $G g = g$. We measure the distance between two curves by the norm $\norm{g}_\infty = \sup_{t \in [0, a]} \norm{g(t)}_{\cal{H}}$.

We first control how much the Stein witness changes when the particles are moved. For $r \geq 0$, let
\[
  \omega(r) = \sup \{\norm{\xi_x - \xi_y}_{\cal{H}} \mid x, y \in \Omega, \norm{x - y} \leq r\}.
\]
This is the modulus of continuity of the Stein feature on $\Omega$. It measures how much $\xi$ varies between two points of $\Omega$ at distance at most $r$. It is nondecreasing, and $\omega(r)$ tends to $0$ as $r$ tends to $0$, by (i) of \cref{lem:xi-regularity}. Let $\psi_1, \psi_2 : \R^d \to \R^d$ be continuous maps that send $\Omega$ into $\Omega$. They represent two ways of moving the particles. As $\rho$ is carried by $\Omega$, the change of variables formula gives $\phi^*_{(\psi_1)_\# \rho} = \E_{x \sim \rho}[\xi_{\psi_1(x)}]$, and likewise for $\psi_2$. As the norm of a mean is at most the mean of the norms, we obtain
\begin{equation}\label{eq:estimate}
  \norm{\phi^*_{(\psi_1)_\# \rho} - \phi^*_{(\psi_2)_\# \rho}}_{\cal{H}} \leq \E_{x \sim \rho}[\norm{\xi_{\psi_1(x)} - \xi_{\psi_2(x)}}_{\cal{H}}] \leq \omega\left(\sup_{x \in \Omega} \norm{\psi_1(x) - \psi_2(x)}\right).
\end{equation}
In words, if the two maps send every particle to two points at distance at most $r$, then the two Stein witnesses differ by at most $\omega(r)$.

This estimate tells how $G$ reacts to a change of the curve. Let $g, g' : [0, a] \to \cal{W}$ be continuous curves. By (i) and (iii) of \cref{lem:frozen-flow} with $R = B$, the positions that they give at time $t \in [0, a]$ send $\Omega$ into $\Omega$, and differ by at most $a \, C \, e^{B K a} \norm{g - g'}_\infty$. By \cref{eq:estimate}, as $\omega$ is nondecreasing, we obtain
\begin{equation}\label{eq:G-continuity}
  \norm{G g - G g'}_\infty \leq \omega(a \, C \, e^{B K a} \norm{g - g'}_\infty).
\end{equation}

A natural idea is to start from any curve $g$ and to apply $G$ repeatedly, as in the proof of the Cauchy--Lipschitz theorem. This works when $G$ is a contraction, \ie when it brings any two curves closer by a fixed factor. By \cref{eq:G-continuity}, this holds when $\omega(r) \leq L r$ for some $L \geq 0$ and $L \, a \, C \, e^{B K a} < 1$. The first condition means that the Stein feature is Lipschitz on $\Omega$. However, the score $s$ and the derivative $D \Phi$ need not be Lipschitz. We only know that the Stein feature is uniformly continuous, and we cannot show that $G$ is a contraction. We use instead the following theorem, which requires no contraction. In return, it gives the existence of a fixed point, and not its uniqueness.

\begin{theorem}[Schauder \cite{Schauder1930}]\label{thm:schauder}
  Let $\Gamma$ be a nonempty, compact and convex subset of a Banach space. Then, every continuous map from $\Gamma$ to $\Gamma$ has a fixed point.
\end{theorem}

To apply it, we need three things.
\begin{enumerate}[label=\arabic*.]
  \item A set $\Gamma$ of curves $g$ that is convex and compact.
  \item The map $G$ sends $\Gamma$ into $\Gamma$.
  \item The map $G$ is continuous on $\Gamma$, \ie two close curves $g$ and $g'$ give two close curves $G g$ and $G g'$.
\end{enumerate}
Compactness is the delicate point. The Arzelà--Ascoli theorem gives a criterion. A closed set of continuous curves is compact when two conditions hold. First, the curves take their values in a compact set. Second, they are equicontinuous, \ie they have a common modulus of continuity. The first condition is given by $\cal{W}$. For the second one, we use that the speed of the particles is bounded.

\begin{lemma}[A consistent candidate on a bounded time interval]\label{lem:compact-interval}
  Suppose that the standing assumptions hold. Let $\rho \in \cal{P}(\Omega)$ and $a > 0$. Then, there exists a continuous curve $g : [0, a] \to \cal{W}$ such that
  \[
    \forall t \in [0, a], \quad g(t) = \phi^*_{(T^g_t)_\# \rho},
  \]
  where $T^g_t$ is given by \cref{lem:frozen-flow} for the extension of $g$ to $[0, \infty)$ by the constant $g(a)$.
\end{lemma}

\begin{proof}
  \inlinetitle{A convex and compact set of curves}{.}~By (ii) of \cref{lem:frozen-flow} with $R = B$, the particles move at speed at most $B C$. Between two times $t$ and $t'$, every particle thus moves by at most $B C \abs{t - t'}$. By \cref{eq:estimate}, the Stein witness changes by at most $\omega(B C \abs{t - t'})$. We only keep the curves that vary at this pace. Let $\Gamma$ be the set of the continuous curves $g : [0, a] \to \cal{W}$ such that
  \[
    \forall t, t' \in [0, a], \quad \norm{g(t) - g(t')}_{\cal{H}} \leq \omega(B C \abs{t - t'}).
  \]
  It is a subset of the Banach space of the continuous curves from $[0, a]$ to $\cal{H}$, with the norm $\norm{\cdot}_\infty$. The set $\Gamma$ is nonempty, as it contains the constant curves with values in $\cal{W}$. It is convex, as $\cal{W}$ is convex and the inequality above is preserved by convex combinations, by the triangle inequality. It is closed, as $\cal{W}$ is closed and the inequality above is preserved by uniform limits. Its curves are equicontinuous, as $\omega(r)$ tends to $0$ as $r$ tends to $0$, and they take their values in the compact set $\cal{W}$. Thus, $\Gamma$ is compact by the Arzelà--Ascoli theorem.

  \inlinetitle{The map $G$ sends $\Gamma$ into $\Gamma$}{.}~Let $g \in \Gamma$ and $t, t' \in [0, a]$. By (i) of \cref{lem:frozen-flow}, $(T^g_t)_\# \rho \in \cal{P}(\Omega)$, so that $G g(t) \in \cal{W}$ by \cref{lem:witness-set}. By \cref{eq:estimate}, then (ii) of \cref{lem:frozen-flow}, as $\omega$ is nondecreasing, we have that
  \[
    \norm{G g(t) - G g(t')}_{\cal{H}} \leq \omega\left(\sup_{x \in \Omega} \norm{T^g_t(x) - T^g_{t'}(x)}\right) \leq \omega(B C \abs{t - t'}).
  \]
  Thus, $G g \in \Gamma$.

  \inlinetitle{The map $G$ is continuous}{.}~Two close curves give close positions, and close positions give close Stein witnesses. This is \cref{eq:G-continuity}. As $\omega(r)$ tends to $0$ as $r$ tends to $0$, $G$ is continuous on $\Gamma$.

  \inlinetitle{Conclusion}{.}~By \cref{thm:schauder}, $G$ has a fixed point $g \in \Gamma$, which concludes the proof.
\end{proof}

\subsubsection{From a bounded time interval to all times}\label{sec:all-times}

\cref{sec:schauder} gives a consistent candidate on a bounded time interval, from any initial distribution. The length of this interval is arbitrary, and we take $a = 1$. To reach all times, we proceed interval by interval. We first solve on $[0, 1]$, from the distribution $\rho_0$. At time $1$, the particles have a new distribution $\rho_1$. We solve again on an interval of length $1$, from $\rho_1$, and so on. Gluing these pieces gives a consistent candidate for all times.

\begin{proof}
  We prove \cref{thm:continuous-time}. We define by induction on $n \in \N$ measures $\rho_n \in \cal{P}(\Omega)$ and continuous curves $g_n : [0, 1] \to \cal{W}$, starting from the initial distribution $\rho_0$. Let $n \in \N$ such that $\rho_n \in \cal{P}(\Omega)$ is defined. \cref{lem:compact-interval} with $\rho = \rho_n$ and $a = 1$ gives a continuous curve $g_n : [0, 1] \to \cal{W}$ such that
  \[
    \forall r \in [0, 1], \quad g_n(r) = \phi^*_{(T^{g_n}_r)_\# \rho_n}.
  \]
  Let $\rho_{n + 1} = (T^{g_n}_1)_\# \rho_n$, which lies in $\cal{P}(\Omega)$ by (i) of \cref{lem:frozen-flow}. Here, $\rho_n$ is the distribution of the particles at time $n$, and not the $n$-th iterate of \cref{eq:svgd}. The curve $g_n$ is the piece of the candidate between the times $n$ and $n + 1$. For $n \in \N$ and $t \in [n, n + 1]$, let $m(t) = g_n(t - n)$. It is well defined, as both pieces are equal to the Stein witness of $\rho_{n + 1}$ at time $n + 1$, \ie $g_n(1) = \phi^*_{\rho_{n + 1}} = g_{n + 1}(0)$. Thus, $m$ is continuous, with values in $\cal{W}$, \ie it is a candidate with values in $\cal{W}$.

  Let $t \in [n, n + 1]$ and $x \in \R^d$. Between the times $n$ and $n + 1$, the particles follow the velocity field of $g_n$. Thus, the curve $r \mapsto T^m_{n + r}(x)$ solves on $[0, 1]$ the equation of \cref{lem:frozen-flow} for $g_n$, from the point $T^m_n(x)$. By uniqueness of the solutions, $T^m_t = T^{g_n}_{t - n} \circ T^m_n$. By induction, $(T^m_n)_\# \rho_0 = \rho_n$. Thus,
  \[
    m(t) = g_n(t - n) = \phi^*_{(T^{g_n}_{t - n})_\# \rho_n} = \phi^*_{(T^{g_n}_{t - n} \circ T^m_n)_\# \rho_0} = \phi^*_{(T^m_t)_\# \rho_0}.
  \]
  This is \cref{eq:fixed-point}, \ie $m = G_{\rho_0}(m)$. By \cref{lem:consistent}, this concludes the proof.
\end{proof}

\subsubsection{Properties of the solutions}\label{sec:solution-properties}

We now show that every solution of \cref{eq:flow}, and not only the solutions given by \cref{thm:continuous-time}, has the properties announced in \cref{sec:main-continuous}. The key point is that, once a solution is fixed, its Stein witnesses form a fixed curve. The trajectories then solve a standard differential equation, and \cref{lem:frozen-flow} applies.

\begin{lemma}[Properties of the solutions]\label{lem:solution-properties}
  Suppose that the standing assumptions hold. Let $\rho_0 \in \cal{P}(\Omega)$, let $(T_t)_{t \geq 0}$ be a solution of \cref{eq:flow}, let $\rho_t = (T_t)_\# \rho_0$, and let $m(t) = \phi^*_{\rho_t}$. Then, the following holds.
  \begin{enumerate}[label=(\roman*)]
    \item For all $t, t' \geq 0$, $\norm{m(t) - m(t')}_{\cal{H}} \leq \omega(B C \abs{t - t'})$, where $\omega$ is the modulus of continuity of the Stein feature on $\Omega$, defined in \cref{sec:schauder}. In particular, $m$ is continuous.
    \item For every $t \geq 0$, $T_t = T^m_t$, where $T^m_t$ is given by \cref{lem:frozen-flow} with $R = B$.
    \item For every $t \geq 0$, $T_t$ is a homeomorphism of $\R^d$, and $T_t(\Omega) = \Omega$.
    \item For all $z \in \partial \Omega$ and $t \geq 0$, $T_t(z) = z$.
  \end{enumerate}
\end{lemma}

\begin{proof}
  \inlinetitle{Speed}{.}~Every Stein witness has norm at most $B$. As $\norm{\Phi(y)}_{H_0} \leq C$, we have that $\norm{b(y, \rho_t)} \leq \norm{m(t)}_{\cal{H}} \norm{\Phi(y)}_{H_0} \leq B C$ for all $t \geq 0$ and $y \in \R^d$. By \cref{eq:flow}, the particles thus move at speed at most $B C$, \ie
  \[
    \forall x \in \R^d, \forall t, t' \geq 0, \quad \norm{T_t(x) - T_{t'}(x)} \leq B C \abs{t - t'}.
  \]

  \inlinetitle{Continuity of the Stein witness}{.}~Let $t, t' \geq 0$. As $\rho_t, \rho_{t'} \in \cal{P}(\Omega)$, the points $T_t(x)$ and $T_{t'}(x)$ lie in $\Omega$ for $\rho_0$-almost every $x$. By the change of variables formula, and as the norm of a mean is at most the mean of the norms, we have that
  \[
    \norm{m(t) - m(t')}_{\cal{H}} = \norm{\E_{x \sim \rho_0}[\xi_{T_t(x)} - \xi_{T_{t'}(x)}]}_{\cal{H}} \leq \E_{x \sim \rho_0}[\norm{\xi_{T_t(x)} - \xi_{T_{t'}(x)}}_{\cal{H}}] \leq \omega(B C \abs{t - t'}).
  \]
  The last inequality follows from the speed bound above, the definition of $\omega$, and the fact that $\omega$ is nondecreasing.
  As $\omega(r)$ tends to $0$ as $r$ tends to $0$, $m$ is continuous, which gives (i).

  \inlinetitle{The trajectories}{.}~The solution $(T_t)_{t \geq 0}$ is fixed, so that $m$ is a fixed curve in $\cal{H}$. By \cref{eq:flow}, for every $x \in \R^d$, the curve $t \mapsto T_t(x)$ solves the equation of \cref{lem:frozen-flow} for $m$, from $x$. By (i), $m$ is continuous, and it is bounded by $B$. Thus, \cref{lem:frozen-flow} applies to $m$ with $R = B$. It gives a unique solution of this equation from each initial position. As $t \mapsto T_t(x)$ and $t \mapsto T^m_t(x)$ are two solutions from $x$, they are equal, which gives (ii). This only uses the uniqueness of the trajectories for the fixed curve $m$. It does not use the uniqueness of the solution of \cref{eq:flow}, which is not known. By (ii), (iii) follows from (i) of \cref{lem:frozen-flow}.

  \inlinetitle{The boundary}{.}~Let $z \in \partial \Omega$. As $\Phi(z) = 0$ by (iv) of \cref{ass:kernel}, $m(t)_\Phi(z) = 0$ for every $t \geq 0$. Thus, the constant curve equal to $z$ solves the equation of \cref{lem:frozen-flow} for $m$, from $z$. By uniqueness, $T_t(z) = T^m_t(z) = z$, which concludes the proof.
\end{proof}

\subsection{The time derivative of the KL divergence}\label{sec:kl-derivative}

In \cref{sec:kl-descent}, we showed that the derivative of the divergence, in the direction of the vector field of the Stein witness, is $-\K(\rho | \pi)$ (see (iii) of \cref{thm:kl-descent}). In continuous time, the particles follow the vector field of the Stein witness of their distribution at every time. Thus, we expect the divergence to decrease at rate $\K(\rho_t | \pi)$ at every time. In this section, we prove it.

\subsubsection{Decrease of the Kullback--Leibler divergence}\label{sec:kl-derivative-main}

\begin{theorem}[Time derivative of the KL divergence]\label{thm:kl-derivative}
  Suppose that the standing assumptions hold. Let $\rho_0$ be a probability measure on $\R^d$ such that $\KL(\rho_0 \,\|\, \pi) < \infty$, so that $\rho_0 \in \cal{P}(\Omega)$ (see \cref{sec:first-variation}). Let $(T_t)_{t \geq 0}$ be a solution of \cref{eq:flow}, such as the solutions given by \cref{thm:continuous-time}, and let $\rho_t = (T_t)_\# \rho_0$. Then, $\KL(\rho_t \,\|\, \pi) < \infty$ for every $t \geq 0$. Moreover, the function $t \mapsto \KL(\rho_t \,\|\, \pi)$ is differentiable on $[0, \infty)$, with a right derivative at $0$, and
  \[
    \forall t \geq 0, \quad \frac{\dif}{\dif t} \KL(\rho_t \,\|\, \pi) = -\K(\rho_t | \pi).
  \]
  In particular, $t \mapsto \KL(\rho_t \,\|\, \pi)$ is nonincreasing on $[0, \infty)$.
\end{theorem}

The Kullback--Leibler divergence thus decreases along the dynamics, at rate $\K(\rho_t | \pi)$. It can only stop decreasing when $\K(\rho_t | \pi) = 0$, \ie when the Stein witness $\phi^*_{\rho_t}$ vanishes. In \cref{sec:convergence}, we use this identity to study the convergence of $\rho_t$ to $\pi$.

In \cite{Liu2017}, this identity is derived from the continuity equation satisfied by the density of $\rho_t$. Our proof follows the particles instead. It never differentiates the density of $\rho_t$, and it needs no regularity of $\partial \Omega$.

\subsubsection{Strategy of the proof}\label{sec:kl-derivative-strategy}

The proof follows that of \cref{thm:first-variation}. There, we differentiated $\eps \mapsto \KL((I_d + \eps v)_\# \rho \,\|\, \pi)$ at $\eps = 0$, for a fixed vector field $v$. Here, we differentiate $t \mapsto \KL(\rho_t \,\|\, \pi)$ at every time $t \geq 0$. \cref{thm:first-variation} does not apply directly at time $t$, as $T_{t + r} \circ T_t^{-1}$ is not of the form $I_d + r v$ for a fixed vector field $v$. Instead, we write $\KL(\rho_t \,\|\, \pi)$ as an integral against the fixed measure $\rho_0$, and we differentiate this integral.

In the rest of this section, $\rho_0$, $(T_t)_{t \geq 0}$ and $\rho_t = (T_t)_\# \rho_0$ are as in \cref{thm:kl-derivative}. For $t \geq 0$, let $v_t = b(\cdot, \rho_t) = (\phi^*_{\rho_t})_\Phi$ be the velocity field of the particles at time $t$, \ie the vector field of the Stein witness of $\rho_t$. Let $J_t(x) = \det D T_t(x)$ be the Jacobian determinant of $T_t$ at $x$. By \cref{lem:transport} with $\Psi = T_t$, we have that
\[
  \KL(\rho_t \,\|\, \pi) = \KL(\rho_0 \,\|\, \pi) - \int (\log p(T_t(x)) - \log p(x) + \log J_t(x)) \dif \rho_0(x).
\]
Only the integrand depends on $t$. We differentiate it along the trajectory of $x$. By the chain rule, the derivative of $\log p(T_t(x))$ is $s(T_t(x))(v_t(T_t(x)))$. By Liouville's formula for the Jacobian determinant of $T_t$, the derivative of $\log J_t(x)$ is $\nabla \cdot v_t(T_t(x))$. The sum of these two derivatives is the Stein operator of $v_t$, evaluated at $T_t(x)$. Then, we differentiate under the integral sign, and we change variables with $y = T_t(x)$. This gives
\[
  \frac{\dif}{\dif t} \KL(\rho_t \,\|\, \pi) = -\E_{y \sim \rho_t}[\cal{A}_s v_t(y)].
\]
The expectation is under $\rho_t$, and $v_t$ is the vector field of the Stein witness of the same measure $\rho_t$. Thus, (ii) of \cref{thm:stein-witness} shows that this expectation is $\K(\rho_t | \pi)$.

Compared with \cref{thm:first-variation}, the difficulty is the map $T_t$. The Jacobian determinant of $I_d + \eps v$ is explicit, whereas that of $T_t$ is not. \cref{sec:jacobian} shows that $T_t$ is differentiable, with $J_t > 0$, and it computes the derivative of $\log J_t$. \cref{sec:kl-derivative-conclusion} carries out the computation above. It checks the assumptions of \cref{lem:transport}, and it bounds the derivative of the integrand by a constant, which justifies the differentiation under the integral sign.

\subsubsection{The Jacobian of \texorpdfstring{$T_t$}{T\_t}}\label{sec:jacobian}

We first gather the facts on $v_t$ that we use. Let $m(t) = \phi^*_{\rho_t}$, so that $v_t = m(t)_\Phi$. By (i) of \cref{lem:solution-properties}, $m$ is continuous. It is bounded by $B$, as every Stein witness is (see \cref{sec:stein-witness}). Recall that $P_f : u \mapsto (\ket{f_i, u}_{H_0})_{1 \leq i \leq d}$ satisfies $\norm{P_f} \leq \norm{f}_{\cal{H}}$, and that $f_\Phi = P_f \circ \Phi$ (see \cref{sec:kernel-assumption}). By (i) to (iii) of \cref{ass:kernel}, $v_t$ is continuously differentiable, with $D v_t(y) = P_{m(t)} \circ D \Phi(y)$, and
\[
  \forall t \geq 0, \forall y \in \R^d, \quad \norm{v_t(y)} \leq B C \quad \text{and} \quad \norm{D v_t(y)} \leq B K.
\]
As the trace of a linear map $L$ of $\R^d$ satisfies $\abs{\operatorname{tr} L} \leq d \norm{L}$, we obtain $\abs{\nabla \cdot v_t(y)} \leq d B K$ for all $t \geq 0$ and $y \in \R^d$. Finally, by \cref{eq:flow}, for every $x \in \R^d$, the curve $t \mapsto T_t(x)$ solves
\begin{equation}\label{eq:fixed-velocity}
  \alpha'(t) = v_t(\alpha(t)), \quad \alpha(0) = x.
\end{equation}
Once the solution $(T_t)_{t \geq 0}$ is fixed, $v_t$ is a fixed velocity field, so that \cref{eq:fixed-velocity} is a standard differential equation.

The next lemma states that $T_t$ is differentiable with respect to the initial position. Its Jacobian determinant measures how much the particles spread out. It is positive, and its logarithm grows at most linearly in time.

\begin{lemma}[Jacobian of $T_t$]\label{lem:jacobian}
  Suppose that the standing assumptions hold, and let $(T_t)_{t \geq 0}$ and $v_t$ be as in \cref{sec:kl-derivative-strategy}. Then, for every $t \geq 0$, $T_t$ is continuously differentiable on $\R^d$. Let $J_t(x) = \det D T_t(x)$. Then, for all $x \in \R^d$ and $t \geq 0$, the following holds.
  \begin{enumerate}[label=(\roman*)]
    \item $J_t(x) > 0$.
    \item $\partial_t \log J_t(x) = \nabla \cdot v_t(T_t(x))$, with a right derivative at $t = 0$.
    \item $\abs{\log J_t(x)} \leq d B K \, t$.
  \end{enumerate}
\end{lemma}

\begin{proof}
  \inlinetitle{Differentiability}{.}~For all $t, t' \geq 0$ and $y, y' \in \R^d$, we have that
  \[
    \norm{D v_t(y) - D v_{t'}(y')} \leq \norm{m(t) - m(t')}_{\cal{H}} \norm{D \Phi(y)} + B \norm{D \Phi(y) - D \Phi(y')}.
  \]
  As $m$ and $D \Phi$ are continuous, $D v_t(y)$ is continuous in $(t, y)$. We extend $m$ to negative times by $m(0)$, so that $(t, y) \mapsto v_t(y)$ is defined on the open set $\R \times \R^d$, continuous, with $D v_t(y)$ continuous in $(t, y)$. As $v_t$ is $B K$-Lipschitz, \cref{eq:fixed-velocity} has a unique solution from each $x$, which is $t \mapsto T_t(x)$. Thus, the differentiable dependence on initial conditions \cite[Chapter V]{Hartman2002} shows that $T_t$ is continuously differentiable. Let $x \in \R^d$ and $Y_t = D T_t(x)$. Then, $Y_t$ solves the variational equation
  \[
    \partial_t Y_t = D v_t(T_t(x)) Y_t, \quad Y_0 = I_d.
  \]
  In words, a small displacement of the initial position is transported by the derivative of the velocity field along the trajectory.

  \inlinetitle{Liouville's formula}{.}~Let $A_t = D v_t(T_t(x))$, and let $\text{adj}(Y)$ be the adjugate matrix of a matrix $Y$, so that $Y \text{adj}(Y) = \det(Y) I_d$. By Jacobi's formula, we have that
  \[
    \partial_t \det Y_t = \operatorname{tr}(\text{adj}(Y_t) \partial_t Y_t) = \operatorname{tr}(\text{adj}(Y_t) A_t Y_t) = \operatorname{tr}(A_t Y_t \text{adj}(Y_t)) = \operatorname{tr}(A_t) \det Y_t.
  \]
  As $\operatorname{tr}(A_t) = \nabla \cdot v_t(T_t(x))$, the function $t \mapsto J_t(x)$ solves a scalar linear differential equation, with $J_0(x) = \det I_d = 1$. Its solution is given by Liouville's formula, \ie
  \[
    J_t(x) = \exp\left(\int_0^t \nabla \cdot v_r(T_r(x)) \dif r\right) > 0,
  \]
  which gives (i). The integrand is continuous in $r$, as $D v_r(y)$ is continuous in $(r, y)$ and $r \mapsto T_r(x)$ is continuous. Thus, $\log J_t(x)$ is differentiable in $t$, with derivative $\nabla \cdot v_t(T_t(x))$, which gives (ii).

  \inlinetitle{Bound}{.}~As $\abs{\nabla \cdot v_r(y)} \leq d B K$ for all $r \geq 0$ and $y \in \R^d$, Liouville's formula gives $\abs{\log J_t(x)} \leq d B K \, t$, which concludes the proof.
\end{proof}

\subsubsection{Differentiation of the KL divergence}\label{sec:kl-derivative-conclusion}

We can now carry out the computation of \cref{sec:kl-derivative-strategy}.

\begin{proof}
  We prove \cref{thm:kl-derivative}. Let $v_t$ be as in \cref{sec:kl-derivative-strategy}, and let $J_t$ be as in \cref{lem:jacobian}.

  \inlinetitle{Transport}{.}~Let $t \geq 0$. By (iii) of \cref{lem:solution-properties}, $T_t$ is a homeomorphism of $\R^d$ such that $T_t(\Omega) = \Omega$. By \cref{lem:jacobian}, it is differentiable, with $J_t > 0$, and $\log J_t$ is bounded. Thus, $\log J_t$ is $\rho_0$-integrable. By \cref{lem:transport}, $\KL(\rho_t \,\|\, \pi) < \infty$ and
  \[
    \KL(\rho_t \,\|\, \pi) = \KL(\rho_0 \,\|\, \pi) - \int \ell(t, x) \dif \rho_0(x),
  \]
  where $\ell(t, x) = \log p(T_t(x)) - \log p(x) + \log J_t(x)$.

  \inlinetitle{Differentiation}{.}~Let $x \in \Omega$, which holds for $\rho_0$-almost every $x$. We first differentiate $t \mapsto \log p(T_t(x))$. If $x \in \Omega^\circ$, then $T_t(x) \in \Omega^\circ$ for every $t \geq 0$. Indeed, a homeomorphism of $\R^d$ that maps $\Omega$ onto $\Omega$ maps $\Omega^\circ$ onto $\Omega^\circ$. As $\log p$ is differentiable at $T_t(x)$, with derivative $s(T_t(x))$, the chain rule gives
  \[
    \partial_t \log p(T_t(x)) = s(T_t(x))(v_t(T_t(x))).
  \]
  Otherwise, $x \in \partial \Omega$. Then, $x$ is at rest, \ie $T_t(x) = x$ by (iv) of \cref{lem:solution-properties}, and $v_t(x) = 0$ as $\Phi(x) = 0$. Thus, both sides vanish. With (ii) of \cref{lem:jacobian}, we obtain
  \[
    \partial_t \ell(t, x) = s(T_t(x))(v_t(T_t(x))) + \nabla \cdot v_t(T_t(x)) = \cal{A}_s v_t(T_t(x)).
  \]
  As in \cref{thm:first-variation}, we bound this derivative uniformly. Recall that $M_s = \max_{y \in \Omega} \norm{s(y)}$ (see \cref{sec:svgd-setting}). By \cref{sec:jacobian}, $\norm{v_t(y)} \leq B C$ and $\abs{\nabla \cdot v_t(y)} \leq d B K$ for all $t \geq 0$ and $y \in \R^d$. As $T_t(x) \in \Omega$, $\norm{s(T_t(x))} \leq M_s$. Thus,
  \[
    \abs{\partial_t \ell(t, x)} \leq M_s B C + d B K.
  \]
  This constant is $\rho_0$-integrable. Every $\ell(t, \cdot)$ is also continuous and bounded on $\Omega$, which carries $\rho_0$. Indeed, $\log p$ is continuous and bounded on $\Omega$, $T_t$ is continuous with $T_t(\Omega) = \Omega$, and $\log J_t$ is continuous, as $D T_t$ is continuous, and bounded by (iii) of \cref{lem:jacobian}. By the Leibniz integral rule \cite[Theorem 2.27]{Folland1999}, $t \mapsto \KL(\rho_t \,\|\, \pi)$ is differentiable at every $t \geq 0$, with
  \begin{align*}
    \frac{\dif}{\dif t} \KL(\rho_t \,\|\, \pi) & = -\int \partial_t \ell(t, x) \dif \rho_0(x) \\
                                              & = -\E_{x \sim \rho_0}[\cal{A}_s v_t(T_t(x))] \\
                                              & = -\E_{y \sim \rho_t}[\cal{A}_s v_t(y)].
  \end{align*}
  The last equality is the change of variables $y = T_t(x)$.

  \inlinetitle{The Stein witness}{.}~As $v_t = (\phi^*_{\rho_t})_\Phi$ and $\rho_t \in \cal{P}(\Omega)$, (ii) of \cref{thm:stein-witness} gives
  \[
    \E_{y \sim \rho_t}[\cal{A}_s v_t(y)] = \K(\rho_t | \pi).
  \]
  Thus, $\frac{\dif}{\dif t} \KL(\rho_t \,\|\, \pi) = -\K(\rho_t | \pi)$.

  \inlinetitle{Monotonicity}{.}~As $\K(\rho_t | \pi) \geq 0$, the derivative is nonpositive on $[0, \infty)$. By the mean value theorem, $t \mapsto \KL(\rho_t \,\|\, \pi)$ is nonincreasing on $[0, \infty)$, which concludes the proof.
\end{proof}

\subsection{Convergence}\label{sec:convergence}

In this section, we show that the distribution $\rho_t$ of the particles converges to $\pi$ as $t$ tends to infinity, when the kernel is rich enough. We first define the notion of convergence, and a condition on the kernel, called Stein-universality. Under this condition, the KSD vanishes only at $\pi$. Then, we state the main result of this section, and we prove it in three steps.

\subsubsection{Weak convergence and Stein-universal kernels}\label{sec:stein-universal}

Let $C^0(\Omega)$ be the set of continuous functions from $\Omega$ to $\R$. Let $\rho \in \cal{P}(\Omega)$ and $g \in C^0(\Omega)$. As $\Omega$ is compact, $g$ is bounded. As $\rho$ is carried by $\Omega$, the expectation $\E_{x \sim \rho}[g(x)]$ is thus well defined (see \cref{sec:svgd-setting}). We compare measures through these expectations. A sequence $(\rho_n)_{n \in \N}$ of $\cal{P}(\Omega)$ converges weakly to $\rho \in \cal{P}(\Omega)$ if
\[
  \forall g \in C^0(\Omega), \quad \E_{x \sim \rho_n}[g(x)] \lima{n \to \infty} \E_{x \sim \rho}[g(x)].
\]
We then write $\rho_n \xrightharpoonup[n \to \infty]{} \rho$, and likewise for a family $(\rho_t)_{t \geq 0}$ as $t$ tends to infinity.

This is the usual weak convergence of probability measures on the metric space $\Omega$ \cite{Billingsley1999}. Such a measure is only defined on the Borel subsets of $\Omega$, which are the sets $A \in \cal{B}(\R^d)$ such that $A \subseteq \Omega$, as $\Omega$ is closed. A measure of $\cal{P}(\Omega)$ is defined on all of $\cal{B}(\R^d)$. We go from one to the other by restriction. Let $\rho \in \cal{P}(\Omega)$. Its restriction $\rho|_{\Omega}$ to $\Omega$ is defined by
\[
  \forall A \in \cal{B}(\R^d), \quad A \subseteq \Omega \implies \rho|_{\Omega}(A) = \rho(A).
\]
As $\rho(\R^d \setminus \Omega) = 0$, we have that $\rho|_{\Omega}(\Omega) = \rho(\R^d) = 1$, \ie $\rho|_{\Omega}$ is a probability measure on $\Omega$. No information is lost, as
\[
  \forall A \in \cal{B}(\R^d), \quad \rho(A) = \rho|_{\Omega}(A \cap \Omega).
\]
Conversely, every probability measure $\rho'$ on $\Omega$ is the restriction of the measure $A \mapsto \rho'(A \cap \Omega)$ of $\cal{P}(\Omega)$. Moreover, $\E_{x \sim \rho}[g(x)] = \int_\Omega g \dif \rho|_{\Omega}$ for every $g \in C^0(\Omega)$. The functions of $C^0(\Omega)$ are the bounded continuous functions on $\Omega$. Thus, $\rho_n \rightharpoonup \rho$ if and only if $\rho_n|_{\Omega}$ converges weakly to $\rho|_{\Omega}$, as probability measures on the metric space $\Omega$. This allows us to apply two classical results on probability measures on a metric space, with the metric space $\Omega$, in \cref{prop:ksd-valid} and in \cref{sec:identification}.

We now turn to the kernel. Let $\rho \in \cal{P}(\Omega)$. By \cref{thm:stein-witness}, we have that $\K(\rho | \pi) = \norm{\phi^*_\rho}_{\cal{H}}^2$ and $\E_{x \sim \rho}[\cal{A}_s f_\Phi(x)] = \ket{f, \phi^*_\rho}_{\cal{H}}$ for every $f \in \cal{H}$. Thus, we have that
\begin{equation}\label{eq:ksd-zero}
  \K(\rho | \pi) = 0 \iff \phi^*_\rho = 0 \iff \forall f \in \cal{H}, \quad \E_{x \sim \rho}[\cal{A}_s f_\Phi(x)] = 0,
\end{equation}
where the converse implication of the second equivalence follows by taking $f = \phi^*_\rho$. In words, the KSD of $\rho$ vanishes if and only if $\rho$ satisfies Stein's identity for every vector field $f_\Phi$, as $\pi$ does (see \cref{cor:stein-identity}). To deduce that $\rho = \pi$, the functions $\cal{A}_s f_\Phi$ must be rich enough to determine a measure. This is not always the case. For instance, let $y \in \partial \Omega$ such that $D \Phi(y) = 0$. As $\Phi(y) = 0$ by (iv) of \cref{ass:kernel}, we have that $\xi_y = 0$. Then, $\phi^*_{\delta_y} = \xi_y = 0$, \ie $\K(\delta_y | \pi) = 0$, whereas $\delta_y \neq \pi$, as $\pi$ has a density. \cref{fig:boundary-degeneracy} illustrates this situation in dimension $1$.

\begin{figure}[ht]
  \centering
  \includegraphics[width=0.85\textwidth]{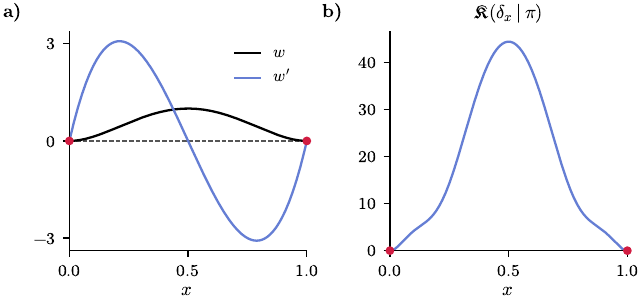}
  \caption{A kernel whose KSD vanishes at Dirac masses that are not $\pi$. Here, $\Omega = [0, 1]$, the target $\pi$ and the Gaussian kernel $k_0$ are those of \cref{fig:stein-witness}, and the kernel is $k(x, y) = w(x) w(y) k_0(x, y)$ with $w(x) = 16 x^2 (1 - x)^2$. Its feature map is $\Phi = w \Phi_0$, where $\Phi_0$ is a feature map of $k_0$. \textbf{a)}~The weight $w$, in black, and its derivative $w'$, in blue, vanish on $\partial \Omega = \{0, 1\}$ (red dots). Thus, so do $\Phi = w \Phi_0$ and $D \Phi = w' \Phi_0 + w D \Phi_0$. \textbf{b)}~The KSD of the Dirac mass at $x$. By \cref{thm:stein-witness}, $\K(\delta_x | \pi) = \norm{\xi_x}_{\cal{H}}^2 = (s(x) w(x) + w'(x))^2 + w(x)^2 / 0.15^2$. It is positive on $(0, 1)$, but it vanishes at $0$ and $1$ (red dots), whereas $\delta_0 \neq \pi$ and $\delta_1 \neq \pi$.}
  \label{fig:boundary-degeneracy}
\end{figure}

Stein-universality excludes this situation (see \cref{prop:ksd-valid}). Let $f \in \cal{H}$. As $\cal{A}_s f_\Phi$ is continuous on $\Omega$ (see \cref{sec:ksd-definition}), its restriction to $\Omega$ lies in $C^0(\Omega)$. Stein-universality requires that these functions, up to an additive constant, approximate every function of $C^0(\Omega)$, uniformly on $\Omega$.

\begin{definition}[Stein-universal kernel]\label{def:stein-universal}
  Suppose that the standing assumptions hold. The kernel $k$ is \emph{Stein-universal} if the set of functions
  \[
    \{\cal{A}_s f_\Phi + c \mid f \in \cal{H}, c \in \R\},
  \]
  restricted to $\Omega$, is dense in $C^0(\Omega)$ for the uniform convergence on $\Omega$, \ie
  \[
    \forall g \in C^0(\Omega), \forall \eta > 0, \exists f \in \cal{H}, \exists c \in \R, \quad \sup_{x \in \Omega} \abs{g(x) - \cal{A}_s f_\Phi(x) - c} \leq \eta.
  \]
\end{definition}

This is the analog, for the Stein operator, of the universality of a kernel, which requires that its RKHS be dense in $C^0(\Omega)$ \cite{Steinwart2001}. This condition depends on $\Phi$ through the vector fields $f_\Phi$, on $\pi$ through its score $s$, and on $\Omega$.

The constants $c$ are needed. Indeed, let $f \in \cal{H}$. The vector field $f_\Phi$ is continuously differentiable and Lipschitz, and it vanishes on $\partial \Omega$ (see \cref{sec:kernel-assumption}). By \cref{cor:stein-identity}, we have that $\E_{x \sim \pi}[\cal{A}_s f_\Phi(x)] = 0$. Thus, the functions $\cal{A}_s f_\Phi$ alone only approximate functions with zero mean under $\pi$. In fact, the kernel is Stein-universal if and only if
\begin{equation}\label{eq:stein-universal-mean}
  \forall g \in C^0(\Omega), \forall \eta > 0, \exists f \in \cal{H}, \quad \sup_{x \in \Omega} \abs{g(x) - \E_{y \sim \pi}[g(y)] - \cal{A}_s f_\Phi(x)} \leq \eta.
\end{equation}
In words, every function of $C^0(\Omega)$ with zero mean under $\pi$ is a uniform limit on $\Omega$ of functions $\cal{A}_s f_\Phi$. Indeed, if \cref{eq:stein-universal-mean} holds, then the inequality of \cref{def:stein-universal} holds with $c = \E_{y \sim \pi}[g(y)]$. Conversely, suppose that the kernel is Stein-universal. Let $g \in C^0(\Omega)$ and $\eta > 0$. By \cref{def:stein-universal} applied with $\eta / 2$, there exist $f \in \cal{H}$ and $c \in \R$ such that
\[
  \forall x \in \Omega, \quad \abs{g(x) - \cal{A}_s f_\Phi(x) - c} \leq \frac{\eta}{2}.
\]
As $\pi$ is a probability measure carried by $\Omega$ and $\E_{y \sim \pi}[\cal{A}_s f_\Phi(y)] = 0$, integrating this inequality against $\pi$ gives
\[
  \abs{\E_{y \sim \pi}[g(y)] - c} = \abs{\E_{y \sim \pi}[g(y) - \cal{A}_s f_\Phi(y) - c]} \leq \frac{\eta}{2}.
\]
By the triangle inequality, we obtain that, for every $x \in \Omega$,
\[
  \abs{g(x) - \E_{y \sim \pi}[g(y)] - \cal{A}_s f_\Phi(x)} \leq \abs{g(x) - \cal{A}_s f_\Phi(x) - c} + \abs{c - \E_{y \sim \pi}[g(y)]} \leq \eta,
\]
which is \cref{eq:stein-universal-mean}.

\begin{proposition}[The KSD vanishes only at $\pi$]\label{prop:ksd-valid}
  Suppose that the standing assumptions hold, and that the kernel is Stein-universal. Then,
  \[
    \forall \rho \in \cal{P}(\Omega), \quad \K(\rho | \pi) = 0 \iff \rho = \pi.
  \]
\end{proposition}

\begin{proof}
  If $\rho = \pi$, then $\K(\rho | \pi) = 0$ (see \cref{sec:ksd-definition}). Conversely, let $\rho \in \cal{P}(\Omega)$ such that $\K(\rho | \pi) = 0$. By \cref{eq:ksd-zero}, $\E_{x \sim \rho}[\cal{A}_s f_\Phi(x)] = 0$ for every $f \in \cal{H}$. Let $g \in C^0(\Omega)$ and $\eta > 0$, and let $f \in \cal{H}$ be given by \cref{eq:stein-universal-mean}. As $\rho$ is a probability measure carried by $\Omega$, we have that
  \begin{align*}
    \abs{\E_{x \sim \rho}[g(x)] - \E_{y \sim \pi}[g(y)]} & = \abs{\E_{x \sim \rho}[g(x) - \E_{y \sim \pi}[g(y)] - \cal{A}_s f_\Phi(x)]} \\
                                                         & \leq \sup_{x \in \Omega} \abs{g(x) - \E_{y \sim \pi}[g(y)] - \cal{A}_s f_\Phi(x)} \\
                                                         & \leq \eta.
  \end{align*}
  As $\eta$ is arbitrary, $\E_{x \sim \rho}[g(x)] = \E_{y \sim \pi}[g(y)]$ for every $g \in C^0(\Omega)$. Thus, $\rho|_{\Omega}$ and $\pi|_{\Omega}$ are two probability measures on the metric space $\Omega$ that give the same integral to every bounded continuous function. Such measures are equal \cite[Theorem 1.2]{Billingsley1999}. As $\rho$ and $\pi$ are determined by their restrictions to $\Omega$, we have that $\rho = \pi$, which concludes the proof.
\end{proof}

\subsubsection{Convergence to the target distribution}\label{sec:convergence-main}

We can now state the main result of this section. Like \cref{thm:kl-derivative}, it holds for every solution of \cref{eq:flow}.

\begin{theorem}[Convergence to the target distribution]\label{thm:weak-convergence}
  Suppose that the standing assumptions hold, and that the kernel is Stein-universal. Let $\rho_0$ be a probability measure on $\R^d$ such that $\KL(\rho_0 \,\|\, \pi) < \infty$. Let $(T_t)_{t \geq 0}$ be a solution of \cref{eq:flow}, such as the solutions given by \cref{thm:continuous-time}, and let $\rho_t = (T_t)_\# \rho_0$. Then,
  \[
    \rho_t \xrightharpoonup[t \to \infty]{} \pi.
  \]
\end{theorem}

For the particles, \cref{thm:weak-convergence} reads as follows. Let $g \in C^0(\Omega)$. By \cref{lem:solution-properties} (iii), $T_t(x) \in \Omega$ for every $x \in \Omega$, so that $g(T_t(x))$ is defined for $\rho_0$-almost every $x$. The average of $g$ over the positions of the particles at time $t$ is
\[
  \E_{x \sim \rho_0}[g(T_t(x))] = \E_{y \sim \rho_t}[g(y)] \lima{t \to \infty} \E_{y \sim \pi}[g(y)].
\]
Thus, in the long run, the distribution of the particles approaches $\pi$.

\cref{thm:weak-convergence} has four limitations. First, it gives no rate of convergence. Second, it concerns the continuous-time dynamics \cref{eq:flow}. The convergence of the iteration \cref{eq:svgd} is not studied (see \cref{sec:flow}). Third, the assumption $\KL(\rho_0 \,\|\, \pi) < \infty$ implies that $\rho_0 \ll \pi$, \ie that $\rho_0$ has a density with respect to $\pi$. It thus excludes the empirical measure of finitely many particles, used in practice (see \cref{sec:kl-descent}). Finally, it assumes that the kernel is Stein-universal. We do not address here the existence or the construction of Stein-universal kernels.

\subsubsection{Strategy of the proof}\label{sec:convergence-strategy}

By \cref{thm:kl-derivative}, the Kullback--Leibler divergence decreases along the dynamics, at rate $\K(\rho_t | \pi)$. As it is nonnegative, it has a limit as $t$ tends to infinity. Informally, its derivative $-\K(\rho_t | \pi)$ should then tend to $0$, and $\rho_t$ should approach a measure whose KSD vanishes, which is $\pi$ by \cref{prop:ksd-valid}. The proof makes this argument rigorous, in three steps.
\begin{enumerate}[label=\arabic*.]
  \item We show that $\K(\rho_t | \pi)$ tends to $0$ (\cref{sec:ksd-to-zero}). This does not follow from the convergence of the divergence alone, as a function can converge while its derivative keeps oscillating. We show that $t \mapsto \K(\rho_t | \pi)$ is uniformly continuous, which excludes these oscillations.
  \item We show that the limit points of $\rho_t$ have a zero KSD (\cref{sec:limit-points}). A limit point is a measure $\rho_\infty \in \cal{P}(\Omega)$ such that $\rho_{t_n} \rightharpoonup \rho_\infty$ for some sequence $t_n \to \infty$. The weak convergence controls the expectation of $\cal{A}_s f_\Phi$ for each fixed $f$. As the KSD is a supremum of the squares of such expectations, the KSD of the limit is at most the limit of the KSD, which is $0$.
  \item We show that $\rho_t$ converges weakly to $\pi$ (\cref{sec:identification}). As $\Omega$ is compact, every sequence $(\rho_{t_n})_{n \in \N}$ with $t_n \to \infty$ has a weakly convergent subsequence. By the second step and \cref{prop:ksd-valid}, its limit is $\pi$. Thus, $\rho_t$ converges to $\pi$.
\end{enumerate}
Only the last step uses the Stein-universality of the kernel.

\subsubsection{The KSD tends to zero}\label{sec:ksd-to-zero}

Let $\rho_0$, $(T_t)_{t \geq 0}$ and $\rho_t$ be as in \cref{cor:ksd-convergence} below. Unlike \cref{thm:weak-convergence}, this corollary does not assume that the kernel is Stein-universal. By \cref{thm:kl-derivative}, the derivative of $t \mapsto \KL(\rho_t \,\|\, \pi)$ is $-\K(\rho_t | \pi)$. Integrating it, which is justified once we know that the KSD is continuous in time (\cref{lem:ksd-uniform-continuity}), gives
\[
  \forall a \geq 0, \quad \int_0^a \K(\rho_t | \pi) \dif t = \KL(\rho_0 \,\|\, \pi) - \KL(\rho_a \,\|\, \pi) \leq \KL(\rho_0 \,\|\, \pi).
\]
Thus, the nonnegative function $t \mapsto \K(\rho_t | \pi)$ has a finite integral on $[0, \infty)$. This is not enough to make it tend to $0$. For instance, the function
\[
  \theta(t) = \sum_{n \geq 1} \max(0, 1 - n^2 \abs{t - n})
\]
is a sum of spikes centered at the integers. The spike centered at $n$ has height $1$ and area $1 / n^2$. Thus, $\theta$ is nonnegative, and its integral on $[0, \infty)$ is $\sum_{n \geq 1} 1 / n^2 < \infty$, whereas $\theta(n) = 1$ for every $n \geq 1$. Let $\ell \in \R$ and $F(t) = \ell + \int_t^\infty \theta(\tau) \dif \tau$. Then, $F$ decreases and converges to $\ell$, whereas its derivative $F' = -\theta$ does not tend to $0$ (see \cref{fig:spikes}). The spikes of $\theta$ get thinner and thinner, so that $\theta$ is continuous but not uniformly continuous. Barbalat's lemma \cite[Lemma 8.2]{Khalil2002} states that uniform continuity excludes this behavior. We give its short proof.

\begin{figure}[ht]
  \centering
  \includegraphics[width=\textwidth]{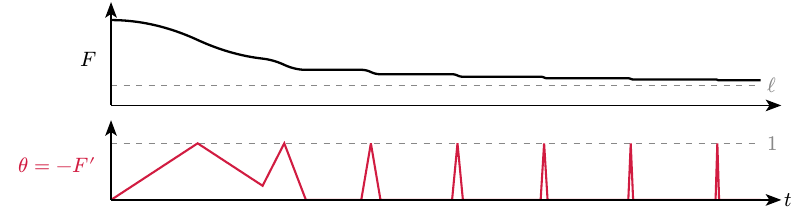}
  \caption{A function that converges while its derivative does not tend to $0$. \textbf{Top.}~The function $F(t) = \ell + \int_t^\infty \theta(\tau) \dif \tau$ decreases to its limit $\ell$ (dashed). \textbf{Bottom.}~Its derivative $-F' = \theta$, with $\theta(t) = \sum_{n \geq 1} \max(0, 1 - n^2 \abs{t - n})$, keeps reaching the height $1$ (dashed). Its spikes have area $1 / n^2$ and get thinner and thinner, so that $\theta$ is not uniformly continuous. Here, $F$ plays the role of $t \mapsto \KL(\rho_t \,\|\, \pi)$, and $\theta$ the role of $t \mapsto \K(\rho_t | \pi)$. This illustrates why \cref{lem:barbalat} requires uniform continuity.}
  \label{fig:spikes}
\end{figure}

\begin{lemma}[Barbalat]\label{lem:barbalat}
  Let $u : [0, \infty) \to [0, \infty)$ be uniformly continuous, and let $M \geq 0$ such that $\int_0^a u(t) \dif t \leq M$ for every $a \geq 0$. Then,
  \[
    u(t) \lima{t \to \infty} 0.
  \]
\end{lemma}

\begin{proof}
  Suppose that $u(t)$ does not tend to $0$. Then, there exists $\eta > 0$ such that $u(t) \geq \eta$ for arbitrarily large $t$. Thus, there exists a sequence $(t_n)_{n \in \N}$ such that
  \[
    \forall n \in \N, \quad u(t_n) \geq \eta \quad \text{and} \quad t_{n + 1} \geq t_n + 1.
  \]
  As $u$ is uniformly continuous, there exists $\delta \in (0, 1)$ such that
  \[
    \forall t, t' \geq 0, \quad \abs{t - t'} \leq \delta \implies \abs{u(t) - u(t')} \leq \frac{\eta}{2}.
  \]
  Thus, $u \geq \eta / 2$ on every interval $[t_n, t_n + \delta]$. As $\delta < 1$, these intervals are disjoint. As $u \geq 0$, we have that
  \[
    \forall N \in \N, \quad \int_0^{t_N + \delta} u(t) \dif t \geq \sum_{n = 0}^N \int_{t_n}^{t_n + \delta} u(t) \dif t \geq (N + 1) \frac{\delta \eta}{2}.
  \]
  The right-hand side tends to infinity with $N$, which contradicts the bound $M$. Thus, $u(t)$ tends to $0$.
\end{proof}

It remains to show that $t \mapsto \K(\rho_t | \pi)$ is uniformly continuous. This follows from the continuity in time of the Stein witness (\cref{lem:solution-properties} (i)), and from the bound $B$ on the Stein witnesses. Recall that $\omega$ is the modulus of continuity of the Stein feature on $\Omega$ (see \cref{sec:schauder}).

\begin{lemma}[Uniform continuity of the KSD]\label{lem:ksd-uniform-continuity}
  Suppose that the standing assumptions hold. Let $\rho_0 \in \cal{P}(\Omega)$, let $(T_t)_{t \geq 0}$ be a solution of \cref{eq:flow}, and let $\rho_t = (T_t)_\# \rho_0$. Then,
  \[
    \forall t, t' \geq 0, \quad \abs{\K(\rho_t | \pi) - \K(\rho_{t'} | \pi)} \leq 2 B \, \omega(B C \abs{t - t'}).
  \]
  In particular, $t \mapsto \K(\rho_t | \pi)$ is uniformly continuous on $[0, \infty)$.
\end{lemma}

\begin{proof}
  Let $t, t' \geq 0$. As $\rho_t, \rho_{t'} \in \cal{P}(\Omega)$, (ii) of \cref{thm:stein-witness} gives $\K(\rho_t | \pi) = \norm{\phi^*_{\rho_t}}_{\cal{H}}^2$, and likewise at $t'$. Every Stein witness has norm at most $B$ (see \cref{sec:stein-witness}). By the reverse triangle inequality and \cref{lem:solution-properties} (i), we have that
  \begin{align*}
    \abs{\K(\rho_t | \pi) - \K(\rho_{t'} | \pi)} & = (\norm{\phi^*_{\rho_t}}_{\cal{H}} + \norm{\phi^*_{\rho_{t'}}}_{\cal{H}}) \abs{\norm{\phi^*_{\rho_t}}_{\cal{H}} - \norm{\phi^*_{\rho_{t'}}}_{\cal{H}}} \\
                                                 & \leq 2 B \norm{\phi^*_{\rho_t} - \phi^*_{\rho_{t'}}}_{\cal{H}} \\
                                                 & \leq 2 B \, \omega(B C \abs{t - t'}).
  \end{align*}
  As $\omega(r)$ tends to $0$ as $r$ tends to $0$, $t \mapsto \K(\rho_t | \pi)$ is uniformly continuous, which concludes the proof.
\end{proof}

We can now conclude the first step.

\begin{corollary}[Convergence of the KSD]\label{cor:ksd-convergence}
  Suppose that the standing assumptions hold. Let $\rho_0$ be a probability measure on $\R^d$ such that $\KL(\rho_0 \,\|\, \pi) < \infty$. Let $(T_t)_{t \geq 0}$ be a solution of \cref{eq:flow}, and let $\rho_t = (T_t)_\# \rho_0$. Then,
  \[
    \K(\rho_t | \pi) \lima{t \to \infty} 0.
  \]
\end{corollary}

\begin{proof}
  As $\KL(\rho_0 \,\|\, \pi) < \infty$, we have that $\rho_0 \in \cal{P}(\Omega)$ (see \cref{sec:first-variation}). For $t \geq 0$, let $F(t) = \KL(\rho_t \,\|\, \pi)$ and $u(t) = \K(\rho_t | \pi)$. By \cref{thm:kl-derivative}, $F$ is real-valued and differentiable on $[0, \infty)$, with $F' = -u$. By \cref{lem:ksd-uniform-continuity}, $u$ is continuous. Thus, the fundamental theorem of calculus gives
  \[
    \forall a \geq 0, \quad \int_0^a u(t) \dif t = F(0) - F(a) \leq F(0),
  \]
  as $F(a) \geq 0$. As $u \geq 0$ and $u$ is uniformly continuous, \cref{lem:barbalat} applies with $M = F(0)$, which concludes the proof.
\end{proof}

As $\K(\rho_t | \pi) = \norm{\phi^*_{\rho_t}}_{\cal{H}}^2$, the Stein witness $\phi^*_{\rho_t}$ tends to $0$ in $\cal{H}$. Thus, the velocity of the particles tends to $0$, uniformly in space, as $\norm{b(y, \rho_t)} \leq C \norm{\phi^*_{\rho_t}}_{\cal{H}}$ for every $y \in \R^d$. The particles slow down. This does not yet show that $\rho_t$ converges, nor that its limit is $\pi$.

\subsubsection{Limit points of the dynamics}\label{sec:limit-points}

Let $(t_n)_{n \in \N}$ be a sequence such that $t_n \to \infty$, and let $\rho_\infty \in \cal{P}(\Omega)$ such that $\rho_{t_n} \rightharpoonup \rho_\infty$. By \cref{cor:ksd-convergence}, $\K(\rho_{t_n} | \pi) \to 0$. We want to deduce that $\K(\rho_\infty | \pi) = 0$. By \cref{def:ksd}, the KSD of $\rho$ is the supremum of $\E_{x \sim \rho}[\cal{A}_s f_\Phi(x)]^2$ over the closed unit ball of $\cal{H}$. For each such $f$, this quantity at $\rho_{t_n}$ converges to its value at $\rho_\infty$ by weak convergence, and it is at most $\K(\rho_{t_n} | \pi)$.

\begin{lemma}[Limit points have a zero KSD]\label{lem:limit-points}
  Suppose that the standing assumptions hold. Let $\rho_0$ be a probability measure on $\R^d$ such that $\KL(\rho_0 \,\|\, \pi) < \infty$. Let $(T_t)_{t \geq 0}$ be a solution of \cref{eq:flow}, and let $\rho_t = (T_t)_\# \rho_0$. Let $(t_n)_{n \in \N}$ be a sequence such that $t_n \to \infty$, and let $\rho_\infty \in \cal{P}(\Omega)$ such that $\rho_{t_n} \xrightharpoonup[n \to \infty]{} \rho_\infty$. Then, $\K(\rho_\infty | \pi) = 0$.
\end{lemma}

\begin{proof}
  Let $f \in \cal{H}$ such that $\norm{f}_{\cal{H}} \leq 1$. The function $\cal{A}_s f_\Phi$ is continuous on $\Omega$ (see \cref{sec:ksd-definition}). By the weak convergence of $\rho_{t_n}$, and by \cref{def:ksd}, we have that
  \[
    \E_{x \sim \rho_\infty}[\cal{A}_s f_\Phi(x)]^2 = \lim_{n \to \infty} \E_{x \sim \rho_{t_n}}[\cal{A}_s f_\Phi(x)]^2 \leq \lim_{n \to \infty} \K(\rho_{t_n} | \pi) = 0,
  \]
  where the last limit is given by \cref{cor:ksd-convergence}. Thus, $\E_{x \sim \rho_\infty}[\cal{A}_s f_\Phi(x)] = 0$ for every $f$ of the closed unit ball of $\cal{H}$. By \cref{def:ksd}, $\K(\rho_\infty | \pi) = 0$, which concludes the proof.
\end{proof}

\subsubsection{Identification of the limit}\label{sec:identification}

It remains to show that $\rho_t$ has limit points, and that they are all equal to $\pi$. The first point follows from the compactness of $\Omega$. On a compact metric space, every sequence of probability measures has a weakly convergent subsequence. This follows from Prokhorov's theorem \cite[Theorem 5.1]{Billingsley1999}. Its tightness condition requires that, for every $\eta > 0$, a single compact set have mass at least $1 - \eta$ under every measure. On a compact space, it holds with the whole space. By the identification of \cref{sec:stein-universal}, every sequence of $\cal{P}(\Omega)$ thus has a subsequence that converges weakly to a measure of $\cal{P}(\Omega)$.

\begin{proof}
  We prove \cref{thm:weak-convergence}. Suppose that $\rho_t$ does not converge weakly to $\pi$. Then, there exist $g \in C^0(\Omega)$ and $\eta > 0$ such that $\abs{\E_{x \sim \rho_t}[g(x)] - \E_{x \sim \pi}[g(x)]} \geq \eta$ for arbitrarily large $t$. Thus, there exists a sequence $(t_n)_{n \in \N}$ such that $t_n \to \infty$ and
  \begin{equation}\label{eq:far-from-target}
    \forall n \in \N, \quad \abs{\E_{x \sim \rho_{t_n}}[g(x)] - \E_{x \sim \pi}[g(x)]} \geq \eta.
  \end{equation}
  As $\rho_{t_n} \in \cal{P}(\Omega)$ for every $n \in \N$, Prokhorov's theorem gives a subsequence $(\tau_k)_{k \in \N}$ of $(t_n)_{n \in \N}$ and a measure $\rho_\infty \in \cal{P}(\Omega)$ such that $\rho_{\tau_k} \xrightharpoonup[k \to \infty]{} \rho_\infty$. As $\tau_k \to \infty$, \cref{lem:limit-points} gives $\K(\rho_\infty | \pi) = 0$. By \cref{prop:ksd-valid}, $\rho_\infty = \pi$. Then, $\E_{x \sim \rho_{\tau_k}}[g(x)] \to \E_{x \sim \pi}[g(x)]$, which contradicts \cref{eq:far-from-target}. Hence, $\rho_t \xrightharpoonup[t \to \infty]{} \pi$, which concludes the proof.
\end{proof}

\subsection{Application to Stein Boltzmann Sampling}\label{sec:svgd-sbs}
We now apply the results of this section to \sbs (\cref{sec:sbs-method}), whose function to minimize is $V = f$. Let $V : \R^d \to \R$ be continuously differentiable, let $\kappa > 0$, and assume that $\lambda(\Omega) > 0$. The target distribution $\pi$ is the Boltzmann distribution of $V$ on $\Omega$ (\cref{def:boltzmann-distribution}). By \cref{ex:boltzmann}, it satisfies \cref{ass:target}, with $s = -\kappa D V$.

The Gaussian kernel $k_\sigma$, used by \sbs, does not satisfy (iv) of \cref{ass:kernel} (see \cref{fig:stein-witness}). The results of this section thus apply to \sbs with another kernel. For instance, one can weight $k_\sigma$ by a function $w : \R^d \to \R$ that vanishes on $\partial \Omega$, \ie
\[
  k(x, y) = w(x) w(y) k_\sigma(x, y).
\]
Let $H_\sigma$ be the reproducing kernel Hilbert space of $k_\sigma$, and let $\Phi_\sigma(x) = k_\sigma(\cdot, x)$ be its feature map. Then, $\Phi = w \Phi_\sigma$ is a feature map of $k$, with values in $H_\sigma$.

\begin{remark}
  Let $\sigma > 0$, and let $w : \R^d \to \R$ be continuously differentiable, bounded by $C_w$ and $L_w$-Lipschitz, with $w(y) = 0$ for every $y \in \partial \Omega$. Then, $\Phi = w \Phi_\sigma$ satisfies \cref{ass:kernel}, with $K = L_w + C_w \sqrt{2 / \sigma}$ and $C = C_w$. Indeed, as $k_\sigma$ is infinitely differentiable, $\Phi_\sigma$ is continuously differentiable (see \cite{Zhou2008}), and $\norm{\Phi_\sigma(x)}_{H_\sigma}^2 = k_\sigma(x, x) = 1$ for every $x \in \R^d$. Thus, $\Phi$ is continuously differentiable, bounded by $C_w$, and zero on $\partial \Omega$. Moreover, as $1 - e^{-u} \leq u$ for every $u \geq 0$, we have that
  \[
    \norm{\Phi_\sigma(x) - \Phi_\sigma(y)}_{H_\sigma}^2 = 2 \pare*{1 - e^{-\norm{x - y}^2 / \sigma}} \leq \frac{2}{\sigma} \norm{x - y}^2
  \]
  for all $x, y \in \R^d$. Thus,
  \begin{align*}
    \norm{\Phi(x) - \Phi(y)}_{H_\sigma} & \leq \abs{w(x) - w(y)} \norm{\Phi_\sigma(x)}_{H_\sigma} + \abs{w(y)} \norm{\Phi_\sigma(x) - \Phi_\sigma(y)}_{H_\sigma} \\
    & \leq \pare*{L_w + C_w \sqrt{2 / \sigma}} \norm{x - y}.
  \end{align*}
\end{remark}

The results of this section apply to \sbs with any kernel that satisfies \cref{ass:kernel}, as stated in \cref{cor:sbs}. We now prove this corollary.

\begin{proof}[Proof of \cref{cor:sbs}]
  Using \cref{ex:boltzmann}, the standing assumptions hold. Thus, \cref{thm:continuous-time,lem:solution-properties} give (i). Then, \cref{thm:kl-derivative} gives (ii), and \cref{thm:weak-convergence} gives (iii). Finally, let $\rho_0 = \lambda|_\Omega / \lambda(\Omega)$. The density of $\rho_0$ with respect to $\pi$ is $1 / (\lambda(\Omega) p)$ on $\Omega$. As $p$ reaches a positive minimum on $\Omega$, this density is bounded, and so is its logarithm. Thus, $\KL(\rho_0 \,\|\, \pi) < \infty$, which concludes the proof.
\end{proof}

\end{document}